\documentclass[a4paper,reqno]{amsart}
\usepackage{eurosym}
\usepackage{amsmath, amssymb}
\usepackage{amsfonts}
\usepackage{graphicx}
\usepackage{color}
\usepackage[usenames,dvipsnames]{xcolor}
\usepackage[english]{babel}
\usepackage{bm}

\newtheorem{teo}{Theorem}[section]

\newtheorem{Lemma}{Lemma}[section]
\newtheorem{theorem}[Lemma]{Theorem}
\newtheorem{proposition}[Lemma]{Proposition}
\newtheorem{lemma}[Lemma]{Lemma}

\newtheorem{remark}[Lemma]{Remark}
\newtheorem{definition}[Lemma]{Definition}

\begin{document}
\title[NSC]{Optimal control of Newtonian fluids in a stochastic environment}
\author{Nikolai Chemetov}
\address{Department of Computing and Mathematics-FFCLRP, University of S{\~a}o
Paulo, 14040-901 Ribeir{\~a}o Preto - SP, Brazil}
\email{nvchemetov@gmail.com }
\author{Fernanda Cipriano}
\address{Centro de Matem{\'a}tica e Aplica\c c\~oes (CMA), FCT-UNL and Dep.
de Matem{\'a}tica, Faculdade de Ci\^encias e Tecnologia, UNL \\
Quinta da Torre \\
2829 -516 Caparica, Lisboa \\
Portugal}
\email{cipriano@fct.unl.pt}
\date{}

\begin{abstract}
We consider a velocity tracking problem { for stochastic } Navier-Stokes
equations in a 2D-bounded domain. The control acts on the boundary through
an injection-suction device with { uncertainty,} which acts in accordance with the non-homogeneous Navier-slip boundary conditions. After establishing a
suitable { stability } result for the solution of the stochastic state
equation, we prove the well-posedness of the stochastic linearized state
equation and show that the G\^ateaux derivative of the control-to-state
mapping corresponds to the unique solution of the linearized equation. Next,
we study the stochastic backward adjoint equation and establish a duality
relation between the solutions of the forward linearized equation and the
backward adjoint equation. Finally, we derive the first-order optimality
conditions.
\end{abstract}

\maketitle
\tableofcontents

\textit{Mathematics Subject Classification (2000)}: 60H15, 6D55, 93E20,
49K45.

\medskip \textit{Key words}: Stochastic Navier-Stokes equations, Navier-slip
boundary conditions, Stochastic backward equation, First-order optimality
conditions.

\section{Introduction}

\label{sec1} \setcounter{equation}{0}

\bigskip 

Optimal control problems for fluid flows have been extensively studied in
the literature and are of major importance in technology. Roughly speaking,
the control is exercised either within the domain occupied by the fluid, by
some distributed force (acting over the entire domain or over some specific
region), or on the boundary of the domain.

Distributed optimal control problems for Newtonian fluids, described by the
deterministic Navier-Stokes equations have been addressed in the literature,
let us mention \cite{deRK05}, \cite{FGHM00}, \cite{TW06}, \cite{WR10} and
references therein (see also the recent related model \cite{TC23}). The
deterministic case for distributed forces is quite well understood from the
theoretical and numerical point of views. The boundary control problems are
more singular ones but of great importance because they appear in many
technological applications, namely where the control of the flow is
implemented by an injection/suction device placed on the boundary of the
domain (see for instance \cite{bla}, \cite{bra}, \cite{CC18}, \cite{CC16}, 
\cite{deRY09}, \cite{FGHM00}, \cite{ghs}, \cite{gm}).

On the other hand, real world physical systems always experience random
fluctuations and inaccuracy of measurements, which are modelled by adding to
the deterministic partial differential equation suitable random
internal/external forces or by considering random initial data and boundary
conditions. It is worth to stress that the optimal control of a fluid flow
in a stochastic environment is much more involved that its deterministic
analogous, and just few results are available in the literature. Let us
mention \cite{BT19}, \cite{BT21}, \cite{B99}, \cite{L00}, \cite{L02}, \cite%
{S98}, \cite{TC23S}, where the authors solved stochastic tracking control
problems for Newtonian and non-Newtonian fluids in 2D and 3D. In these works
the control variables act in the interior of the domain.

The authors in \cite{ZG23} studied a stochastic boundary control problem for
the deterministic steady Navier-Stokes equations, where the stochastic
control is imposed just on the boundary by a non-homogeneous Dirichlet
boundary condition. We recall that the solution of the Navier-Stokes
supplemented with the Dirichlet boundary condition develops strong boundary
layer for small values of the viscosity. Then in the last decades a great
attention has been developed to the study of deterministic Navier-Stokes
equations supplemented with the { Navier-slip } boundary conditions. According to
the studies \cite{C}-\cite{CC16}, \cite{C96}, the non-homogeneous
Navier-slip boundary conditions are compatible with the inviscid limit
transition, which suggests that, comparing with the non-homogeneous
Dirichlet boundary conditions, the Navier-slip boundary conditions seems
more appropriate to control the evolution of turbulent flows typically
associated with high Reynolds number (or small viscosity).

The present work addresses an optimal boundary control problem for stochastic
Navier-Stokes equation driven by a multiplicative Gaussian noise, on a
bounded domain $\mathcal{O}\subset \mathbb{R}^2$. The dynamical law reads 
\begin{equation}
\left\{ 
\begin{array}{ll}
\begin{array}{l}
d\mathbf{y}=(\nu \Delta \mathbf{y}-\left( \mathbf{y\cdot }\nabla \right) 
\mathbf{y}-\nabla \pi )\,dt+\mathbf{G}(t,\mathbf{y})\,d{\mathcal{W}}_{t}, \\ 
\\ 
\mbox{div}\,\mathbf{y}=0,%
\end{array}
& \mbox{in}{\ \mathcal{O}_{T}=(0,T)\times \mathcal{O}},\vspace{2mm} \\ 
\mathbf{y}\cdot \mathbf{n}=a,\;\quad \left[ 2D(\mathbf{y})\,\mathbf{n}%
+\alpha \mathbf{y}\right] \cdot {\bm{\tau }}=b\;\quad & \text{on}\ \Gamma
_{T}=(0,T)\times \Gamma ,\vspace{2mm} \\ 
\mathbf{y}(0,\mathbf{x})=\mathbf{y}_{0}(\mathbf{x}) & \mbox{in}\ {\mathcal{O}%
},%
\end{array}%
\right.  \label{NSy}
\end{equation}
where $\mathbf{y}=\mathbf{y}(t,\mathbf{x})$ is the 2D-velocity random field, 
$\pi =\pi (t,\mathbf{x})$ is the pressure, $\nu >0$ is the viscosity and $%
\mathbf{y}_{0}$ is the initial condition that verifies 
\begin{equation}
\mbox{div}\,\mathbf{y}_{0}=0\qquad \mbox{ in   }\ {\mathcal{O}}.
\label{ICNS}
\end{equation}%
Here $\ $%
\begin{equation*}
D(\mathbf{y})=\frac{1}{2}[\nabla \mathbf{y}+(\nabla \mathbf{y})^{T}]
\end{equation*}%
is the rate-of-strain tensor; $\mathbf{n}$ is the external unit normal to
the boundary $\Gamma \in C^{2}$ of the domain ${\mathcal{O}}$ and $\bm{\tau }
$ is the tangent unit vector to $\Gamma ,$ such that $(\mathbf{n},\bm{\tau }%
) $ forms a standard orientation in $\mathbb{R}^{2}.$\ The positive constant 
$\alpha $ \ is the so-called friction coefficient. The quantity $a$
corresponds to the inflow and outflow fluid through $\Gamma $, satisfying \
the compatibility condition%
\begin{equation}
\int_{\Gamma }a(t,\mathbf{x})\,\,d\mathbf{\gamma }=0\quad \quad 
\mbox{ for
any  }\;t\in \lbrack 0,T].  \label{eqC2}
\end{equation}%
This condition means that the quantity of inflow fluid should coincide with
the quantity of outflow fluid. The boundary functions $a$ \ and $b$ will be
considered as the control variables for the physical system \eqref{NSy}. The
term $\mathbf{G}(t,\mathbf{y})\,{\mathcal{W}}_{t}$ is a multiplicative white
noise.

The main goal is to minimize the following cost functional 
\begin{equation*}
\displaystyle J(a,b,\mathbf{y})=\mathbb{E}\int_{{\mathcal{O}}_{T}}\frac{1}{2}%
|\mathbf{y}-\mathbf{y}_{d}|^{2}\,d\mathbf{x}dt+\mathbb{E}\int_{\Gamma
_{T}}\left( \frac{\lambda _{1}}{2}|a|^{2}+\frac{\lambda _{2}}{2}%
|b|^{2}\right) \,d\mathbf{\gamma }dt,
\end{equation*}%
constrained to the stochastic Navier-Stokes equation (\ref{NSy}), where $%
\mathbf{y}_{d}\in L_{2}(\Omega \times {\mathcal{O}}_{T})$ is a desired
target field, $\lambda _{1},\lambda _{2}>0$ and $(a,b)$ will be taken is an
appropriate space of admissible controls. {
From physical point of view, the variable  $a$ and $b$, describe the quantity of the fluid crossing the boundary,  and the  vorticity component  on the boundary (see
\cite{C} Corollary $1$ and \cite{CC1}), respectively,  which can be prescribed  by the operational controller of the physical system.
In practice, the vorticity can be induced  through some mechanisms such as  rotating container walls, moving walls,  jet injection/suction (tangentially to the wall), rough surfaces, electromagnetic forces (for conducting fluids), and many others. On the other hand, the induction of vorticity in fluid dynamics is crucial for optimizing certain industrial processes, enhancing mixing efficiency, and controlling flow patterns in various applications. By precisely manipulating boundary conditions to generate desired vorticity, engineers can improve the performance of systems such as chemical reactors, aerodynamic surfaces, heating-ventilation-air conditioning systems, etc. This control allows for efficient heat transfer, uniform mixing of reactants, and improved stability of fluid flows. Optimal vorticity induction is also essential also in environmental applications, such as pollutant dispersion and water treatment processes, ensuring effective and sustainable solutions.

}

The first stage towards the resolution of this optimal control problem has
been already performed in the article \cite{CC16}, where { the authors }
established the well-posedness of the state system \eqref{NSy}, and showed
the existence of an optimal solution $(a,b)$ in a compact set of predictable
stochastic processes verifying an exponential integrability condition.

The plan of the present paper is as follows. Section \ref{sec2} introduces
the appropriate functional spaces and recalls convenient results from \cite%
{CC16} about the solvability of the stochastic Navier-Stokes equations
supplemented with the non-homogeneous Navier-slip boundary conditions. The
Section \ref{sec4} presents a stability result for the solution of the
stochastic state equation. Next, we formulate the control problem in Section %
\ref{sec3}.

Section \ref{sec5} improves the integrability properties of the state, which
is crucial in Section \ref{sec6} to show the well-posedness of the
linearized state equation. Then, we proceed by studying the G\^ateaux
differentiability of the control-to-state mapping in Section \ref{sec7}.
Section \ref{sec8} is devoted to the analysis of the backward stochastic
adjoint equation. Finally, Section \ref{sec9} is devoted to show a suitable
duality relation between the solution of the linearized equation and the
solution of the adjoint equation, which is the main ingredient to deduce the
first-order-optimality conditions in Section \ref{sec10}.

\section{Functional setting and preliminary results to the state equation}

\label{sec2}\setcounter{equation}{0}

First let us introduce the notations and present some results, used in the
article.

Let us consider a real Banach space $X$, endowed with the norm $\left\Vert
\cdot \right\Vert _{X}.$ \ The space of $X$-valued measurable $p-$integrable
functions, defined on the time interval $[0,T],$ is denoted by $L_{p}(0,T;X)$
for $p\geqslant 1$.

The space $L_{p}(\Omega ,L_{r}(0,T;X))$ for $p,r\geqslant 1$ of the
processes $\mathbf{v}=\mathbf{v}(\omega ,t)$ with values in $X,$ defined on $%
\ \Omega \times \lbrack 0,T]$ and adapted to the filtration $\left\{ 
\mathcal{F}_{t}\right\} _{t\in \lbrack 0,T]}$, is endowed with the norms%
\begin{equation*}
\left\Vert \mathbf{v}\right\Vert _{L_{p}(\Omega ,L_{r}(0,T;X))}=\left( 
\mathbb{E(}\int_{0}^{T}\left\Vert \mathbf{v}\right\Vert _{X}^{r}\,dt)^{\frac{%
p}{r}}\right) ^{\frac{1}{p}}\text{ }
\end{equation*}%
and%
\begin{equation*}
\left\Vert \mathbf{v}\right\Vert _{L_{p}(\Omega ,L_{\infty }(0,T;X))}=\left( 
\mathbb{E}\sup_{t\in \lbrack 0,T]}\left\Vert \mathbf{v}\right\Vert _{X}^{p}\
\right) ^{\frac{1}{p}}\quad \text{for }r=\infty ,
\end{equation*}%
where $\mathbb{E}$ is the mathematical expectation with respect to the
probability measure $P,$ defined on $\Omega .$ \ By default we will omit the
dependence on the parameter $\omega \in \Omega $ in the notation for
processes $\mathbf{v}=\mathbf{v}(\omega ,t).$

We use the standard notation for the Lebesgue spaces $L_{p}(\mathcal{O})$
and the Sobolev spaces $H^{p}(\mathcal{O})$ with the norms denoted by $\Vert
\cdot \Vert _{p}$ and $\Vert \cdot \Vert _{H^{p}}$,\ respectively. The inner
product in $L_{2}(\mathcal{O})$ is denoted as $(\cdot ,\cdot )$ with the
associated norm $\Vert \cdot \Vert _{2}$. Let us introduce the following
divergence free spaces%
\begin{eqnarray*}
H &=&\{\mathbf{v}\in L_{2}({\mathcal{O}}):\,\mbox{div }\mathbf{v}=0\;\ \text{
in}\;\mathcal{D}^{\prime }({\mathcal{O}}),\;\ \mathbf{v}\cdot \mathbf{n}%
=0\;\ \text{ in}\;H^{-1/2}(\Gamma )\}, \\
V &=&\{\mathbf{v}\in H^{1}({\mathcal{O}}):\,\mbox{div }\mathbf{v}=0\;\;\text{%
a.e. in}\;{\mathcal{O}},\;\ \mathbf{v}\cdot \mathbf{n}=0\;\ \text{ in}%
\;H^{1/2}(\Gamma )\},
\end{eqnarray*}%
where the following inner product 
\begin{equation*}
\left( \mathbf{v},\mathbf{z}\right) _{V}=2\left( D\mathbf{v},D\mathbf{z}%
\right) +\alpha \int_{\Gamma }\mathbf{v}\cdot \mathbf{z}\,\,d\mathbf{\gamma }
\end{equation*}%
is considered on the space $V$ endowed with the norm $\Vert \mathbf{v}\Vert
_{V}=\sqrt{\left( \mathbf{v},\mathbf{v}\right) _{V}}.$ \ 

In what follows we will often use the results of continuous embedding%
\begin{equation}
H^{1}(0,T)\subset C([0,T]),\qquad H^{1}(\mathcal{O})\subset L_{2}(\Gamma ).
\label{aa}
\end{equation}%
Let us present the results of \cite{lad} (see the pages 62, 69), of \cite%
{nir} (the page 125) and of \cite{tem} (the pages 16-20) in the next lemma.

\begin{lemma}
\label{gag} Let us introduce the notation 
\begin{equation}
\mathbf{v}_{\mathcal{O}}=\int_{\mathcal{O}}\mathbf{v\ }d\mathbf{x.}
\label{000}
\end{equation}%
For $q\geqslant 2$ and $\mathbf{v}\in H^{1}(\mathcal{O})$ \ the
Gagliano-Nirenberg-Sobolev inequality 
\begin{equation}
||\mathbf{v-v}_{\mathcal{O}}||_{q}\leqslant C||\mathbf{v}||_{2}^{2/q}||%
\nabla \mathbf{v}||_{2}^{1-2/q},\quad  \label{LI}
\end{equation}%
and the trace interpolation inequality%
\begin{equation}
||\mathbf{v-v}_{\mathcal{O}}||_{L_{q}(\Gamma )}\leqslant C||\mathbf{v}%
||_{2}^{1/q}||\nabla \mathbf{v}||_{2}^{1-1/q}  \label{TT}
\end{equation}%
are valid. Moreover, any $\mathbf{v}\in V$ satisfies the Korn inequality 
\begin{equation}
\left\Vert \mathbf{v}\right\Vert _{H^{1}}\leqslant C\left\Vert \mathbf{v}%
\right\Vert _{V}.  \label{korn}
\end{equation}
\end{lemma}

Let us remark that 
\begin{equation*}
\mathbf{v}_{\mathcal{O}}=0,\qquad \forall \mathbf{v}\in V,
\end{equation*}%
\ since 
\begin{equation*}
\int_{\mathcal{O}}v_{j}\mathbf{\ }d\mathbf{x}=\int_{\mathcal{O}}\mbox{div }(%
\mathbf{v}x_{j})\mathbf{\ }d\mathbf{x}=\int_{\Gamma }x_{j}(\mathbf{v}\cdot 
\mathbf{n)\ }d\mathbf{\gamma }=0\qquad \text{for}\mathit{\ \ \ }j=1,2.
\end{equation*}

Also for arbitrary $\mathbf{v}\in H^{2}(\mathcal{O})$, $\ \mathbf{z}\in V$
integrating by parts, we have 
\begin{equation}
-\int_{\mathcal{O}}\triangle \mathbf{v}\cdot \mathbf{z}\,d\mathbf{x}=2\int_{%
\mathcal{O}}\,D\mathbf{v}\cdot D\mathbf{z}\,d\mathbf{x}-\int_{\Gamma }2(%
\mathbf{n}\cdot D\mathbf{v})\cdot \mathbf{z}\,d\mathbf{\gamma },
\label{integrate}
\end{equation}%
therefore 
\begin{equation}
-\int_{\mathcal{O}}\triangle \mathbf{v}\cdot \mathbf{z}\,d\mathbf{x}=(%
\mathbf{v},\mathbf{z}\,)_{V}-\int_{\Gamma }b(\mathbf{z}\cdot {\bm{\tau })}\,d%
\mathbf{\gamma },  \label{integrate2}
\end{equation}%
if the function $\mathbf{v}$ fulfills the Navier-slip boundary conditions
(see in the system \eqref{NSy}).

We will frequently apply in our considerations the Young inequality 
\begin{equation}
uv\leqslant \frac{u^{p}}{p}+\frac{v^{q}}{q},\quad \quad \frac{1}{p}+\frac{1}{%
q}=1,\quad \forall p,\,q>1.  \label{yi}
\end{equation}

Let us define the norm and the absolute value of the inner product for a
vector 
\begin{equation*}
\mathbf{h}=(\mathbf{h}_{1},\dots ,\mathbf{h}_{m})\in H^{m}=\overbrace{%
H\times ...\times H}^{m-times}
\end{equation*}%
with a fixed $\mathbf{v}\in H$ as%
\begin{equation}
\left\Vert \mathbf{h}\right\Vert _{2}=\left(\sum_{k=1}^{m}\left\Vert \mathbf{h}%
_{k}\right\Vert _{2}^2\right)^{1/2}\quad \text{and}\quad |\left( \mathbf{h},\mathbf{v}%
\right) |=\left( \sum_{k=1}^{m}\left( \mathbf{h}_{k},\mathbf{v}\right)
^{2}\right) ^{1/2}.  \label{product}
\end{equation}%
Let 
\begin{equation*}
\mathbf{G}(t,\mathbf{y}):[0,T]\times H\rightarrow H^{m}\quad \text{with }\ 
\mathbf{G}(t,\mathbf{y})=(G^{1}(t,\mathbf{y}),\dots ,G^{m}(t,\mathbf{y}))
\end{equation*}%
be Lipschitz on $\mathbf{y}$ and satisfy the linear growth 
\begin{align}
\left\Vert \mathbf{G}(t,\mathbf{v})-\mathbf{G}(t,\mathbf{z})\right\Vert
_{2}^{2}& \leqslant K\left\Vert \mathbf{v}-\mathbf{z}\right\Vert _{2}^{2}, 
\notag \\
\left\Vert \mathbf{G}(t,\mathbf{v})\right\Vert _{2}& \leqslant K\left(
1+\left\Vert \mathbf{v}\right\Vert _{2}\right) ,\qquad \forall \mathbf{v},%
\mathbf{z}\in H,\;t\in \lbrack 0,T],  \label{G}
\end{align}%
for some constant $K>0.$ Let%
\begin{equation*}
\mathbf{G}(t,\mathbf{y})\,d{\mathcal{W}}_{t}=\sum_{k=1}^{m}G^{k}(t,\mathbf{y}%
)\,d{\mathcal{W}}_{t}^{k}
\end{equation*}%
be the stochastic noise, where ${\mathcal{W}}_{t}=({\mathcal{W}}%
_{t}^{1},\dots ,{\mathcal{W}}_{t}^{m})$ is a standard $\mathbb{R}^{m}$%
-valued Wiener process, defined on a complete probability space $(\Omega ,%
\mathcal{F},P)$ and endowed with a filtration $\left\{ \mathcal{F}%
_{t}\right\} _{t\in \lbrack 0,T]}$, such that $\mathcal{F}_{0}$\ contains
every $P$-null subset of $\Omega $.

Let us introduce the space of functions \ 
\begin{equation*}
\mathcal{H}_{p}(\Gamma )=\left\{ (a,b):||(a,b)||_{\mathcal{H}_{p}(\Gamma
)}<+\infty \right\}
\end{equation*}%
with the norm%
\begin{equation*}
||(a,b)||_{\mathcal{H}_{p}(\Gamma )}=||a||_{W_{p}^{1-\frac{1}{p}}(\Gamma
)}+\,||\partial _{t}a||_{H^{-\frac{1}{2}}(\Gamma )}+\Vert b\Vert _{W_{p}^{-%
\frac{1}{p}}(\Gamma )}+\Vert b\Vert _{L_{2}(\Gamma )}+||\partial
_{t}b||_{H^{-\frac{1}{2}}(\Gamma )}.
\end{equation*}%
Let $p \in (2,+\infty )$ be given and the data $a,b$ and $\mathbf{u}_{0}$\
belong to the following Banach spaces 
\begin{equation}
(a,b)\in L_{4}(\Omega \times (0,T);\mathcal{H}_{p}(\Gamma )), \qquad \mathbf{%
u}_{0}\in L_{4}(\Omega ;H),  \label{eq00sec12}
\end{equation}%
assuming that $(a,b)$ is a pair of predictable stochastic processes.

Let us formulate the auxiliary result, demonstrated in \cite{CC16}.

\begin{lemma}
\label{navier slip} Let $(a,b)$ be a given pair of functions, satisfying %
\eqref{eq00sec12}. Then there exists a unique solution $\mathbf{a}$ of the
Stokes problem with the non-homogeneous Navier-slip boundary condition 
\begin{equation}
\left\{ 
\begin{array}{l}
-\Delta \mathbf{a}+\nabla \pi =0,\quad \nabla \cdot \mathbf{a}=0\quad \text{
in }{\mathcal{O}}, \\ 
\\ 
\mathbf{a}\cdot \mathbf{n}=a,\quad \left[ 2D(\mathbf{a})\,\mathbf{n}+\alpha 
\mathbf{a}\right] \cdot {\bm{\tau }}=b\quad \text{ on }\Gamma ,%
\end{array}%
\right. \quad {\text{a.e. in }}\Omega \times (0,T),  \label{ha}
\end{equation}%
such that%
\begin{equation}
||\mathbf{a}||_{W_{p}^{1}({\mathcal{O}})}+||\partial _{t}\mathbf{a}||_{L_{2}(%
{\mathcal{O}})}\leqslant C||(a,b)||_{\mathcal{H}_{p}(\Gamma )}\quad \text{%
a.e. in }\Omega \times (0,T).  \label{cal}
\end{equation}%
In particular, we have%
\begin{equation}
\label{calderon1}
\begin{array}{lll}
&\mathbf{a}\in L_{4}(\Omega ;C([0,T];L_{2}({\mathcal{O}})))\cap
L_{4}(\Omega \times (0,T);C(\overline{{\mathcal{O}}})\cap W_{p}^{1}({%
\mathcal{O}})),  \vspace{2mm}\\
\qquad &\partial _{t}\mathbf{a} \in L_{4}(\Omega \times (0,T);L_{2}({%
\mathcal{O}})).  
\end{array}
\end{equation}
{\ }Here and below positive constants $C$ depend on the data considered for
our problem, such as the domain $\mathcal{O},$ \ the regularity of $\Gamma ,$
the physical constants $\nu $, $\alpha $\ and a given time moment $T$.
\end{lemma}

\bigskip

Let us recall the concept of the solution of the stochastic differential
system (\ref{NSy}) and the properties of this solution, studied in \cite%
{CC16}, that will be relevant in our study of the control problem. Using the
formula (\ref{integrate2}) and the solution $\mathbf{a}$, founded in Lemma %
\ref{navier slip}, we can introduce the concept of the strong (\textit{in
the stochastic sense}) solution to the system (\ref{NSy}).

\begin{definition}
\label{1def} Assume that the data $(a,b)$ and $\mathbf{u}_{0}$ satisfy the
regularity (\ref{eq00sec12}). { An adapted  stochastic process $\mathbf{y}=\mathbf{u}+%
\mathbf{a}$ with $\mathbf{u}\in C([0,T];H)\cap L_{2}(0,T;V),$} $P$-${\text{a.e. in }}\Omega
, $\ \ is a strong solution of \eqref{NSy} with $\mathbf{y}_{0}=\mathbf{u}%
_{0}+\mathbf{a}(0)$ \ if \ the equation 
\begin{align}
\left( \mathbf{y}(t),\boldsymbol{\varphi }\right) & =\int_{0}^{t}\left[ -\nu
\left( \mathbf{y},\boldsymbol{\varphi }\right) _{V}+\int_{\Gamma }\nu b(%
\boldsymbol{\varphi }\cdot {\bm{\tau }})\,d\mathbf{\gamma }-((\mathbf{y}%
\cdot \nabla )\mathbf{y},\boldsymbol{\varphi })\,\right] \,ds+\left( \mathbf{%
y}_{0},\boldsymbol{\varphi }\right)  \notag \\
&  \notag \\
& \quad +\int_{0}^{t}\left( \mathbf{G}(s,\mathbf{y}(s)),\boldsymbol{\varphi }%
\right) \,d{\mathcal{W}}_{s}\qquad \text{for a.e. }(\omega ,t)\in \Omega
\times (0,T),  \label{res1}
\end{align}%
holds for any $\boldsymbol{\varphi }\in V$.
\vspace{2mm}
\end{definition}
{
We notice that \eqref{G},  \eqref{calderon1}  
and  $\mathbf{u}\in C([0,T];H)$,  $P-\text{a.e. in }\Omega$, 
yield
\begin{equation*}
\int_{0}^{T}\
\| \mathbf{G}(s,\mathbf{y}(s))\|_2^2
\,ds\leq K
\int_{0}^{T}(1+\|\mathbf{y}(s)\|_2^2)\,ds< \infty, \qquad 
 P-\text{a.e. in }\Omega. 
\end{equation*}
Then  the stochastic integral
$$
\int_{0}^{t} \mathbf{G}(s,\mathbf{y}(s))
 \,d{\mathcal{W}}_{s}, \quad t\in[0,T],
$$
is well defined as a  $H$-valued local martingale (c.f. \cite{Daprato}, p. 99-100), and 
\begin{equation*}
\int_{0}^{t}\left( \mathbf{G}(s,\mathbf{y}(s)),\boldsymbol{\varphi }\right)
\,d{\mathcal{W}}_{s}=\sum_{k=1}^{m}\int_{0}^{t}\left( G^{k}(s,\mathbf{y}(s)),%
\boldsymbol{\varphi }\right) \,d{\mathcal{W}}_{s}^{k}
\quad \text{for any } \boldsymbol{\varphi }\in V.
\end{equation*}%
}
\vspace{2mm}

By the { embedding } result (\ref{aa})$_{2}$ we note that the boundary condition 
$\mathbf{y}\cdot \mathbf{n}=a$ on $\Gamma ,$\ a.e. in $\Omega \times (0,T),$%
\ is well defined$\ $for $\mathbf{y}=\mathbf{u}+\mathbf{a}$ with $\mathbf{u}%
\in L_{2}(0,T;V)$.

The following result has been proved in Theorem 3.5 of the article \cite{CC16}.
\begin{theorem}
\label{the_1} Assume that the data $(a,b)$ and $\mathbf{u}_{0}$ satisfy (\ref%
{eq00sec12}). Then, a strong solution $\mathbf{y}=\mathbf{u}+\mathbf{a}$ to
the system \eqref{NSy} exists, such that 
\begin{equation*}
\mathbf{u}\in C([0,T];H)\cap L_{4}(0,T;V),\quad P\text{-a.e. in }\Omega ,
\end{equation*}%
and there exists a positive constant $\widehat{C}_{0},$ such that for any $%
t\in \lbrack 0,T]:$ 
\begin{eqnarray*}
\mathbb{E}\sup_{s\in \lbrack 0,t]}\xi _{0}^{2}(s)\left\Vert \mathbf{u}%
(s)\right\Vert _{2}^{2} &+&\nu \mathbb{E}\int_{0}^{t}\xi
_{0}^{2}(s)\left\Vert \mathbf{u}(s)\right\Vert _{V}^{2}\,ds \\
&\leqslant &C(\mathbb{E}\left\Vert \mathbf{u}_{0}\right\Vert _{2}^{2}+%
\mathbb{E}\int_{0}^{t}\xi _{0}^{2}(||(a,b)||_{\mathcal{H}_{p}(\Gamma
)}^{2}+1)\,ds),
\end{eqnarray*}%
\begin{eqnarray}
\mathbb{E}\sup_{s\in \lbrack 0,t]}\xi _{0}^{4}(s)\left\Vert \mathbf{u}%
(s)\right\Vert _{2}^{4} &+&\nu ^{2}\mathbb{E}\left( \int_{0}^{t}\xi
_{0}^{2}\left\Vert \mathbf{u}\right\Vert _{V}^{2}\,ds\right) ^{2}  \notag \\
&\leqslant &C(\mathbb{E}\left\Vert \mathbf{y}_{0}\right\Vert _{2}^{4}+%
\mathbb{E}\int_{0}^{t}\xi _{0}^{4}(||(a,b)||_{\mathcal{H}_{p}(\Gamma
)}^{4}+1)\,ds)  \label{uny}
\end{eqnarray}%
with the function%
\begin{equation}
\xi _{0}(t)=e^{-\widehat{C}_{0}\int_{0}^{t}(1+||(a,b)||_{\mathcal{H}%
_{p}(\Gamma )}^{2})\,ds},\qquad P\text{-a.e. in }\Omega \text{.}  \label{ksi}
\end{equation}
\end{theorem}

\bigskip
{
\begin{remark}It is worth mentioning that we do not know if the solution  $\mathbf{y}$,  provided by Theorem \ref{the_1} is  integrable  with respect to the variable $\omega\in \Omega$,
 we just  have integrability of its multiplication by the weight $\xi _{0}$, which verifies 
$0<\xi _{0}\leq 1.$
On the other hand, to tackle the control problem, the 
 cost functional $J$ must be well defined, requiring the square integrability of $\mathbf{y}$. 
The  square integrability
of $\mathbf{y}$ holds under suitable integrability of the inverse of the weight.
In Section  \ref{sec3},  we  will impose additional assumptions on the boundary conditions $a, \;b$, which are necessary to deduce the first-order optimality condition, and also guarantee integrability
of $\left(\xi _{0}\right)^{-1}. $

\end{remark}
}

\section{Lipschitz continuity of the control-to-state mapping}

\label{sec4} \setcounter{equation}{0}

\bigskip

This section  shows the Lipschitz continuity of the control-to-state mapping,
which is the first step to study the Gateaux derivative of this mapping.
Here, we denote by 
\begin{equation*}
\widehat{\mathbf{\varphi }}=\mathbf{\varphi }_{1}-\mathbf{\varphi }_{2}
\end{equation*}%
{the difference} of two given functions $\mathbf{\varphi }_{1},\mathbf{%
\varphi }_{2}.$

\begin{theorem}
\label{Lips}Under the assumption \eqref{eq00sec12}, $(a_{i},b_{i})\in%
\mathcal{A}$ and $\mathbf{y}_{i,0}$, $i=1,2$, let $\mathbf{y}_{1}=\mathbf{u}%
_{1}+\mathbf{a}_{1},$\ $\mathbf{y}_{2}=\mathbf{u}_{2}+\mathbf{a}_{2}$ with 
\begin{equation*}
\mathbf{u}_{1},\mathbf{u}_{2}\in C([0,T];H)\cap L_{4}(0,T;V),\quad P\text{%
-a.e.in }\Omega ,
\end{equation*}%
be two solutions of (\ref{NSy}) in the sense of the variational equality $(%
\ref{res1}),$ satisfying the estimates $(\ref{uny})$ \ with two
corresponding boundary conditions $a_{1},\;b_{1}$, $a_{2},$ $b_{2}$ and the
initial conditions 
\begin{equation*}
\mathbf{y}_{1,0}=\mathbf{u}_{1,0}+\mathbf{a}_{1}(0),\qquad \mathbf{y}_{2,0}=%
\mathbf{u}_{2,0}+\mathbf{a}_{2}(0).
\end{equation*}%
Then there exists a constant $\widehat{C}_{1}$ such that the following
estimates hold 
\begin{align}
\mathbb{E}\sup_{s\in \lbrack 0,t]}& \xi _{1}^{2}(s)\left\Vert \widehat{%
\mathbf{u}}(s)\right\Vert _{2}^{2}+2\nu \mathbb{E}\int_{0}^{t}\xi
_{1}^{2}\Vert \widehat{\mathbf{u}}\Vert _{H^{1}}^{2}\,ds\leqslant C\biggl\{%
\mathbb{E}\left\Vert \widehat{\mathbf{u}}_{0}\right\Vert _{2}^{2}  \notag \\
& \quad +\left( \mathbb{E}\int_{0}^{t}||(\widehat{a},\widehat{b})||_{%
\mathcal{H}_{p}(\Gamma )}^{4}\,ds\right) ^{1/2}\biggr\},  \label{2222}
\end{align}%
\begin{align}
\mathbb{E}\sup_{s\in \lbrack 0,t]}& \xi _{1}^{4}(s)\left\Vert \widehat{%
\mathbf{u}}(s)\right\Vert _{2}^{4}+2\nu ^{2}\mathbb{E}\left( \int_{0}^{t}\xi
_{1}^{2}\Vert \widehat{\mathbf{u}}\Vert _{H^{1}}^{2}\,ds\right)
^{2}\leqslant C\biggl\{\mathbb{E}\left\Vert \widehat{\mathbf{u}}%
_{0}\right\Vert _{2}^{4}  \notag \\
& \quad +\mathbb{E}\int_{0}^{t}||(\widehat{a},\widehat{b})||_{\mathcal{H}%
_{p}(\Gamma )}^{4}\,ds\biggr\},  \label{1111}
\end{align}%
where the function $\xi _{1}$ is defined as 
\begin{equation}
\xi _{1}(t)=e^{-\int_{0}^{t}f_{1}(s)ds},  \label{aux}
\end{equation}%
with $f_{1}\in L_{1}(\Omega \times (0,T))$ given by 
\begin{equation}
f_{1}(t)=\widehat{C}_{1}(\nu ^{-1}+1)\left( \Vert (a_{1},b_{1})\Vert _{%
\mathcal{H}_{p}(\Gamma )}^{2}+\Vert (a_{2},b_{2})\Vert _{\mathcal{H}%
_{p}(\Gamma )}^{2}+\Vert \mathbf{u}_{1}\Vert _{V}^{2}+1\right) .
\label{f__12}
\end{equation}%
The constant $\widehat{C}_{1}$ is defined by \ the relations \eqref{f_111}-%
\eqref{f__1}.
\end{theorem}

\textbf{Proof.} Let us denote $\widehat{\mathbf{a}}$ the solution of the
Stokes problem \eqref{ha} with $(a,b)=(\widehat{a},\widehat{b})$. Recalling
that $\mathbf{y}_{1}$ and $\mathbf{y}_{2}$ verify the system (\ref{NSy}) in
the sense of Definition \ref{1def}, we have%
\begin{align*}
d\left( \widehat{\mathbf{y}}(t),\boldsymbol{\varphi }\right) & =\left[ -\nu
\left( \widehat{\mathbf{y}},\boldsymbol{\varphi }\right) _{V}+\nu
\int_{\Gamma }\widehat{b}(\boldsymbol{\varphi }\cdot {\bm{\tau }})\,d\mathbf{%
\gamma }-(\widehat{B},\boldsymbol{\varphi )}\,\right] dt \\
& \quad +(\widehat{\mathbf{G}},\boldsymbol{\varphi })\,d{\mathcal{W}}%
_{t},\quad \forall t\in \lbrack 0,T], \\
\widehat{\mathbf{y}}(0)& =\widehat{\mathbf{y}}_{0},\qquad \forall 
\boldsymbol{\varphi }\in V,\quad P\text{-a.e.in }\Omega ,
\end{align*}%
with 
\begin{equation*}
\widehat{B}=\left( \mathbf{y}_{1}\cdot \nabla \right) \mathbf{y}_{1}-\left( 
\mathbf{y}_{2}\cdot \nabla \right) \mathbf{y}_{2}\text{,\qquad }\widehat{%
\mathbf{G}}(t)=\mathbf{G}(t,\mathbf{y}_{1}(t))-\mathbf{G}(t,\mathbf{y}%
_{2}(t)).
\end{equation*}%
Taking $\boldsymbol{\varphi }=\mathbf{e}_{i}$ for each $i\in \mathbb{N},$ we
can verify that the process $\widehat{\mathbf{u}}=\widehat{\mathbf{y}}-%
\widehat{\mathbf{a}}$ satisfies the system%
\begin{eqnarray*}
d\left( \widehat{\mathbf{u}}(t),\mathbf{e}_{i}\right) &=&\left[ -\nu \left( 
\widehat{\mathbf{u}}+\widehat{\mathbf{a}},\mathbf{e}_{i}\right) _{V}-\left(
\left( \mathbf{y}_{2}\cdot \nabla \right) \widehat{\mathbf{u}}\mathbf{,e}%
_{i}\right) +\int_{\Gamma }\nu \widehat{b}(\mathbf{e}_{i}\cdot {\bm{\tau }}%
)\,d\mathbf{\gamma }+\left( \mathbf{U},\mathbf{e}_{i}\right) \,\right] \,dt
\\
&&+(\widehat{\mathbf{G}},\mathbf{e}_{i})\,d{\mathcal{W}}_{t}, \\
\widehat{\mathbf{u}}(0) &=&\widehat{\mathbf{y}}_{0}-\widehat{\mathbf{a}}(0),
\end{eqnarray*}%
with $\mathbf{U}=-\left[ \partial _{t}\widehat{\mathbf{a}}+\left( \left( 
\widehat{\mathbf{u}}+\widehat{\mathbf{a}}\right) \cdot \nabla \right) 
\mathbf{y}_{1}+\left( \mathbf{y}_{2}\cdot \nabla \right) \widehat{\mathbf{a}}%
\right] .$ Hence the It\^{o} formula gives \ 
\begin{align*}
d\left( \left( \widehat{\mathbf{u}}(t),\mathbf{e}_{i}\right) ^{2}\right) &
=2\left( \widehat{\mathbf{u}}(t),\mathbf{e}_{i}\right) [-\nu \left( \widehat{%
\mathbf{u}}+\widehat{\mathbf{a}},\mathbf{e}_{i}\right) _{V}-\left( \left( 
\mathbf{y}_{2}\cdot \nabla \right) \widehat{\mathbf{u}}\mathbf{,e}_{i}\right)
\\
& +\int_{\Gamma }\nu \widehat{b}(\mathbf{e}_{i}\cdot {\bm{\tau }})\,d\mathbf{%
\gamma }+\left( \mathbf{U},\mathbf{e}_{i}\right) \,]\,dt \\
& +2\left( \widehat{\mathbf{u}}(t),\mathbf{e}_{i}\right) (\widehat{\mathbf{G}%
},\mathbf{e}_{i})\,d{\mathcal{W}}_{t}+|(\widehat{\mathbf{G}},\mathbf{e}%
_{i})|^{2}\,dt.
\end{align*}%
Summing these equalities over $i\in \mathbb{N},$ we obtain%
\begin{equation}
d\left( \left\Vert \widehat{\mathbf{u}}\right\Vert _{2}^{2}\right) +2\nu
\left\Vert \widehat{\mathbf{u}}\right\Vert _{V}^{2}\,dt=J\,dt\,+2(\widehat{%
\mathbf{G}},\widehat{\mathbf{u}})\,d{\mathcal{W}}_{t},  \label{UP}
\end{equation}%
where we denote 
\begin{eqnarray*}
J &=&\int_{\Gamma }\left\{ -a_{2}(\widehat{\mathbf{u}}\cdot \bm{\tau }%
)^{2}+2\nu \widehat{b}(\widehat{\mathbf{u}}\cdot \bm{\tau })\,\right\} \,d%
\mathbf{\gamma }-2\left( \left[ \partial _{t}\widehat{\mathbf{a}}+\left(
\left( \widehat{\mathbf{u}}+\widehat{\mathbf{a}}\right) \cdot \nabla \right) 
\mathbf{y}_{1}\right] ,\widehat{\mathbf{u}}\right) \\
&&-2\left( \left( \mathbf{y}_{2}\cdot \nabla \right) \widehat{\mathbf{a}},%
\widehat{\mathbf{u}}\right) -2\nu \left( \widehat{\mathbf{a}},\widehat{%
\mathbf{u}}\right) _{V}+\Vert \widehat{\mathbf{G}}\Vert _{2}^{2} \\
&=&J_{1}+J_{2}+J_{3}+J_{4}+J_{5}.
\end{eqnarray*}%
We have 
\begin{align*}
J_{1}& \leqslant (\Vert a_{2}\Vert _{L_{\infty }(\Gamma )}\,+\nu )\Vert 
\widehat{\mathbf{u}}\Vert _{L_{2}(\Gamma )}^{2}+C\nu \Vert \widehat{b}\Vert
_{L_{2}(\Gamma )}^{2} \\
& \leqslant h_{1}(t)||\widehat{\mathbf{u}}||_{2}^{2}+\frac{\nu }{4}||%
\widehat{\mathbf{u}}||_{V}^{2}+C\nu \Vert \widehat{b}\Vert _{L_{2}(\Gamma
)}^{2}\,
\end{align*}%
with 
\begin{equation*}
h_{1}(t)=\frac{C}{\nu }(\Vert (a_{2},b_{2})\Vert _{\mathcal{H}_{p}(\Gamma
)}\,+\nu )^{2}\in L_{1}(0,T),\quad P\text{-a.e. in }\Omega ,
\end{equation*}%
by \eqref{eq00sec12}. The term $J_{2}$ is estimated as follows 
\begin{eqnarray*}
J_{2} &\leqslant &2\left( \Vert \partial _{t}\widehat{\mathbf{a}}\Vert
_{2}+\,\Vert \widehat{\mathbf{a}}\Vert _{C(\overline{\Omega })}\Vert \nabla 
\mathbf{y}_{1}\Vert _{2}\right) \Vert \widehat{\mathbf{u}}\Vert
_{2}+2\,\Vert \nabla \mathbf{y}_{1}\Vert _{2}\,\Vert \widehat{\mathbf{u}}%
\Vert _{4}^{2} \\
&\leqslant &2\left( \Vert \partial _{t}\widehat{\mathbf{a}}\Vert
_{2}+\,\Vert \widehat{\mathbf{a}}\Vert _{C(\overline{\Omega })}\right)
\left( 1+\Vert \nabla \mathbf{y}_{1}\Vert _{2}\right) \Vert \widehat{\mathbf{%
u}}\Vert _{2} \\
&&+2\Vert \nabla \mathbf{y}_{1}\Vert _{2}\Vert \widehat{\mathbf{u}}\Vert
_{2}\Vert \nabla \widehat{\mathbf{u}}\Vert _{2} \\
&\leqslant &h_{2}(t)\Vert \widehat{\mathbf{u}}\Vert _{2}^{2}+C\Vert (%
\widehat{a},\widehat{b})\Vert _{\mathcal{H}_{p}(\Gamma )}^{2}+\frac{\nu }{4}%
||\widehat{\mathbf{u}}||_{V}^{2}
\end{eqnarray*}%
with 
\begin{equation*}
h_{2}(t)=C\left( 1+(\nu ^{-1}+1)\Vert \nabla \mathbf{y}_{1}\Vert
_{2}^{2}\right) \in L_{1}(0,T),\quad P\text{-a.e. in }\Omega ,
\end{equation*}%
by \eqref{TT} and \eqref{uny}. Using \eqref{LI} with $q=4$, we have 
\begin{eqnarray*}
J_{3} &=&2\left[ \left( \left( \widehat{\mathbf{u}}\cdot \nabla \right) 
\widehat{\mathbf{a}},\widehat{\mathbf{u}}\right) +\left( \left( \widehat{%
\mathbf{a}}\cdot \nabla \right) \widehat{\mathbf{a}},\widehat{\mathbf{u}}%
\right) -\left( \left( \mathbf{y}_{1}\cdot \nabla \right) \widehat{\mathbf{a}%
},\widehat{\mathbf{u}}\right) \right] \\
&\leqslant &2\left[ \,\Vert \widehat{\mathbf{u}}\Vert _{4}^{2}||\nabla 
\widehat{\mathbf{a}}\Vert _{2}+\Vert \widehat{\mathbf{a}}\Vert _{L_{\infty
}}\Vert \nabla \widehat{\mathbf{a}}\Vert _{2}\Vert \widehat{\mathbf{u}}\Vert
_{2}+\Vert \mathbf{y}_{1}\Vert _{4}\Vert \nabla \widehat{\mathbf{a}}\Vert
_{2}\Vert \widehat{\mathbf{u}}\Vert _{4})\right] \\
&\leqslant &\,C\Vert \widehat{\mathbf{u}}\Vert _{2}\Vert \widehat{\mathbf{u}}%
\Vert _{V}||\nabla \widehat{\mathbf{a}}\Vert _{2}+C\Vert \nabla \widehat{%
\mathbf{a}}\Vert _{2}^{2}\Vert \widehat{\mathbf{u}}\Vert _{2} \\
&&+C\Vert \mathbf{y}_{1}\Vert _{2}^{1/2}\Vert \mathbf{y}_{1}\Vert
_{H^{1}}^{1/2}\Vert \nabla \widehat{\mathbf{a}}\Vert _{2}\Vert \widehat{%
\mathbf{u}}\Vert _{2}^{1/2}\Vert \widehat{\mathbf{u}}\Vert _{V}^{1/2} \\
&\leqslant &\,C(\nu ^{-1}+1)\Vert \nabla \widehat{\mathbf{a}}\Vert
_{2}^{2}\Vert \widehat{\mathbf{u}}\Vert _{2}^{2}+\frac{\nu }{8}\Vert 
\widehat{\mathbf{u}}\Vert _{V}^{2}+\Vert \widehat{\mathbf{a}}\Vert
_{L_{\infty }}^{2} \\
&&+\,\frac{C}{\nu }\Vert \mathbf{y}_{1}\Vert _{H^{1}}^{2}\Vert \widehat{%
\mathbf{u}}\Vert _{2}^{2}+\frac{\nu }{8}\Vert \widehat{\mathbf{u}}\Vert
_{V}^{2}+C\Vert \mathbf{y}_{1}\Vert _{2}\Vert \nabla \widehat{\mathbf{a}}%
\Vert _{2}^{2} \\
&\leqslant &\,h_{3}(t)\Vert \widehat{\mathbf{u}}\Vert _{2}^{2}+\frac{\nu }{4}%
\Vert \widehat{\mathbf{u}}\Vert _{V}^{2}+\Vert (\widehat{a},\widehat{b}%
)\Vert _{\mathcal{H}_{p}(\Gamma )}^{2}+C\Vert \mathbf{y}_{1}\Vert _{2}\Vert (%
\widehat{a},\widehat{b})\Vert _{\mathcal{H}_{p}(\Gamma )}^{2}
\end{eqnarray*}%
with 
\begin{equation*}
h_{3}(t)=C\left( (\nu ^{-1}+1)\Vert (\widehat{a},\widehat{b})\Vert _{%
\mathcal{H}_{p}(\Gamma )}^{2}+\nu ^{-1}\Vert \mathbf{y}_{1}\Vert
_{H^{1}}^{2}\right) \in L_{1}(0,T),\quad P\text{-a.e. in }\Omega ,
\end{equation*}%
Finally we have%
\begin{equation*}
J_{4}\leqslant C\nu \,\Vert \widehat{\mathbf{a}}\Vert _{H^{1}}^{2}\,+\frac{%
\nu }{4}||\widehat{\mathbf{u}}||_{V}^{2}\leqslant C\nu \,\,\Vert (\widehat{a}%
,\widehat{b})\Vert _{\mathcal{H}_{p}(\Gamma )}^{2}+\frac{\nu }{4}||\widehat{%
\mathbf{u}}||_{V}^{2}
\end{equation*}%
and 
\begin{equation*}
J_{5}\leqslant K||\widehat{\mathbf{u}}||_{2}^{2}
\end{equation*}%
by the assumption (\ref{G}). Therefore, there exists a specific constant $%
\widehat{C}_{1},$ such that $\widehat{C}_{1}\geqslant \widehat{C}_{0}$ with $%
\widehat{C}_{0}$ is introduced in (\ref{uny})-(\ref{ksi}) and 
\begin{equation}
\frac{1}{2}\left( h_{1}(t)+h_{2}(t)+h_{3}(t)+K\right) \ \leqslant f_{1}(t),
\label{f_111}
\end{equation}%
where%
\begin{equation}
f_{1}(t)=\widehat{C}_{1}(\nu ^{-1}+1)\left( \Vert (a_{1},b_{1})\Vert _{%
\mathcal{H}_{p}(\Gamma )}^{2}+\Vert (a_{2},b_{2})\Vert _{\mathcal{H}%
_{p}(\Gamma )}^{2}+\Vert \mathbf{u}_{1}\Vert _{V}^{2}+1\right) .
\label{f__1}
\end{equation}%
{Combining } the above deduced estimates for $J_{i},$ $i=1,...,5,$ we get%
\begin{eqnarray}
J &\leqslant &\left( h_{1}(t)+h_{2}(t)+h_{3}(t)+K\right) \Vert \widehat{%
\mathbf{u}}\Vert _{2}^{2}+\nu ||\widehat{\mathbf{u}}||_{V}^{2}  \notag \\
&&+C\left\{ \nu \Vert \widehat{b}\Vert _{L_{2}(\Gamma )}^{2}+\,(1+\nu )\Vert
(\widehat{a},\widehat{b})\Vert _{\mathcal{H}_{p}(\Gamma )}^{2}+\Vert \mathbf{%
y}_{1}\Vert _{2}||(\widehat{a},\widehat{b})||_{\mathcal{H}_{p}(\Gamma
)}^{2}\right\}  \notag \\
&\leqslant &2f_{1}(t)\Vert \widehat{\mathbf{u}}\Vert _{2}^{2}+\nu ||\widehat{%
\mathbf{u}}||_{V}^{2}+C\left( 1+\Vert \mathbf{y}_{1}\Vert _{2}\right) \Vert (%
\widehat{a},\widehat{b})\Vert _{\mathcal{H}_{p}(\Gamma )}^{2}.  \label{J}
\end{eqnarray}

Now introducing the function 
\begin{equation}
\xi _{1}(t)=e^{-\int_{0}^{t}f_{1}(s)ds}  \label{auxx}
\end{equation}%
and applying It\^{o}'s formula to (\ref{UP}), we infer%
\begin{eqnarray*}
d\left( \xi _{1}^{2}(t)\left\Vert \widehat{\mathbf{u}}\right\Vert
_{2}^{2}\right) &+&2\nu \xi _{1}^{2}(t)\left\Vert \widehat{\mathbf{u}}%
\right\Vert _{V}^{2}\,dt=\xi _{1}^{2}(t)J\,dt\, \\
&+&2\xi _{1}^{2}(t)(\widehat{\mathbf{G}},\widehat{\mathbf{u}})\,d{\mathcal{W}%
}_{t}-2f_{1}(t)\xi _{1}^{2}(t)||\widehat{\mathbf{u}}||_{2}^{2}\,dt.
\end{eqnarray*}%
Integrating it over the interval $[0,t]$ and using the estimate (\ref{J}),
we derive 
\begin{align}
\xi _{1}^{2}(t)\left\Vert \widehat{\mathbf{u}}(t)\right\Vert _{2}^{2}+\nu
\int_{0}^{t}\xi _{1}^{2}\Vert \widehat{\mathbf{u}}\Vert _{V}^{2}\,ds&
\leqslant \left\Vert \widehat{\mathbf{u}}_{0}\right\Vert
_{2}^{2}+C\int_{0}^{t}\xi _{1}^{2}\left( 1+\Vert \mathbf{y}_{1}\Vert
_{2}\right) ||(\widehat{a},\widehat{b})||_{\mathcal{H}_{p}(\Gamma )}^{2}\,ds
\notag \\
& \quad +2\int_{0}^{t}\xi _{1}^{2}(\widehat{\mathbf{G}},\widehat{\mathbf{u}}%
)\,d{\mathcal{W}}_{s},  \label{111}
\end{align}%
which implies 
\begin{align*}
\sup_{s\in \lbrack 0,t]}\xi _{1}^{2}(s)\left\Vert \widehat{\mathbf{u}}%
(s)\right\Vert _{2}^{2}& +\nu \int_{0}^{t}\xi _{1}^{2}\Vert \widehat{\mathbf{%
u}}\Vert _{V}^{2}\,ds \\
& \leqslant \left\Vert \widehat{\mathbf{u}}_{0}\right\Vert
_{2}^{2}+C\int_{0}^{t}\xi _{1}^{2}\left( 1+\Vert \mathbf{y}_{1}\Vert
_{2}\right) ||(\widehat{a},\widehat{b})||_{\mathcal{H}_{p}(\Gamma )}^{2}\,ds
\\
& \quad +2\sup_{s\in \lbrack 0,t]}\int_{0}^{s}\xi _{1}^{2}(\widehat{\mathbf{G%
}},\widehat{\mathbf{u}})\,d{\mathcal{W}}_{r}.
\end{align*}%
Taking the expectation in this inequality, using the assumption (\ref{G})
and the Burkholder-Davis-Gundy inequality 
\begin{align*}
\mathbb{E}\sup_{s\in \lbrack 0,t]}\left\vert \int_{0}^{s}\xi _{1}^{2}(r)(%
\widehat{\mathbf{G}},\widehat{\mathbf{u}})\,\,d{\mathcal{W}}_{r}\right\vert
& \leqslant \mathbb{E}\left( \int_{0}^{t}\xi _{1}^{4}\left\vert \left( 
\mathbf{G}(s,\mathbf{y}_{1})-\mathbf{G}(s,\mathbf{y}_{2}),\widehat{\mathbf{u}%
}\right) \right\vert ^{2}\,ds\right) ^{\frac{1}{2}} \\
& \leqslant \frac{1}{2}\,\mathbb{E}\sup_{s\in \lbrack 0,t]}\xi
_{1}^{2}(s)\Vert \widehat{\mathbf{u}}(s)\Vert _{2}^{2}+C\mathbb{E}%
\int_{0}^{t}\xi _{1}^{2}(\left\Vert \widehat{\mathbf{u}}\right\Vert
_{2}^{2}+\left\Vert \widehat{\mathbf{a}}\right\Vert _{2}^{2})\,ds,
\end{align*}%
we deduce 
\begin{align*}
\frac{1}{2}\mathbb{E}\sup_{s\in \lbrack 0,t]}&\xi _{1}^{2}(s)\left\Vert 
\widehat{\mathbf{u}}(s)\right\Vert _{2}^{2} +\nu \mathbb{E}\int_{0}^{t}\xi
_{1}^{2}\Vert \widehat{\mathbf{u}}\Vert _{V}^{2}\,ds \\
\leqslant& C\mathbb{E}\int_{0}^{t}\sup_{s\in \lbrack 0,t]}\xi
_{1}^{2}(s)\Vert \widehat{\mathbf{u}}(s)\Vert _{2}^{2}\,ds+C\mathbb{E}%
\left\Vert \widehat{\mathbf{u}}_{0}\right\Vert _{2}^{2} \\
&+C\left( \mathbb{E}\int_{0}^{t}\xi _{1}^{2}\left( 1+\Vert \mathbf{y}%
_{1}\Vert _{2}^{2}\right) \,ds\right) ^{1/2}\left( \mathbb{E}\int_{0}^{t}\xi
_{1}^{2}||(\widehat{a},\widehat{b})||_{\mathcal{H}_{p}(\Gamma
)}^{4}\,ds\right) ^{1/2}
\end{align*}%
by (\ref{cal}).  Then the { Gr\"onwall } inequality yields 
\begin{align*}
\mathbb{E}\sup_{s\in \lbrack 0,t]}&\xi _{1}^{2}(s)\left\Vert \widehat{%
\mathbf{u}}(s)\right\Vert _{2}^{2} +2\nu \mathbb{E}\int_{0}^{t}\xi
_{1}^{2}\Vert \widehat{\mathbf{u}}\Vert _{V}^{2}\,ds\leqslant C\biggl\{%
\mathbb{E}\left\Vert \widehat{\mathbf{u}}_{0}\right\Vert _{2}^{2} \\
&+\left( \mathbb{E}\int_{0}^{t}\xi _{1}^{2}\left( 1+\Vert \mathbf{y}%
_{1}\Vert _{2}^{2}\right) \,ds\right) ^{1/2}\left( \mathbb{E}\int_{0}^{t}\xi
_{1}^{2}||(\widehat{a},\widehat{b})||_{\mathcal{H}_{p}(\Gamma
)}^{4}\,ds\right) ^{1/2}\biggr\},
\end{align*}%
which implies \eqref{2222}, applying (\ref{cal}), (\ref{uny}), (\ref{ksi})
and the property that $\xi _{1}\leqslant 1$ a.e. in $\Omega \times \lbrack
0,T]$.

On the other hand, taking the square of \eqref{111}, we infer that 
\begin{align*}
\mathbb{E}\sup_{s\in \lbrack 0,t]}& \xi _{1}^{4}(s)\left\Vert \widehat{%
\mathbf{u}}(s)\right\Vert _{2}^{4}+\nu ^{2}\mathbb{E}\left( \int_{0}^{t}\xi
_{1}^{2}\Vert \widehat{\mathbf{u}}\Vert _{V}^{2}\,ds\right) ^{2} \\
& \leqslant C\mathbb{E}\left\Vert \widehat{\mathbf{u}}_{0}\right\Vert
_{2}^{4}+C\mathbb{E}\left( \int_{0}^{t}\xi _{1}^{2}\left( 1+\Vert \mathbf{y}%
_{1}\Vert _{2}\right) ||(\widehat{a},\widehat{b})||_{\mathcal{H}_{p}(\Gamma
)}^{2}\,ds\right) ^{2} \\
& \quad +C\mathbb{E}\sup_{s\in \lbrack 0,t]}\left( \int_{0}^{s}\xi
_{1}^{2}(r)(\widehat{\mathbf{G}},\widehat{\mathbf{u}})\,d{\mathcal{W}}%
_{r}\right) ^{2} \\
& \leqslant C\mathbb{E}\left\Vert \widehat{\mathbf{u}}_{0}\right\Vert
_{2}^{4}+C\mathbb{E}\int_{0}^{t}\xi _{1}^{2}\left( 1+\Vert \mathbf{y}%
_{1}\Vert _{2}^{2}\right) \,ds\times \mathbb{E}\int_{0}^{t}\xi _{1}^{2}||(%
\widehat{a},\widehat{b})||_{\mathcal{H}_{p}(\Gamma )}^{4}\,ds \\
& \quad +C\mathbb{E}\sup_{s\in \lbrack 0,t]}\left( \int_{0}^{s}\xi _{1}^{2}(%
\widehat{\mathbf{G}},\widehat{\mathbf{u}})\,d{\mathcal{W}}_{r}\right) ^{2}.
\end{align*}%
The assumption (\ref{G}) and the Burkholder-Davis-Gundy inequality 
\begin{align*}
\mathbb{E}\sup_{s\in \lbrack 0,t]}&\left\vert \int_{0}^{s}\xi _{1}^{2}(r)(%
\widehat{\mathbf{G}},\widehat{\mathbf{u}})\,\,d{\mathcal{W}}_{r}\right\vert
^{2}\leqslant \mathbb{E}\int_{0}^{t}\xi _{1}^{4}(s)\left\vert \left( \mathbf{%
G}(s,\mathbf{y}_{1})-\mathbf{G}(s,\mathbf{y}_{2}),\widehat{\mathbf{u}}%
\right) \right\vert ^{2}\,ds \\
& \leqslant \frac{1}{2}\,\mathbb{E}\sup_{s\in \lbrack 0,t]}\xi
_{1}^{4}(s)\Vert \widehat{\mathbf{u}}(s)\Vert _{2}^{4}+C\mathbb{E}%
\int_{0}^{t}\xi _{1}^{4}(\left\Vert \widehat{\mathbf{u}}\right\Vert
_{2}^{4}+\left\Vert \widehat{\mathbf{a}}\right\Vert _{2}^{4})\,ds
\end{align*}%
allow to deduce the following { Gr\"onwall } inequality 
\begin{align*}
\frac{1}{2}\mathbb{E}\sup_{s\in \lbrack 0,t]}&\xi _{1}^{4}(s)\left\Vert 
\widehat{\mathbf{u}}(s)\right\Vert _{2}^{4} +\nu ^{2}\mathbb{E}\left(
\int_{0}^{t}\xi _{1}^{2}\Vert \widehat{\mathbf{u}}\Vert _{V}^{2}\,ds\right)
^{2} \\
& \leqslant C\mathbb{E}\left\Vert \widehat{\mathbf{u}}_{0}\right\Vert
_{2}^{4}+C\mathbb{E}\int_{0}^{t}\xi _{1}^{2}\left( 1+\Vert \mathbf{y}%
_{1}\Vert _{2}^{2}\right) \,ds\times \mathbb{E}\int_{0}^{t}\xi _{1}^{2}||(%
\widehat{a},\widehat{b})||_{\mathcal{H}_{p}(\Gamma )}^{4}\,ds \\
& \quad +C\mathbb{E}\int_{0}^{t}\xi _{1}^{4}\left\Vert \widehat{\mathbf{a}}%
\right\Vert _{2}^{4}\,ds+C\mathbb{E}\int_{0}^{t}\sup_{r\in \lbrack 0,s]}\xi
_{1}^{4}(r)\left\Vert \widehat{\mathbf{u}}(r)\right\Vert _{2}^{4}\,ds,
\end{align*}%
implying the inequality 
\begin{align*}
\mathbb{E}\sup_{s\in \lbrack 0,t]}&\xi _{1}^{4}(s)\left\Vert \widehat{%
\mathbf{u}}(s)\right\Vert _{2}^{4} +2\nu ^{2}\mathbb{E}\left(
\int_{0}^{t}\xi _{1}^{2}\Vert \widehat{\mathbf{u}}\Vert _{V}^{2}\,ds\right)
^{2} \\
& \leqslant C\biggl\{\mathbb{E}\left\Vert \widehat{\mathbf{u}}%
_{0}\right\Vert _{2}^{4}+\mathbb{E}\int_{0}^{t}\xi _{1}^{2}\left( 1+\Vert 
\mathbf{y}_{1}\Vert _{2}^{2}\right) \,ds\times \mathbb{E}\int_{0}^{t}\xi
_{1}^{2}||(\widehat{a},\widehat{b})||_{\mathcal{H}_{p}(\Gamma )}^{4}\,ds \\
& \quad +C\mathbb{E}\int_{0}^{t}\xi _{1}^{4}\left\Vert \widehat{\mathbf{a}}%
\right\Vert _{2}^{4}\,ds\biggr\},
\end{align*}%
which gives \eqref{1111}, using (\ref{cal}), (\ref{uny}), (\ref{ksi}) and $%
\xi _{1}\leqslant 1$ a.e. in $\Omega \times \lbrack 0,T]$.$\hfill
\blacksquare $

\bigskip

\section{Solvability of control problem}

\label{sec3} \setcounter{equation}{0}

\bigskip

The main goal of this paper is to control the solution of the system (\ref%
{NSy}) by boundary values $(a,b)$. From now on, we assume these boundary
values belong to the space of $\mathcal{H}_{p}(\Gamma )$-valued bounded
stochastic processes (c.f. \cite{B99}, \cite{L00}, \cite{L02})); namely we
defined the space $\mathcal{A}$ of admissible controls as a 
{
bounded subset of 
$L_{\infty}(\Omega \times (0,T);\mathcal{H}_{p}(\Gamma ))$, which is \textit{compact}
in   $L_{2}(\Omega \times (0,T);\mathcal{H}_{p}(\Gamma )).$
}

The cost functional is given by%
\begin{equation}
\displaystyle J(a,b,\mathbf{y})=\mathbb{E}\int_{{\mathcal{O}}_{T}}\frac{1}{2}%
|\mathbf{y}-\mathbf{y}_{d}|^{2}\,d\mathbf{x}dt+\mathbb{E}\int_{\Gamma _{T}}(%
\frac{\lambda _{1}}{2}|a|^{2}+\frac{\lambda _{2}}{2}|b|^{2})\,d\mathbf{%
\gamma }dt,  \label{cost}
\end{equation}%
where $\mathbf{y}_{d}\in L_{2}(\Omega \times {\mathcal{O}}_{T})$ is a
desired target field and $\lambda _{1},\lambda _{2}>0.$ We aim to control
the solution $\mathbf{y}$ through the minimization of the cost functional (%
\ref{cost}) over $\mathcal{A}$ and constrained to \eqref{NSy}. More
precisely, our goal is to solve the following problem 
\begin{equation*}
(\mathcal{P})\left\{ 
\begin{array}{l}
\underset{(a,b)}{\mbox{minimize}}\{J(a,b,\mathbf{y}):~(a,b)\in \mathcal{A}%
\quad \text{and}\quad \\ 
\\ 
\mathbf{y}\mbox{  is  the solution of  the system }\eqref{NSy}\mbox{  for   }%
(a,b)\in \mathcal{A}\}.%
\end{array}%
\vspace{3mm}\quad \right.
\end{equation*}%
Let us notice that for $(a,b)\in \mathcal{A}$, the solution $\mathbf{y}=%
\mathbf{u}+\mathbf{a}$ of the state equation \eqref{res1} satisfies 
\begin{equation*}
\mathbf{y}\in L_{4}(\Omega ;C([0,T];L_{2}(\mathcal{O})))\subset L_{2}(\Omega
\times \mathcal{O}_{T}),
\end{equation*}%
ensuring that the cost functional \eqref{cost} is well defined.

\bigskip

{The existence result stated in Theorem 4.1 of \cite{CC16} applies, and the optimal solution belongs to $\mathcal{A}$, namely we have the next theorem.}
\begin{teo}
\label{main_existence} Let  $\mathbf{y}%
_{0}$ verify the assumptions \eqref{eq00sec12}. Then there exists at least
one solution for the optimal control problem $(\mathcal{P}).$
\end{teo}

{
\begin{remark}
It is worth mentioning that the existence result established in \cite{CC16} does not require the admissible set $\mathcal{A}$  to be bounded in $L_{\infty}(\Omega \times (0,T);\mathcal{H}_{p}(\Gamma ))$.
However, the deduction of first-order optimality conditions is more demanding, requiring integrability of the Gâteaux derivative of the control-to-state mapping. This will be achieved proving an exponential integrability condition for the state  in Proposition \ref{pop} below.
For that we need $\mathcal{A}$  to be bounded in $L_{\infty}(\Omega \times (0,T);\mathcal{H}_{p}(\Gamma ))$.
This requirement comes from the structure of the stochastic Navier-Stokes equation with multiplicative noise.
Roughly speaking, this means that the  control (in an optimal way) the stochastic dynamics inside the domain, through the control actions on the boundary, it is possible even if the control action have some randomness, however this randomness should be bounded (see the assumption (vi) on p. 40 of \cite{L00} and the assumption on \it{bounded controls} given on p. 42 of \cite{B99}). 
A very particular case corresponds to deterministic control actions.
\end{remark}
}

The next sections are devoted to establishing first-order optimality
conditions.

\section{Exponential integrability of the state}


\label{sec5} \setcounter{equation}{0}  According to Sections \ref{sec2}, \ref%
{sec3}, let $\mathbf{y}=\mathbf{u}+\mathbf{a}$ be the unique solution of the
stochastic differential equation\ (\ref{NSy}) satisfying the estimates (\ref%
{uny}) { and } with data $(a,b),\mathbf{u}_{0}$ verifying (\ref{eq00sec12}).
Hereafter, we assume the following additional assumptions on the data 
\begin{equation}
(a,b)\in \mathcal{A},\qquad ||\mathbf{u}_{0}||_{2}\in L_{\infty }(\Omega ),
\label{bound0}
\end{equation}%
in order to deduce a suitable exponential integrability condition for the
stochastic process $\mathbf{y}$. We also introduce additional hypothesis on
the diffusion operator $\mathbf{G}$, namely $\mathbf{G}$ is bounded by a
positive constant $L$ in the space $H,$ such that\ 
\begin{equation}
||\mathbf{G}\left( t,\mathbf{y}\right) ||_{2}^{2}\leqslant \frac{L}{1+||%
\mathbf{y}||_{2}^{2}}\quad \quad \text{for a.e. \ }t\in \lbrack 0,T],\quad
\forall \mathbf{y}\in H.  \label{GG}
\end{equation}%
Since $\mathcal{A}$ is a closed subset of $L_{\infty }(\Omega \times (0,T);%
\mathcal{H}_{p}(\Gamma ))$, we can take the real number 
\begin{equation*}
r_{\ast }=\sup_{(a,b)\in \mathcal{A}}\,2\widehat{C}_{0}(1+||(a,b)||_{L_{%
\infty }(\Omega _{T};\mathcal{H}_{p}(\Gamma ))}^{2}),
\end{equation*}%
where the constant $\widehat{C}_{0}$ is introduced in (\ref{ksi}) of Theorem %
\ref{the_1}, and define the time dependent functions $\lambda _{\ast }(t),$ $%
\beta _{\ast }(t)$\ and the constants $A_{\ast },$ $B_{\ast }$ by%
\begin{eqnarray}
\lambda _{\ast }(t) &=&\frac{\nu e^{-r_{\ast }t}}{L},\quad A_{\ast }=\frac{%
\nu ^{2}e^{-2r_{\ast }T}}{2L},  \notag \\
\beta _{\ast }(t) &=&\frac{\nu e^{-4\left( r_{\ast }+L\right) t}}{8L},\quad
B_{\ast }=\frac{\nu ^{2}e^{-8\left( r_{\ast }+L\right) T}}{8L}.  \label{aaa}
\end{eqnarray}

Let us note that there exists a constant $\widehat{C},$ depending only on ${%
\mathcal{O}},$ such that 
\begin{equation}
||\mathbf{v}||_{2}^{4}\leqslant \widehat{C}||\mathbf{v}||_{2}^{2}||\mathbf{v}%
||_{V}^{2},\qquad \forall \mathbf{v}\in V,  \label{CCC}
\end{equation}%
being a particular case of the inequality (\ref{LI}).

Let us mention that the main arguments to show the exponential integrability
of the stochastic process $\mathbf{y}$ rely on the structure of the first
equation in (\ref{NSy}) for $\mathbf{y}$, and on the martingale property of
the exponential process.

\begin{proposition}
\label{pop} Assume that the data $(a,b)\in \mathcal{A},$ $\mathbf{u}_{0}$\
and $\mathbf{G}$\ \ satisfy (\ref{G}), (\ref{eq00sec12}), (\ref{bound0}) and
(\ref{GG}). Then there exist positive constants $C$, such that the following
estimates are valid 
\begin{eqnarray}
\mathbb{E}\,\exp \left( \lambda _{\ast }(t)e^{-tr_{\ast }}\left\Vert \mathbf{%
u}(t)\right\Vert _{2}^{2}\right) &\leqslant &C,\quad \mathbb{E}\,\exp %
\biggl\{A_{\ast }\int_{0}^{t}\left\Vert \mathbf{u}\right\Vert _{V}^{2}\,ds%
\biggr\}\leqslant C  \notag \\
\mathbb{E}\,\exp (\beta _{\ast }(t)e^{-4\left( r_{\ast }+L\right)
t}\left\Vert \mathbf{u}(t)\right\Vert _{2}^{4}) &\leqslant &C,\quad \mathbb{E%
}\,\exp \biggl\{B_{\ast }\int_{0}^{t}\left\Vert \mathbf{u}\right\Vert
_{2}^{2}\left\Vert \mathbf{u}\right\Vert _{V}^{2}\,ds\biggr\}\leqslant C 
\notag \\
\mathbb{E}\,\exp \left( \frac{B_{\ast }}{\widehat{C}}\int_{0}^{t}\left\Vert 
\mathbf{u}\right\Vert _{2}^{4}ds\right) &\leqslant &C,\qquad \forall t\in
\lbrack 0,T],  \label{yy}
\end{eqnarray}%
with $\lambda _{\ast },\;A_{\ast },$ $\beta _{\ast },B_{\ast }$\ and $%
\widehat{C}$ defined by (\ref{aaa}) and (\ref{CCC}).
\end{proposition}

\begin{proof}
\textit{1st step. Deduction of the estimates { (\ref{yy})}$_{1, 2}.$ }
Since $\mathbf{y}$ is the solution of the state system (\ref{NSy}), which
exists by Theorem \ref{the_1}, then taking the test function $\boldsymbol{%
\varphi }=\mathbf{u}$ in (\ref{res1}) with $\mathbf{u}=\mathbf{y}-\mathbf{a}$%
, we obtain the inequality 
\begin{eqnarray}
d\left( \left\Vert \mathbf{u}\right\Vert _{2}^{2}\right) +\nu \left\Vert 
\mathbf{u}\right\Vert _{V}^{2}\,dt &\leqslant &2\widehat{C}_{0}\left(
||(a,b)||_{\mathcal{H}_{p}(\Gamma )}^{2}+1\right) (||\mathbf{u}\Vert
_{2}^{2}+1)\,dt  \notag \\
&&  \notag \\
&&+2\left( \mathbf{G}(t,\mathbf{y}),\mathbf{u}\right) \,d{\mathcal{W}}_{t},
\label{C1}
\end{eqnarray}%
as it was done in the article \cite{CC16} (see the deduction of the formula (3.11) in the article \cite%
{CC16}). Integrating over the time interval $(0,t)$, we can write 
\begin{equation}
\left\Vert \mathbf{u}(t)\right\Vert _{2}^{2}+\nu \int_{0}^{t}\left\Vert 
\mathbf{u}\right\Vert _{V}^{2}\,ds\leqslant \kappa \int_{0}^{t}\left\Vert 
\mathbf{u}\right\Vert _{2}^{2}\,ds+(C_{\kappa }+g(t)),  \label{qq}
\end{equation}%
where%
\begin{eqnarray*}
\kappa &=&2r_{\ast }\geq 2\widehat{C}_{0}\Vert 1+||(a,b)||_{\mathcal{H}%
_{p}(\Gamma )}^{2}\Vert _{L_{\infty }(\Omega _{T})},\qquad C_{\kappa
}=\left\Vert \mathbf{u}_{0}\right\Vert _{2}^{2}+\kappa T, \\
g(t) &=&\int_{0}^{t}f(s)\,d{\mathcal{W}}_{s},\qquad f(s)=2\left( \mathbf{G}%
(s,\mathbf{y}),\mathbf{u}\right) .
\end{eqnarray*}

The relation (\ref{qq}) corresponds to the following differential inequality 
\begin{equation}
z^{\prime }\leqslant \kappa z+(C_{\kappa }+g(t))\qquad \text{for}\quad
z(t)=\int_{0}^{t}\left\Vert \mathbf{u}\right\Vert _{2}^{2}\,ds,  \label{dz}
\end{equation}%
which can be integrated by { Gr\"onwall's } lemma. Hence using Fubini's theorem,
we get 
\begin{eqnarray*}
z(t) &=&\int_{0}^{t}\left\Vert \mathbf{u}\right\Vert _{2}^{2}\,ds\leqslant
C_{\kappa }\left( \frac{e^{\kappa t}-1}{\kappa }\right) +e^{\kappa
t}\int_{0}^{t}e^{-\kappa s}g(s)\,ds \\
&=&C_{\kappa }\left( \frac{e^{\kappa t}-1}{\kappa }\right)
+\int_{0}^{t}\left( \frac{e^{\kappa (t-s)}-1}{\kappa }\right) \ f(s)\,d{%
\mathcal{W}}_{s}.
\end{eqnarray*}%
Substituting it into the right hand side of (\ref{qq}), we deduce 
\begin{equation}
\left\Vert \mathbf{u}(t)\right\Vert _{2}^{2}+\nu \int_{0}^{t}\left\Vert 
\mathbf{u}\right\Vert _{V}^{2}\,ds\leqslant C_{\kappa }e^{\kappa
t}+\int_{0}^{t}e^{\kappa (t-s)}\ f(s)\,d{\mathcal{W}}_{s}.  \label{dzz}
\end{equation}

Setting $F(s)=e^{-\kappa s}f(s)$ and multiplying the inequality (\ref{dzz})
by $\lambda e^{-\kappa t}$ with some $\lambda >0$, we infer that 
\begin{align*}
\lambda e^{-\kappa t}\left\Vert \mathbf{u}(t)\right\Vert _{2}^{2}+\nu
\lambda e^{-\kappa t}\int_{0}^{t}\left\Vert \mathbf{u}\right\Vert
_{V}^{2}\,ds& \leqslant \lambda C_{\kappa }+\frac{\lambda ^{2}}{2}%
\int_{0}^{t}|F(s)|^{2}\,ds \\
& \quad +\left( \lambda \int_{0}^{t}F(s)\,d{\mathcal{W}}_{s}-\frac{\lambda
^{2}}{2}\int_{0}^{t}|F(s)|^{2}\,ds\right) .
\end{align*}%
The boundedness (\ref{GG}) for $\mathbf{G}$ implies 
\begin{equation*}
|F(t)|^{2}=e^{-2\kappa t}|\left( \mathbf{G}(t,\mathbf{y}),\mathbf{u}\right)
|^{2}\leqslant L\left\Vert \mathbf{u}(t)\right\Vert _{2}^{2}\leqslant
L\left\Vert \mathbf{u}(t)\right\Vert _{V}^{2}\qquad \quad \forall t\in (0,T).
\end{equation*}%
In addition, for $\lambda _{\ast }(t)=\frac{\nu e^{-\kappa t}}{L}$, we have 
\begin{equation*}
\max_{\lambda }(\nu \lambda e^{-\kappa t}-\frac{L\lambda ^{2}}{2})=\frac{%
L\lambda _{\ast }^{2}}{2}\geqslant \frac{\nu ^{2}e^{-2\kappa T}}{2L}=A_{\ast
},\quad \forall t\in \lbrack 0,T].
\end{equation*}%
Hence, we obtain 
\begin{equation*}
\lambda _{\ast }e^{-\kappa t}\left\Vert \mathbf{u}\right\Vert
_{2}^{2}+A_{\ast }\int_{0}^{t}\left\Vert \mathbf{u}\right\Vert
_{V}^{2}\,ds\leqslant \lambda _{\ast }C_{\kappa }+\left( \int_{0}^{t}\left(
\lambda _{\ast }F\right) \,d{\mathcal{W}}_{s}-\frac{1}{2}\int_{0}^{t}\left(
\lambda _{\ast }F\right) ^{2}\,ds\right) .
\end{equation*}

Taking the exponential function and the expectation in the last deduced
inequality, we obtain 
\begin{align*}
\mathbb{E}\,& \exp \left( \lambda _{\ast }e^{-\kappa t}\left\Vert \mathbf{u}%
(t)\right\Vert _{2}^{2}+A_{\ast }\int_{0}^{t}\left\Vert \mathbf{u}%
\right\Vert _{V}^{2}\,ds\right) \\
& \leqslant H_{1}\times \exp \biggl\{\int_{0}^{t}\left( \lambda _{\ast
}F\right) \,d{\mathcal{W}}_{s}-\frac{1}{2}\int_{0}^{t}\left( \lambda _{\ast
}F\right) ^{2}\,ds\biggr\}
\end{align*}%
with $H_{1}=\exp \left( \lambda _{\ast }C_{\kappa }\right) <\infty $ by the
assumption (\ref{bound0}). Since the expectation of the right hand side is
equal to $1$ due to the Levy equality,\ we derive 
\begin{equation*}
\mathbb{E}\,\exp \left( \lambda _{\ast }e^{-\kappa t}\left\Vert \mathbf{u}%
(t)\right\Vert _{2}^{2}+A_{\ast }\int_{0}^{t}\left\Vert \mathbf{u}%
\right\Vert _{V}^{2}\,ds\right) \leqslant H_{1},
\end{equation*}%
which gives (\ref{yy})$_{1,2}$.

\textit{2nd step. Deduction of the estimates (\ref{yy})}$_{3,4,5}.$  Using the inequality (\ref{C1}) and applying the It\^{o} formula, we
ensure 
\begin{eqnarray*}
\left\Vert \mathbf{u}\right\Vert _{2}^{4} &+&2\nu \int_{0}^{t}\left\Vert 
\mathbf{u}\right\Vert _{2}^{2}\left\Vert \mathbf{u}\right\Vert
_{V}^{2}\,ds\leqslant \left\Vert \mathbf{u}_{0}\right\Vert _{2}^{4} \\
&+&4\int_{0}^{t}\widehat{C}_{0}(||(a,b)||_{\mathcal{H}_{p}(\Gamma
)}^{2}+1)\left\Vert \mathbf{u}\right\Vert _{2}^{2}(||\mathbf{u}\Vert
_{2}^{2}+1)\,ds \\
&+&4\int_{0}^{t}\left\Vert \mathbf{u}\right\Vert _{2}^{2}\left( \mathbf{G}(t,%
\mathbf{y}),\mathbf{u}\right) \,d{\mathcal{W}}_{s}+\int_{0}^{t}4\left( 
\mathbf{G}(s,\mathbf{y}),\mathbf{u}\right) ^{2}\,ds \\
&\leqslant &\left\Vert \mathbf{u}_{0}\right\Vert
_{2}^{4}+\int_{0}^{t}4\left( \left\{ 2\widehat{C}_{0}(||(a,b)||_{\mathcal{H}%
_{p}(\Gamma )}^{2}+1)\right\} +L\right) ||\mathbf{u}\Vert _{2}^{4}\ ds \\
&&+\int_{0}^{t}\left( \left\{ 2\widehat{C}_{0}(||(a,b)||_{\mathcal{H}%
_{p}(\Gamma )}^{2}+1\right\} +1\right) \,dt \\
&&+4\int_{0}^{t}\left\Vert \mathbf{u}\right\Vert _{2}^{2}\left( \mathbf{G}(s,%
\mathbf{y}),\mathbf{u}\right) \,d{\mathcal{W}}_{s},
\end{eqnarray*}%
by the inequality $a\leqslant a^{2}+\frac{1}{4}$ and the assumption (\ref{GG}%
). Therefore 
\begin{equation}
\left\Vert \mathbf{u}(t)\right\Vert _{2}^{4}+2\nu \int_{0}^{t}\left\Vert 
\mathbf{u}\right\Vert _{2}^{2}\left\Vert \mathbf{u}\right\Vert
_{V}^{2}\,ds\leqslant \kappa \int_{0}^{t}\left\Vert \mathbf{u}\right\Vert
_{2}^{4}\,ds+(C_{\kappa }+g(t)),  \label{qq00}
\end{equation}%
where%
\begin{eqnarray*}
\kappa &=&4\left( r_{\ast }+L\right) ,\qquad C_{\kappa }=\left\Vert \mathbf{u%
}_{0}\right\Vert _{2}^{4}+(r_{\ast }+1)T, \\
g(t) &=&\int_{0}^{t}f(s)\,d{\mathcal{W}}_{s},\qquad f(s)=4\left\Vert \mathbf{%
u}\right\Vert _{2}^{2}\left( \mathbf{G}(s,\mathbf{y}),\mathbf{u}\right) .
\end{eqnarray*}%
The expression (\ref{qq00}) can be written as the differential inequality (%
\ref{dz})\ for 
\begin{equation*}
z(t)=\int_{0}^{t}\left\Vert \mathbf{u}\right\Vert _{2}^{4}\,ds.
\end{equation*}%
Hence,  comparing with (\ref{dzz}), { Gr\"onwall's } inequality gives%
\begin{equation}
\left\Vert \mathbf{u}(t)\right\Vert _{2}^{4}+2\nu \int_{0}^{t}\left\Vert 
\mathbf{u}\right\Vert _{2}^{2}\left\Vert \mathbf{u}\right\Vert
_{V}^{2}\,\,ds\leqslant C_{\kappa }e^{\kappa t}+\int_{0}^{t}e^{\kappa
(t-s)}\ f(s)\,d{\mathcal{W}}_{s}.  \label{2100}
\end{equation}%
Let us denote by $F(s)=e^{-\kappa s}f(s)$ and multiply (\ref{2100}) by $%
\beta e^{-\kappa t}$ with $\beta >0,$ we obtain%
\begin{eqnarray}
\beta e^{-\kappa t}\left\Vert \mathbf{u}(t)\right\Vert _{2}^{4} &+&2\nu
\beta e^{-\kappa t}\int_{0}^{t}\left\Vert \mathbf{u}\right\Vert
_{2}^{2}\left\Vert \mathbf{u}\right\Vert _{V}^{2}\,ds\leqslant \beta
C_{\kappa }+\frac{\beta ^{2}}{2}\int_{0}^{t}|F(s)|^{2}\,ds  \notag \\
&&+\left( \beta \int_{0}^{t}F(s)\,d{\mathcal{W}}_{s}-\frac{\beta ^{2}}{2}%
\int_{0}^{t}|F(s)|^{2}\,ds\right) .  \label{rrr}
\end{eqnarray}%
The boundedness (\ref{GG}) for $\mathbf{G}$ implies%
\begin{eqnarray*}
|F(t)|^{2} &\leqslant &16\left\Vert \mathbf{u}\right\Vert _{2}^{4}\left( 
\mathbf{G}(t,\mathbf{y}),\mathbf{y}-\mathbf{a}\right) ^{2}\leqslant 16\left\Vert 
\mathbf{u}(t)\right\Vert _{2}^{4}\left(\frac{L\left\Vert \mathbf{y}(t)\right\Vert
_{2}^{2}}{1+||\mathbf{y}||_{2}^{2}}+C\right) \\
&\leqslant &D_{\ast }\left\Vert \mathbf{u}(t)\right\Vert _{2}^{2}\left\Vert 
\mathbf{u}(t)\right\Vert _{V}^{2}\quad \text{with }D_{\ast }=16L+C,\quad \text{%
for a.e. }t\in (0,T).
\end{eqnarray*}%
On the other hand for $\beta _{\ast }(t)=\frac{2\nu e^{-\kappa t}}{D_{\ast }}
$, we have 
\begin{equation*}
\max_{\beta }(2\nu \beta e^{-\kappa t}-\frac{D_{\ast }\beta ^{2}}{2})=\frac{%
D_{\ast }\beta _{\ast }^{2}}{2}\geqslant \frac{4\nu ^{2}e^{-2\kappa T}}{%
2D_{\ast }}=B_{\ast },\quad \forall t\in \lbrack 0,T].
\end{equation*}%
Due to these relations, \eqref{rrr} yields 
\begin{eqnarray*}
\beta _{\ast }e^{-\kappa t}\left\Vert \mathbf{u}(t)\right\Vert _{2}^{4}
&+&B_{\ast }\int_{0}^{t}\left\Vert \mathbf{u}\right\Vert _{2}^{2}\left\Vert 
\mathbf{u}\right\Vert _{V}^{2}\,ds \\
&\leqslant &\beta _{\ast }C_{\kappa }+\left( \int_{0}^{t}\left( \beta _{\ast
}F\right) \,d{\mathcal{W}}_{s}-\frac{1}{2}\int_{0}^{t}\left( \beta _{\ast
}F\right) ^{2}\,ds\right) .
\end{eqnarray*}%
Taking the exponential function and the expectation in the last deduced
inequality, we obtain 
\begin{align*}
& \mathbb{E}\,\exp \left( \beta _{\ast }e^{-\kappa t}\left\Vert \mathbf{u}%
(t)\right\Vert _{2}^{4}+B_{\ast }\int_{0}^{t}\left\Vert \mathbf{u}%
\right\Vert _{2}^{2}\left\Vert \mathbf{u}\right\Vert _{V}^{2}\,ds\right) \\
& \leqslant H_{2}\times \mathbb{E}\exp \left( \int_{0}^{t}\left( \beta
_{\ast }F\right) \,d{\mathcal{W}}_{s}-\frac{1}{2}\int_{0}^{t}\left( \beta
_{\ast }F\right) ^{2}\,\,ds\right) =H_{2}
\end{align*}%
with $H_{2}=\exp \left( \beta _{\ast }C_{\kappa }\right) <\infty $ by the
assumption (\ref{bound0}). Also, using the inequality (\ref{CCC}), we
conclude that 
\begin{equation*}
\mathbb{E}\,\exp \left( \frac{B_{\ast }}{\widehat{C}}\int_{0}^{t}\left\Vert 
\mathbf{u}(t)\right\Vert _{2}^{4}\,ds\right) \leqslant \mathbb{E}\,\exp
\left( B_{\ast }\int_{0}^{t}\left\Vert \mathbf{u}\right\Vert
_{2}^{2}\left\Vert \mathbf{u}\right\Vert _{V}^{2}ds\right) \leqslant H_{2}.
\end{equation*}%
The last two inequalities are the estimates (\ref{yy})$_{3,4,5.}$
\end{proof}

\bigskip

\section{Linearized state equation}

\label{sec6}\setcounter{equation}{0}

In this section we also assume that $\mathbf{G}(t,\mathbf{y})$ is 
G\^{a}teaux differentiable in the variable $\mathbf{y}\in H:$%
\begin{equation*}
\lim_{s\rightarrow 0}\frac{\mathbf{G}(t,\mathbf{y}+s\mathbf{v})-\mathbf{G}(t,%
\mathbf{y})}{s}=\nabla _{\mathbf{y}}\mathbf{G}(t,\mathbf{y})\mathbf{v}\quad 
\text{for each }t\in \lbrack 0,T],
\end{equation*}%
such that the function$\ \ \nabla _{\mathbf{y}}\mathbf{G}(t,\mathbf{y})$\ is
continuous and bounded in the second variable $\mathbf{y},$ namely 
\begin{equation*}
||\nabla _{\mathbf{y}}\mathbf{G}(t,\mathbf{x})-\nabla _{\mathbf{y}}\mathbf{G}%
(t,\mathbf{y})||_{2}\rightarrow 0\qquad \text{when\ \ }||\mathbf{x}-\mathbf{y%
}||_{2}\rightarrow 0,\qquad \forall t\in \lbrack 0,T],
\end{equation*}%
\begin{eqnarray}
||\nabla _{\mathbf{y}}\mathbf{G}(t,\mathbf{y})\mathbf{v}||_{2} &\leqslant
&C||\mathbf{v}||_{2},\qquad \mathbf{v}\in H,  \notag \\
||\nabla _{\mathbf{y}}\mathbf{G}(t,\mathbf{y})\mathbf{v}||_{V} &\leqslant
&C||\mathbf{v}||_{V},\qquad \mathbf{v}\in V,  \label{cG}
\end{eqnarray}%
for some positive constant $C.$ Due to Propositions A.2, A.3 of \cite{AH} we
have that $\mathbf{G}(t,\mathbf{y})$ \ is Fr\'{e}chet differentiable in the
second variable $\mathbf{y}:$ 
\begin{equation}  \label{LG}
\mathbf{G}(t,\mathbf{x}+\mathbf{y})-\mathbf{G}(t,\mathbf{y})-\nabla _{%
\mathbf{y}}\mathbf{G}(t,\mathbf{y})\mathbf{x}=o\left( t,\left\Vert \mathbf{x}%
\right\Vert _{2}\right) ,\quad \text{where }\lim_{s\rightarrow 0}\frac{%
o\left( t,s\right) }{s}=0
\end{equation}%
for any $\mathbf{x},\,\mathbf{y}\in H,\quad t\in \lbrack 0,T]$.

\medskip

Let $\mathbf{y}$ be the solution of the state system (\ref{NSy}). The
corresponding linearized system for (\ref{NSy}) can be written as the
following Oseen's type system 
\begin{equation}
\left\{ 
\begin{array}{l}
\begin{array}{l}
d\mathbf{z}=(\nu \Delta \mathbf{z}-\left( \mathbf{z\cdot }\nabla \right) 
\mathbf{y}-\left( \mathbf{y\cdot }\nabla \right) \mathbf{z}-\nabla \pi
)\,dt+\nabla _{\mathbf{y}}\mathbf{G}(t,\mathbf{y})\mathbf{z}\,d{\mathcal{W}}%
_{t}\quad \text{in}\ {\mathcal{O}}_{T},\vspace{2mm} \\ 
\mathrm{div}\ \mathbf{z}=0,%
\end{array}
\\ 
\mathbf{z}\cdot \mathbf{n}=f,\mathbf{\qquad }\left[ 2D(\mathbf{z})\,\mathbf{n%
}+\alpha \mathbf{z}\right] \cdot \bm{\tau }=g\mathbf{\qquad \qquad \qquad
\qquad }\text{on}\ \Gamma _{T},\vspace{2mm} \\ 
\mathbf{z}(0)=0\mathbf{\qquad \qquad \qquad \qquad \qquad \qquad \qquad
\qquad \qquad \qquad }\text{in}\ {\mathcal{O}}%
\end{array}%
\right.  \label{linearized}
\end{equation}%
with the boundary data, satisfying the assumption 
\begin{equation}
(f,g)\in\mathcal{A}  \label{regf}
\end{equation}
\medskip

Let $\mathbf{f}$ be the solution of the system (\ref{ha}) with the data ${(}a%
{,b)}$ replaced by $(f,g)$. Then the function $\mathbf{f}$ \ satisfies the
estimates (\ref{cal}), namely 
\begin{equation*}
||\mathbf{f}||_{W_{p}^{1}({\mathcal{O}})}+||\partial _{t}\mathbf{f}||_{L_{2}(%
{\mathcal{O}})}\leqslant C||(f,g)||_{\mathcal{H}_{p}(\Gamma )}\quad \text{%
a.e. in }\Omega \times (0,T),
\end{equation*}%
hence $\mathbf{f}$ has the following regularities%
\begin{eqnarray}
\mathbf{f} &\in &L_{4}(\Omega ;C([0,T];L_{2}({\mathcal{O}})))\cap
L_{4}(\Omega \times (0,T);C(\overline{{\mathcal{O}}})\cap H^{1}({\mathcal{O}}%
)),  \notag \\
\qquad \partial _{t}\mathbf{f} &\in &L_{4}(\Omega \times (0,T);L_{2}({%
\mathcal{O}})).  \label{ref2}
\end{eqnarray}

\bigskip

\begin{definition}
\label{def6.1} A stochastic process $\mathbf{z}=\tilde{\mathbf{z}}+\mathbf{f}
$ with $\ \tilde{\mathbf{z}}\in L_{2}(0,T;V),\quad P$-a.e. in $\Omega ,$ \
is a strong solution of (\ref{linearized})  with $\mathbf{z}(0)=0$\ if \ the
following equation holds 
\begin{align*}
\left( \mathbf{z}(t),\boldsymbol{\varphi }\right) & =\int_{0}^{t}\left[ -\nu
\left( \mathbf{z},\boldsymbol{\varphi }\right) _{V}+\int_{\Gamma }\nu g(%
\boldsymbol{\varphi }\cdot {\bm{\tau }})\,d\mathbf{\gamma }-((\mathbf{z}%
\cdot \nabla )\mathbf{y}+((\mathbf{y}\cdot \nabla )\mathbf{z},\boldsymbol{%
\varphi })\,\right] ds \\
& \\
& \quad +\int_{0}^{t}\left( \nabla _{\mathbf{y}}\mathbf{G}(s,\mathbf{y})%
\mathbf{z},\boldsymbol{\varphi }\right) \,d{\mathcal{W}}_{s},\qquad \forall 
\boldsymbol{\varphi }\in V,\quad \text{a.e. in }\Omega \times (0,T).
\end{align*}
\end{definition}

\bigskip

In what follows we will establish the solvability of the system (\ref%
{linearized}).

\begin{proposition}
\label{ex_uniq_lin} {Let $\mathbf{y}$ be the solution of the state system (%
\ref{NSy}) with the boundary data }${(a,b)}\in\mathcal{A}${, \ which was
constructed in Theorem }\ref{the_1}. Let ${(f,g)}$ satisfy the assumption %
\eqref{regf}. Then there exists a unique solution $\mathbf{z}=\tilde{\mathbf{%
z}}+\mathbf{f} $ for the system (\ref{linearized}), such that%
\begin{equation*}
\tilde{\mathbf{z}}\in C([0,T];H)\cap L_{2}\left( 0,T;V\right) ,\quad P\text{%
-a.e. in }\Omega .
\end{equation*}%
Moreover, there exists a positive constant $\widehat{C}_{2}$, such that%
\begin{align}
\mathbb{E}\sup_{s\in \lbrack 0,t]}\xi _{2}^{2}(s)\left\Vert \tilde{\mathbf{z}%
}(s)\right\Vert _{2}^{2}& +\nu \mathbb{E}\int_{0}^{t}\xi _{2}^{2}\left\Vert 
\tilde{\mathbf{z}}\right\Vert _{V}^{2}\,ds  \notag \\
& \leqslant C\mathbb{E}\int_{0}^{t}\xi _{2}^{2}\left( 1+\nu \right)
||(f,g)||_{\mathcal{H}_{p}(\Gamma )}^{2}\left( 1+\Vert \mathbf{y}\Vert
_{2}^{2}\right) \,ds  \label{eqsa}
\end{align}%
\begin{align}
\mathbb{E}\sup_{s\in \lbrack 0,t]}\xi _{2}^{4}(s)\left\Vert \tilde{\mathbf{z}%
}(s)\right\Vert _{2}^{4}& +\nu ^{2}\mathbb{E}\left( \int_{0}^{t}\xi
_{2}^{2}\left\Vert \tilde{\mathbf{z}}\right\Vert _{V}^{2}\,ds\right) ^{2} 
\notag \\
& \leqslant C\left( \mathbb{E}\int_{0}^{t}\xi _{2}^{2}\left( 1+\nu \right)
||(f,g)||_{\mathcal{H}_{p}(\Gamma )}^{2}\left( 1+\Vert \mathbf{y}\Vert
_{2}^{2}\right) \,ds\right) ^{2}  \label{eqsa4}
\end{align}%
with 
\begin{equation}
\xi _{2}(t)=e^{-\int_{0}^{t}f_{2}(s)ds}  \label{xi2}
\end{equation}%
and 
\begin{equation*}
f_{2}(t)=\widehat{C}_{2}(\nu ^{-1}+1)\left( 1+\Vert \mathbf{u}\Vert
_{V}^{2}+||(a,b)||_{\mathcal{H}_{p}(\Gamma )}^{2}\right) ,
\end{equation*}%
where $\widehat{C}_{2}$ is defined according to the relations \eqref{f_2}-%
\eqref{f_22}.
\end{proposition}

\begin{proof}
The existence of a solution for the system (\ref{linearized})\textbf{\ }will
be shown by Galerkin's method.

The injection operator $I:V\rightarrow H$ is a compact operator. Therefore
there exists a basis $\left\{ \mathbf{e}_{k}\right\} _{k=1}^{\infty }\subset
V$ of eigenfunctions 
\begin{equation}
\left( \mathbf{v},\mathbf{e}_{k}\right) _{V}=\lambda _{k}\left( \mathbf{v},%
\mathbf{e}_{k}\right) ,\qquad \forall \mathbf{v}\in V,\;k\in \mathbb{N},
\label{y3}
\end{equation}%
which is an orthonormal basis for $H,$ such that the sequence $\{\lambda
_{k}\}_{k=1}^{\infty }$ of eigenvalues verifies the following properties%
\begin{equation*}
\lambda _{k}>0,\quad \forall k\in \mathbb{N}\quad \text{and}\quad \lambda
_{k}\rightarrow \infty \quad \text{as }k\rightarrow \infty .
\end{equation*}
To justify it, we refer to a similar situation, which was considered in
Lemma 2.2. of \cite{clop}, Theorem 1 of \cite{S73}, see also \cite{evans},
p. 297-307: Theorem 2, p. 300 and Theorem 5, p. 305.

The ellipticity of the equation (\ref{y3}) and the regularity $\Gamma \in
C^{2}$ imply that $\{\mathbf{e}_{k}\}\subset H^{2}({\mathcal{O}})\cap V$,
hence the sequence $\left\{ \mathbf{e}_{k}\right\} _{k=1}^{\infty }$ forms
the eigenfunctions of the Stokes problem:%
\begin{equation*}
\left\{ 
\begin{array}{ll}
-\Delta \mathbf{e}_{k}+\nabla \pi _{k}=\lambda _{k}\mathbf{e}_{k},\mathbf{%
\qquad }\mbox{div}\,\mathbf{e}_{k}=0, & \mbox{in}{\ \mathcal{O}},\vspace{2mm}
\\ 
\mathbf{e}_{k}\cdot \mathbf{n}=0,\;\quad \left[ 2D(\mathbf{e}_{k})\,\mathbf{n%
}+\alpha \mathbf{e}_{k}\right] \cdot {\bm{\tau }}=0\;\quad & \text{on}\
\Gamma .%
\end{array}%
\right.
\end{equation*}

Let us fix an arbitrary $n\in \mathbb{N}$ and consider the finite
dimensional subspace $V_{n}=\mathrm{span}\,\{\mathbf{e}_{1},\ldots ,\mathbf{e%
}_{n}\}$ of $V.$\ Let 
\begin{equation*}
\mathbf{z}_{n}=\tilde{\mathbf{z}}_{n}+\mathbf{f\qquad }\text{with\quad }%
\tilde{\mathbf{z}}_{n}(t)=\sum_{k=1}^{n}r_{k}^{(n)}(t)\ \mathbf{e}_{k}\tilde{%
\mathbf{z}}_{n}\in V_{n}
\end{equation*}%
\ be a solution of the following finite dimensional problem:%
\begin{equation}
\left\{ 
\begin{array}{l}
d(\mathbf{z}_{n},\mathbf{e}_{k})=\left[ -\nu \left( \mathbf{z}_{n},\mathbf{e}%
_{k}\right) _{V}\,+\nu \int_{\Gamma }g(\mathbf{e}_{k}\cdot {\bm{\tau }})\,d%
\mathbf{\gamma }-((\mathbf{z}_{n}\cdot \nabla )\mathbf{y}+(\left( \mathbf{y}%
\cdot \nabla )\mathbf{z}_{n},\mathbf{e}_{k}\right) \,\right] dt \\ 
\qquad \\ 
\qquad \qquad \quad +\left( \nabla _{\mathbf{y}}\mathbf{G}(t,\mathbf{y})%
\mathbf{z}_{n},\mathbf{e}_{k}\right) \,d{\mathcal{W}}_{t},\qquad \text{a.e.
in }\Omega \times (0,T), \\ 
\\ 
\tilde{\mathbf{z}}_{n}(0)=\tilde{\mathbf{z}}_{n,0},\text{\qquad }k=1,2,\dots
,n,%
\end{array}%
\right.  \label{z}
\end{equation}%
where $\tilde{\mathbf{z}}_{n,0}$ is the orthogonal projections in $H$ of $%
\tilde{\mathbf{z}}_{0}(\mathbf{x})=-\mathbf{f}(0,\mathbf{x})$ \ onto the
space $V_{n}.$

The problem \ (\ref{z}) is a system of $n$-stochastic linear ordinary
differential equations, which has a unique global-in-time solution $\tilde{%
\mathbf{z}}_{n}=\mathbf{z}_{n}-\mathbf{f},$ as an adapted process in the
space $C([0,{T}];V_{n})$. We can write the equation in (\ref{z}) as%
\begin{align*}
d\left( \tilde{\mathbf{z}}_{n},\mathbf{e}_{i}\right) & =[-\nu \left( \tilde{%
\mathbf{z}}_{n}+\mathbf{f},\mathbf{e}_{i}\right) _{V}+\nu \int_{\Gamma }g(%
\mathbf{e}_{i}\cdot {\bm{\tau }})\,d\mathbf{\gamma }\, \\
& \quad +\left( -\mathbf{\partial }_{t}\mathbf{f}-\left( \left( \tilde{%
\mathbf{z}}_{n}+\mathbf{f}\right) \mathbf{\cdot }\nabla \right) \mathbf{y}%
-\left( \mathbf{y\cdot }\nabla \right) \left( \tilde{\mathbf{z}}_{n}+\mathbf{%
f}\right) ,\mathbf{e}_{i}\right) ]dt \\
& \quad +\left( \nabla _{\mathbf{y}}\mathbf{G}(t,\mathbf{y})\left( \tilde{%
\mathbf{z}}_{n}+\mathbf{f}\right) ,\mathbf{e}_{i}\right) \,d{\mathcal{W}}%
_{t}.
\end{align*}

\textit{Step 1. Deducing of the estimate (\ref{eqsa}). }The It\^{o} formula
gives 
\begin{align*}
d\left( \left( \tilde{\mathbf{z}}_{n}\,,\mathbf{e}_{i}\right) ^{2}\right) &
=2\left( \tilde{\mathbf{z}}_{n},\mathbf{e}_{i}\right) [-\nu \left( \tilde{%
\mathbf{z}}_{n},\mathbf{e}_{i}\right) _{V}-\nu \left( \mathbf{f},\mathbf{e}%
_{i}\right) _{V}+\nu \int_{\Gamma }g(\mathbf{e}_{i}\cdot {\bm{\tau }})\,d%
\mathbf{\gamma }\, \\
& \quad +\left( -\mathbf{\partial }_{t}\mathbf{f}-\left( \left( \tilde{%
\mathbf{z}}_{n}+\mathbf{f}\right) \mathbf{\cdot }\nabla \right) \mathbf{y}%
-\left( \mathbf{y\cdot }\nabla \right) \left( \tilde{\mathbf{z}}_{n}+\mathbf{%
f}\right) ,\mathbf{e}_{i}\right) ]\,dt \\
& \quad +2\left( \tilde{\mathbf{z}}_{n},\mathbf{e}_{i}\right) \left( \nabla
_{\mathbf{y}}\mathbf{G}(t,\mathbf{y})\left( \tilde{\mathbf{z}}_{n}+\mathbf{f}%
\right) ,\mathbf{e}_{i}\right) \,d{\mathcal{W}}_{t} \\
& \quad +|\left( \nabla _{\mathbf{y}}\mathbf{G}(t,\mathbf{y})\left( \tilde{%
\mathbf{z}}_{n}+\mathbf{f}\right) ,\mathbf{e}_{i}\right) |^{2}\,dt,
\end{align*}%
where the module in the last term is defined by \eqref{product}. Summing
these equalities over $i=1,\dots ,n,$ we obtain%
\begin{align}
d\left( \left\Vert \tilde{\mathbf{z}}_{n}\right\Vert _{2}^{2}\right) +2\nu
\left\Vert \tilde{\mathbf{z}}_{n}\right\Vert _{V}^{2}& =\left[ \int_{\Gamma
}\left\{ -a(\tilde{\mathbf{z}}_{n}\cdot \bm{\tau })^{2}+2\nu g(\tilde{%
\mathbf{z}}_{n}\cdot \bm{\tau })\right\} \,d\mathbf{\ \gamma }\,\right] \,dt
\notag \\
& \quad -2\left( \partial _{t}\mathbf{f}+\left( \left( \tilde{\mathbf{z}}%
_{n}+\mathbf{f}\right) \mathbf{\cdot }\nabla \right) \mathbf{y},\tilde{%
\mathbf{z}}_{n}\right) \,dt  \notag \\
& \quad -2\left( \left( \mathbf{y}\cdot \nabla \right) \mathbf{f},\tilde{%
\mathbf{z}}_{n}\right) dt-2\nu \left( \mathbf{f},\tilde{\mathbf{z}}%
_{n}\right) _{V}\,dt  \notag \\
& \quad +\sum_{i=1}^{n}|\left( \nabla _{\mathbf{y}}\mathbf{G}(t,\mathbf{y}%
)\left( \tilde{\mathbf{z}}_{n}+\mathbf{f}\right) ,\mathbf{e}_{i}\right)
|^{2}\,dt  \notag \\
& \quad +2\left( \nabla _{\mathbf{y}}\mathbf{G}(t,\mathbf{y})\left( \tilde{%
\mathbf{z}}_{n}+\mathbf{f}\right) ,\tilde{\mathbf{z}}_{n}\right) \,d{%
\mathcal{W}}_{t}  \notag \\
& =J\,dt+J_{6}\,d{\mathcal{W}}_{t},  \label{JJ}
\end{align}%
with $J=J_{1}+J_{2}+J_{3}+J_{4}+J_{5}$.

Let us estimate the terms $J_{i},$ $i=1,..5$. Firstly we have%
\begin{align*}
J_{1}& \leqslant (\Vert a\Vert _{L_{\infty }(\Gamma )}\,+1)\Vert \tilde{%
\mathbf{z}}_{n}\Vert _{L_{2}(\Gamma )}^{2}+C\nu \Vert g\Vert _{L_{2}(\Gamma
)}^{2}\, \\
& \leqslant C(\Vert a\Vert _{W_{p}^{1-\frac{1}{p}}(\Gamma )}+1)||\tilde{%
\mathbf{z}}_{n}||_{2}||\nabla \tilde{\mathbf{z}}_{n}||_{2}+C\nu \Vert g\Vert
_{L_{2}(\Gamma )}^{2}\, \\
& \leqslant h_{1}(t)||\tilde{\mathbf{z}}_{n}||_{2}^{2}+\frac{\nu }{4}||%
\tilde{\mathbf{z}}_{n}||_{V}^{2}+C\nu \Vert g\Vert _{L_{2}(\Gamma )}^{2}\,
\end{align*}%
with 
\begin{equation*}
h_{1}(t)=C\nu ^{-1}(\Vert a\Vert _{W_{p}^{1-\frac{1}{p}}(\Gamma )}^{2}+1)\in
L_{1}(0,T),\quad P\text{-a.e. in }\Omega ,
\end{equation*}%
by the assumption \eqref{eq00sec12} and%
\begin{align*}
J_{2}& \leqslant C\left( \Vert \partial _{t}\mathbf{f}\Vert _{2}+\,\Vert 
\mathbf{f}\Vert _{C(\overline{{\mathcal{O}}})}\Vert \nabla \mathbf{y}\Vert
_{2}\right) \Vert \tilde{\mathbf{z}}_{n}\Vert _{2}+\,\Vert \nabla \mathbf{y}%
\Vert _{2}\,\Vert \tilde{\mathbf{z}}_{n}\Vert _{4}^{2} \\
& \leqslant C\left( \Vert \partial _{t}\mathbf{f}\Vert _{2}+\,\Vert \mathbf{f%
}\Vert _{C(\overline{{\mathcal{O}}})}\right) \left( 1+\Vert \nabla \mathbf{y}%
\Vert _{2}\right) \Vert \tilde{\mathbf{z}}_{n}\Vert _{2} \\
& \quad +\Vert \nabla \mathbf{y}\Vert _{2}\Vert \tilde{\mathbf{z}}_{n}\Vert
_{2}\Vert \nabla \tilde{\mathbf{z}}_{n}\Vert _{2} \\
& \leqslant h_{2}(t)\Vert \tilde{\mathbf{z}}_{n}\Vert _{2}^{2}+\frac{\nu }{4}%
||\tilde{\mathbf{z}}_{n}||_{V}^{2}+\Vert (f,g)\Vert _{\mathcal{H}_{p}(\Gamma
)}^{2}
\end{align*}%
with 
\begin{equation*}
h_{2}(t)=C\max (\nu ^{-1},1)\left( 1+\Vert \nabla \mathbf{y}\Vert
_{2}^{2}\right) \in L_{1}(0,T),\quad P\text{-a.e. in }\Omega ,
\end{equation*}%
by the estimates \eqref{uny}.

Reasoning as above and using \eqref{LI} for $q=4,$ we have 
\begin{equation*}
\Vert \mathbf{y}\Vert _{4}\leqslant C\left( ||\mathbf{y}||_{2}^{1/2}||\nabla 
\mathbf{y}||_{2}^{1/2}+||\mathbf{y}||_{2}\right) ,
\end{equation*}%
that implies%
\begin{eqnarray*}
J_{3} &\leqslant &\,C\Vert \mathbf{y}\Vert _{4}||\nabla \mathbf{f}\Vert
_{2}\Vert \tilde{\mathbf{z}}_{n}\Vert _{4}\leqslant C||\nabla \mathbf{f}%
\Vert _{2}\Vert \mathbf{y}\Vert _{4}||\tilde{\mathbf{z}}_{n}||_{2}^{1/2}||%
\nabla \tilde{\mathbf{z}}_{n}||_{2}^{1/2} \\
&\leqslant &C||\nabla \mathbf{f}\Vert _{2}^{2}\Vert \mathbf{y}\Vert _{2}+%
\frac{C}{\nu }(\Vert \mathbf{y}\Vert _{2}^{2}+\Vert \nabla \mathbf{y}\Vert
_{2}^{2})+\frac{\nu }{4}||\tilde{\mathbf{z}}_{n}||_{V}^{2}.
\end{eqnarray*}%
Hence 
\begin{equation*}
J_{3}\leqslant C\Vert (f,g)\Vert _{\mathcal{H}_{p}(\Gamma )}^{2}\Vert 
\mathbf{y}\Vert _{2}+h_{3}(t)\Vert \tilde{\mathbf{z}}_{n}\Vert _{2}^{2}+%
\frac{\nu }{4}||\tilde{\mathbf{z}}_{n}||_{V}^{2}
\end{equation*}%
with%
\begin{equation*}
h_{3}(t)=\frac{C}{\nu }\left( ||\mathbf{y}||_{2}^{2}+\Vert \nabla \mathbf{y}%
\Vert _{2}^{2}\right) \in L_{1}(0,T),\quad P\text{-a.e. in }\Omega ,
\end{equation*}%
by the estimate \eqref{uny}$_{1}$.

The terms $J_{4}$ and $J_{5}$ are estimated \ as%
\begin{eqnarray*}
J_{4} &\leqslant &C\nu \,\Vert \mathbf{f}\Vert _{V}^{2}\,+||\tilde{\mathbf{z}%
}_{n}||_{V}^{2}+\frac{\nu }{4}||\tilde{\mathbf{z}}_{n}||_{V}^{2} \\
&\leqslant &C\nu \,\,\Vert (f,g)\Vert _{\mathcal{H}_{p}(\Gamma )}^{2}+\frac{%
\nu }{4}||\tilde{\mathbf{z}}_{n}||_{V}^{2}
\end{eqnarray*}%
and 
\begin{eqnarray*}
J_{5} &=&\sum_{i=1}^{n}|\left( \nabla _{\mathbf{y}}\mathbf{G}(t,\mathbf{y}%
)\left( \tilde{\mathbf{z}}_{n}+\mathbf{f}\right) ,\mathbf{e}_{i}\right)
|^{2}\leqslant C||\nabla _{\mathbf{y}}\mathbf{G}(t,\mathbf{y})\left( \tilde{%
\mathbf{z}}_{n}+\mathbf{f}\right) ||_{2}^{2} \\
&\leqslant &C||\tilde{\mathbf{z}}_{n}+\mathbf{f}||_{2}^{2}\leqslant C||%
\tilde{\mathbf{z}}_{n}||_{2}^{2}+C\Vert (f,g)\Vert _{\mathcal{H}_{p}(\Gamma
)}^{2}
\end{eqnarray*}%
by the assumption \eqref{cG}.

The above deduced estimates for the terms $J_{i},$ $i=1,...,5$\ and %
\eqref{JJ} imply the existence of some positive constant $\widehat{C}_{2}$
such that 
\begin{equation}
J\ \leqslant 2f_{2}(t)\Vert \tilde{\mathbf{z}}_{n}\Vert _{2}^{2}+C\left[
\left( 1+\nu \right) ||(f,g)||_{\mathcal{H}_{p}(\Gamma )}^{2}\left( 1+\Vert 
\mathbf{y}\Vert _{2}\right) \right] +\nu ||\tilde{\mathbf{z}}_{n}||_{V}^{2}
\label{f_2}
\end{equation}%
with 
\begin{equation}
f_{2}(t)=\widehat{C}_{2}(\nu ^{-1}+1)\left( 1+\Vert \mathbf{u}\Vert
_{V}^{2}+||(a,b)||_{\mathcal{H}_{p}(\Gamma )}^{2}\right) \in L_{1}(0,T).
\label{f_22}
\end{equation}

Introducing the function 
\begin{equation*}
\xi _{2}(t)=e^{-\int_{0}^{t}f_{2}(s)ds}\qquad \text{for }t\in \lbrack 0,T],
\end{equation*}%
the It\^{o} formula gives 
\begin{eqnarray}
\xi _{2}^{2}(t)\left\Vert \tilde{\mathbf{z}}_{n}(t)\right\Vert _{2}^{2}
&+&\nu \int_{0}^{t}\xi _{2}^{2}\left\Vert \tilde{\mathbf{z}}_{n}\right\Vert
_{V}^{2}\,ds  \notag \\
&\leqslant &C\int_{0}^{t}\xi _{2}^{2}\left[ \left( 1+\nu \right) ||(f,g)||_{%
\mathcal{H}_{p}(\Gamma )}^{2}\left( 1+\Vert \mathbf{y}\Vert _{2}^{2}\right) %
\right] \,ds  \notag \\
&&+2\int_{0}^{t}\xi _{2}^{2}(t)\left( \nabla _{\mathbf{y}}\mathbf{G}(t,%
\mathbf{y})\left( \tilde{\mathbf{z}}_{n}+\mathbf{f}\right) ,\tilde{\mathbf{z}%
}_{n}\right) \,d{\mathcal{W}}_{t}.  \label{zzz}
\end{eqnarray}

For each $n\in \mathbb{N}$, let us define the function 
\begin{equation*}
d(t)=\xi _{2}^{2}(t)\left\Vert \tilde{\mathbf{z}}_{n}(t)\right\Vert
_{2}^{2}+\nu \int_{0}^{t}\xi _{2}^{2}\left\Vert \tilde{\mathbf{z}}%
_{n}\right\Vert _{V}^{2}\,ds,\quad \text{a.e. in }\Omega \times (0,T),
\end{equation*}%
{and consider the sequence $\{\tau _{N}^{n}\}$}$_{{N\in \mathbb{N}}}$ of
the stopping times 
\begin{equation}
\tau _{N}^{n}(\omega )=\inf \{t\geqslant 0:d(t)\geqslant N\}\wedge T,\qquad P%
\text{ -a.e. }\omega \in \Omega .  \label{g}
\end{equation}%
For $0\leqslant s\leqslant \tau _{N}^{n}\wedge t$, we have 
\begin{align}
\xi _{2}^{2}(s)\left\Vert \tilde{\mathbf{z}}_{n}(s)\right\Vert _{2}^{2}&
+\nu \int_{0}^{s}\xi _{2}^{2}(r)\left\Vert \tilde{\mathbf{z}}_{n}\right\Vert
_{V}^{2}\ dr  \notag \\
& \leqslant C\int_{0}^{s}\xi _{2}^{2}\left[ \left( 1+\nu \right) ||(f,g)||_{%
\mathcal{H}_{p}(\Gamma )}^{2}\left( 1+\Vert \mathbf{y}\Vert _{2}^{2}\right) %
\right] \,dr  \notag \\
& +2\int_{0}^{s}\xi _{2}^{2}(r)\left( \nabla _{\mathbf{y}}\mathbf{G}(r,%
\mathbf{y})\left( \tilde{\mathbf{z}}_{n}+\mathbf{f}\right) ,\tilde{\mathbf{z}%
}_{n}\right) \,d{\mathcal{W}}_{r}.  \label{zz}
\end{align}%
The Burkholder-Davis-Gundy inequality gives%
\begin{align*}
\mathbb{E}\sup_{s\in \lbrack 0,\tau _{N}^{n}\wedge t]}& \left\vert
\int_{0}^{s}\xi _{2}^{2}(r)\left( \nabla _{\mathbf{y}}\mathbf{G}(t,\mathbf{y}%
)\left( \tilde{\mathbf{z}}_{n}+\mathbf{f}\right) ,\tilde{\mathbf{z}}%
_{n}\right) \,d{\mathcal{W}}_{r}\right\vert \\
& \leqslant \mathbb{E}\left( \int_{0}^{\tau _{N}^{n}\wedge t}\xi
_{2}^{4}(s)\left\vert \left( \nabla _{\mathbf{y}}\mathbf{G}(s,\mathbf{y}%
)\left( \tilde{\mathbf{z}}_{n}+\mathbf{f}\right) ,\tilde{\mathbf{z}}%
_{n}\right) \right\vert ^{2}\,ds\right) ^{\frac{1}{2}} \\
& \leqslant \mathbb{E}\sup_{s\in \lbrack 0,\tau _{N}^{n}\wedge t]}\xi
_{2}^{2}\left\Vert \tilde{\mathbf{z}}_{n}\right\Vert _{2}\left(
\int_{0}^{\tau _{N}^{n}\wedge t}\xi _{2}^{2}\left\Vert \tilde{\mathbf{z}}%
_{n}+\mathbf{f}\right\Vert _{2}^{2}\,ds\right) ^{\frac{1}{2}} \\
& \leqslant \frac{1}{2}\,\mathbb{E}\sup_{s\in \lbrack 0,\tau _{N}^{n}\wedge
t]}\xi _{2}^{2}(s)\left\Vert \tilde{\mathbf{z}}_{n}\right\Vert _{2}^{2}+C%
\mathbb{E}\int_{0}^{\tau _{N}^{n}\wedge t}\xi _{2}^{2}(1+||\tilde{\mathbf{z}}%
_{n}||_{2}^{2}+||f||_{W_{2}^{1-\frac{1}{2}}(\Gamma )}^{2})\,ds
\end{align*}%
by the assumption \eqref{cG}. Substituting this inequality in (\ref{zz}), we
derive 
\begin{align*}
\frac{1}{2}\mathbb{E}\sup_{s\in \lbrack 0,\tau _{N}^{n}\wedge t]}\xi
_{2}^{2}(s)\Vert \tilde{\mathbf{z}}_{n}(s)\Vert _{2}^{2}& +\nu \mathbb{E}%
\int_{0}^{\tau _{N}^{n}\wedge t}\xi _{2}^{2}\left\Vert \tilde{\mathbf{z}}%
_{n}\right\Vert _{V}^{2}\,ds\leqslant C\mathbb{E}\int_{0}^{\tau
_{N}^{n}\wedge t}\xi _{2}^{2}\left\Vert \tilde{\mathbf{z}}_{n}\right\Vert
_{2}^{2}\,ds \\
& +C\mathbb{E}\int_{0}^{\tau _{N}^{n}\wedge t}\xi _{2}^{2}\left[ \left(
1+\nu \right) ||(f,g)||_{\mathcal{H}_{p}(\Gamma )}^{2}\left( 1+\Vert \mathbf{%
y}\Vert _{2}^{2}\right) \right] \,ds.
\end{align*}%
Let $1_{[0,\tau _{N}^{n}]}$ be the characteristic function of the interval $%
[0,\tau _{N}^{n}].$ Then the function 
\begin{equation*}
f(t)=\mathbb{E}\sup_{s\in \lbrack 0,t]}1_{[0,\tau _{N}^{n}]}\xi
_{2}^{2}(s)\Vert \tilde{\mathbf{z}}_{n}(s)\Vert _{2}^{2}
\end{equation*}%
fulfills the { Gr\"onwall }  type inequality
\begin{equation*}
\frac{1}{2}f(t)\leqslant C\int_{0}^{t}f(s)ds+C\mathbb{E}\int_{0}^{t}\xi
_{2}^{2}\left[ \left( 1+\nu \right) ||(f,g)||_{\mathcal{H}_{p}(\Gamma
)}^{2}\left( 1+\Vert \mathbf{y}\Vert _{2}^{2}\right) \right] \,dr,
\end{equation*}%
which implies that%
\begin{eqnarray}
\mathbb{E}\sup_{s\in \lbrack 0,\tau _{N}^{n}\wedge t]}\xi
_{2}^{2}(s)\left\Vert \tilde{\mathbf{z}}_{n}(s)\right\Vert _{2}^{2} &+&\nu 
\mathbb{E}\int_{0}^{\tau _{N}^{n}\wedge t}\xi _{2}^{2}\left\Vert \tilde{%
\mathbf{z}}_{n}\right\Vert _{V}^{2}\,ds  \label{IMP_1} \\
&\leqslant &C\mathbb{E}\int_{0}^{t}\xi _{2}^{2}\left[ \left( 1+\nu \right)
||(f,g)||_{\mathcal{H}_{p}(\Gamma )}^{2}\left( 1+\Vert \mathbf{y}\Vert
_{2}^{2}\right) \right] \,ds.  \notag
\end{eqnarray}

Now let us justify the limit transition as $N\rightarrow \infty $\ \ in the
estimate (\ref{IMP_1}).\medskip\ By (\ref{eq00sec12}), (\ref{uny})$_{1}$, (%
\ref{regf}) and (\ref{IMP_1}) we have 
\begin{equation*}
\mathbb{E}\sup_{s\in \lbrack 0,\tau _{N}^{n}\wedge T]}d(s)\leqslant C
\end{equation*}%
for some constant $C$ being independent of $N$ and $n$. Let us fix $n\in 
\mathbb{N}$. Since $\tilde{\mathbf{z}}_{n}\in C([0,T];V_{n}),$ then $\
d(\tau _{N}^{n})\geqslant N$ and%
\begin{align*}
\mathbb{E}\sup_{s\in \lbrack 0,\tau _{N}^{n}\wedge T]}d(s)& \geqslant 
\mathbb{E}\left( \sup_{s\in \lbrack 0,\tau _{N}^{n}\wedge T]}1_{\{\tau
_{N}^{n}<T\}}d(s)\right) \\
& =\mathbb{E}\left( 1_{\{\tau _{N}^{n}<T\}}\ d(\tau _{N}^{n})\right)
\geqslant NP\left( \tau _{N}^{n}<T\right) .
\end{align*}%
Hence $P\left( \tau _{N}^{n}<T\right) \rightarrow 0$, as $N\rightarrow
\infty ,$ this means that $\tau _{N}^{n}\rightarrow T$ in probability as $%
N\rightarrow \infty $. Therefore, there exists a subsequence $\{\tau
_{N_{k}}^{n}\}$ of $\{\tau _{N}^{n}\}$ (which may depend on $n$), such that 
\begin{equation*}
\tau _{N_{k}}^{n}(\omega )\rightarrow T\text{\qquad for a.e. \ }\omega \in
\Omega ,\quad \text{as }k\rightarrow \infty .
\end{equation*}%
So $\mathbf{z}_{n}=\tilde{\mathbf{z}}_{n}+\mathbf{f}$ is a global-in-time
solution of the stochastic differential equation (\ref{z}). In addition, the
sequence $\left\{ \tau _{N}^{n}\right\} $ of the stopping times is monotone
on $N$ for each fixed $n$, so we can apply the monotone convergence theorem
in order to pass to the limit in the inequality (\ref{IMP_1}) as $%
N\rightarrow \infty $, \ thereby deducing the estimate (\ref{eqsa}), which
is valid for $\tilde{\mathbf{z}}_{n}$.

\textit{Step 2. Deducing of the estimate (\ref{eqsa4}).} For each $n\in 
\mathbb{N}$, let us consider the function 
\begin{equation*}
d(t)=\xi _{2}^{4}(t)\left\Vert \tilde{\mathbf{z}}_{n}(t)\right\Vert
_{2}^{4}+\nu ^{2}\left( \int_{0}^{t}\xi _{2}^{2}\left\Vert \tilde{\mathbf{z}}%
_{n}\right\Vert _{V}^{2}\,ds\right) ^{2},\quad \text{a.e. in }\Omega \times
(0,T),
\end{equation*}%
{\ and the sequence $\{\tau _{N}^{n}\}$}$_{{N\in \mathbb{N}}}${\ of the
stopping times 
\begin{equation*}
\tau _{N}^{n}(\omega )=\inf \{t\geqslant 0:d(t)\geqslant N\}\wedge T,\qquad P%
\text{ -a.e. }\omega \in \Omega .
\end{equation*}%
}Taking the square of \eqref{zzz}, for $0\leqslant s\leqslant \tau
_{N}^{n}\wedge t$ we infer that 
\begin{align}
& \xi _{2}^{4}(s)\left\Vert \tilde{\mathbf{z}}_{n}(s)\right\Vert
_{2}^{4}+\nu ^{2}\left( \int_{0}^{s}\xi _{2}^{2}(r)\left\Vert \tilde{\mathbf{%
z}}_{n}\right\Vert _{V}^{2}\ dr\right) ^{2}  \notag \\
& \leqslant C\left( \mathbb{E}\int_{0}^{t}\xi _{2}^{2}\left[ \left( 1+\nu
\right) ||(f,g)||_{\mathcal{H}_{p}(\Gamma )}^{2}\left( 1+\Vert \mathbf{y}%
\Vert _{2}^{2}\right) \right] \,ds\right) ^{2}  \notag \\
& \quad +8\left( \int_{0}^{s}\xi _{2}^{2}(r)\left( \nabla _{\mathbf{y}}%
\mathbf{G}(r,\mathbf{y})\left( \tilde{\mathbf{z}}_{n}+\mathbf{f}\right) ,%
\tilde{\mathbf{z}}_{n}\right) \,d{\mathcal{W}}_{r}\right) ^{2}.  \label{zzz0}
\end{align}%
The Burkholder-Davis-Gundy inequality gives%
\begin{align*}
\mathbb{E}\sup_{s\in \lbrack 0,\tau _{N}^{n}\wedge t]}& \left\vert
\int_{0}^{s}\xi _{2}^{2}(r)\left( \nabla _{\mathbf{y}}\mathbf{G}(t,\mathbf{y}%
)\left( \tilde{\mathbf{z}}_{n}+\mathbf{f}\right) ,\tilde{\mathbf{z}}%
_{n}\right) \,d{\mathcal{W}}_{r}\right\vert ^{2} \\
& \leqslant \mathbb{E}\int_{0}^{\tau _{N}^{n}\wedge t}\xi
_{2}^{4}(s)\left\vert \left( \nabla _{\mathbf{y}}\mathbf{G}(s,\mathbf{y}%
)\left( \tilde{\mathbf{z}}_{n}+\mathbf{f}\right) ,\tilde{\mathbf{z}}%
_{n}\right) \right\vert ^{2}\,ds \\
& \leqslant \frac{1}{2}\,\mathbb{E}\sup_{s\in \lbrack 0,\tau _{N}^{n}\wedge
t]}\xi _{2}^{4}(s)\left\Vert \tilde{\mathbf{z}}_{n}\right\Vert _{2}^{4} \\
&\quad +C\mathbb{E}\int_{0}^{\tau _{N}^{n}\wedge t}\xi _{2}^{4}(1+||\tilde{%
\mathbf{z}}_{n}||_{2}^{4}+||f||_{W_{2}^{1-\frac{1}{2}}(\Gamma )}^{4})\,ds
\end{align*}%
by the assumption \eqref{cG}. Using this inequality and (\ref{zzz0}), we
derive 
\begin{align*}
\frac{1}{2}\mathbb{E}&\sup_{s\in \lbrack 0,\tau _{N}^{n}\wedge t]}\xi
_{2}^{4}(s)\Vert \tilde{\mathbf{z}}_{n}(s)\Vert _{2}^{4} +\nu ^{2}\mathbb{E}%
\left( \int_{0}^{\tau _{N}^{n}\wedge t}\xi _{2}^{2}\left\Vert \tilde{\mathbf{%
z}}_{n}\right\Vert _{V}^{2}\,ds\right) ^{2} \\
&\leqslant C\mathbb{E}\int_{0}^{\tau _{N}^{n}\wedge t}\xi _{2}^{4}\left\Vert 
\tilde{\mathbf{z}}_{n}\right\Vert _{2}^{4}\,ds \\
&+C\left( \mathbb{E}\int_{0}^{t}\xi _{2}^{2}\left[ \left( 1+\nu \right)
||(f,g)||_{\mathcal{H}_{p}(\Gamma )}^{2}\left( 1+\Vert \mathbf{y}\Vert
_{2}^{2}\right) \right] \,ds\right) ^{2}.
\end{align*}%
Hence the function 
\begin{equation*}
f(t)=\mathbb{E}\sup_{s\in \lbrack 0,t]}1_{[0,\tau _{N}^{n}]}\xi
_{2}^{4}(s)\Vert \tilde{\mathbf{z}}_{n}(s)\Vert _{2}^{4}
\end{equation*}%
fulfills the { Gr\"onwall } type inequality%
\begin{equation*}
\frac{1}{2}f(t)\leqslant C\int_{0}^{t}f(s)ds+C\left( \mathbb{E}%
\int_{0}^{t}\xi _{2}^{2}\left[ \left( 1+\nu \right) ||(f,g)||_{\mathcal{H}%
_{p}(\Gamma )}^{2}\left( 1+\Vert \mathbf{y}\Vert _{2}^{2}\right) \right]
\,ds\right) ^{2},
\end{equation*}%
which implies that%
\begin{align}
\mathbb{E}\sup_{s\in \lbrack 0,\tau _{N}^{n}\wedge t]}& \xi
_{2}^{4}(s)\left\Vert \tilde{\mathbf{z}}_{n}(s)\right\Vert _{2}^{4}+\nu
^{2}\left( \mathbb{E}\int_{0}^{\tau _{N}^{n}\wedge t}\xi _{2}^{2}\left\Vert 
\tilde{\mathbf{z}}_{n}\right\Vert _{V}^{2}\,ds\right) ^{2}  \notag \\
& \leqslant C\left( \mathbb{E}\int_{0}^{t}\xi _{2}^{2}\left[ \left( 1+\nu
\right) ||(f,g)||_{\mathcal{H}_{p}(\Gamma )}^{2}\left( 1+\Vert \mathbf{y}%
\Vert _{2}^{2}\right) \right] \,ds\right) ^{2}.  \label{rp}
\end{align}%
Following the same approach as in the step 1, for each fixed $n$ we can show
that there exists a subsequence $\{\tau _{N_{k}}^{n}\}$ of $\{\tau
_{N}^{n}\} $, such that 
\begin{equation*}
\tau _{N_{k}}^{n}(\omega )\rightarrow T\text{\qquad for a.e. \ }\omega \in
\Omega ,\quad \text{as }k\rightarrow \infty .
\end{equation*}
Accounting that the sequence $\left\{ \tau _{N}^{n}\right\} $ is monotone on 
$N$ and applying the monotone convergence theorem in the inequality (\ref{rp}%
) as $N\rightarrow \infty $, we deduce the estimate (\ref{eqsa4})\textit{,}
which is valid for $\tilde{\mathbf{z}}_{n}.$

\textit{Step 3.} \textit{Passing to the limit in (\ref{z}).}

Let us consider the function $\widetilde{h}=\widehat{C}_{0}(1+||(a,b)||_{%
\mathcal{H}_{p}(\Gamma )}^{2})+f_{2}$ with the constant $\widehat{C}_{0}$
and the function $f_{2},$ defined by {(\ref{ksi})} and\ {(\ref{f_22}).}
Since 
\begin{equation*}
\int_{0}^{T}\widetilde{h}(s)\ ds\leqslant C(\omega )<+\infty \text{\qquad
for a.e.\ }\omega \in \Omega
\end{equation*}%
by (\ref{eq00sec12}) and (\ref{f_22}), there exists a positive constant $%
K(\omega ),$ which depends only on $\omega ,$ satisfying%
\begin{equation}
0<K(\omega )\leqslant \xi (t)=e^{-\int_{0}^{t}\widetilde{h}(s)\ ds}\leqslant
1\text{\qquad for a.e. \ }(\omega ,t)\in \Omega \times \lbrack 0,T].
\label{q2}
\end{equation}

The estimate (\ref{eqsa}), written for $\tilde{\mathbf{z}}_{n},$ gives that%
\begin{equation}
\mathbb{E}\sup_{t\in \lbrack 0,T]}\left\Vert \xi (t)\tilde{\mathbf{z}}%
_{n}(t)\right\Vert _{2}^{2}+\nu \mathbb{E}\int_{0}^{T}\left\Vert \xi \tilde{%
\mathbf{z}}_{n}\right\Vert _{V}^{2}\ dt\leqslant C  \label{as}
\end{equation}%
for some constant $C$, which is independent of the index $n$.

Therefore, (\ref{q2})-(\ref{as}) imply that there exist a suitable
subsequence $\tilde{\mathbf{z}}_{n}$, which is indexed by the same index $n$
(to simplify the notation), and a function $\tilde{\mathbf{z}},$ such that 
\begin{eqnarray}
\xi \tilde{\mathbf{z}}_{n} &\rightharpoonup &\xi \tilde{\mathbf{z}}\qquad 
\mbox{ weakly in
}\ L^{2}(\Omega \times (0,T);V),  \notag \\
\xi \tilde{\mathbf{z}}_{n} &\rightharpoonup &\xi \tilde{\mathbf{z}}\qquad 
\mbox{ *-weakly in
}\ L^{2}(\Omega ,L^{\infty }(0,T;H)).  \label{c1}
\end{eqnarray}%
The limit function $\tilde{\mathbf{z}}$ satisfies the estimates (\ref{eqsa})
and (\ref{eqsa4}) by the lower semi-continuity of integral in the $L_{2}$%
-space and (\ref{IMP_1}).

Since $\mathbf{z}_{n}$ solves the equation (\ref{z}), then It\^{o}'s formula
gives that for $\ P$-a.e. in $\Omega $ and any $t\in \lbrack 0,T],$ we have%
\begin{equation*}
\left\{ 
\begin{array}{l}
d\left( \xi ^{2}(\mathbf{z}_{n},\boldsymbol{\varphi })\right) =\xi ^{2}[-\nu
\left( \mathbf{z}_{n},\boldsymbol{\varphi }\right) _{V}\,+\nu \int_{\Gamma
}g(\boldsymbol{\varphi }\cdot {\bm{\tau }})\,d\mathbf{\gamma }\, \\ 
\\ 
\qquad \qquad \qquad \quad -(\left( \mathbf{z}_{n}\cdot \nabla )\mathbf{y},%
\boldsymbol{\varphi }\right) -(\left( \mathbf{y}\cdot \nabla )\mathbf{z}_{n},%
\boldsymbol{\varphi }\right) -2\widetilde{h}\left( \mathbf{z}_{n},%
\boldsymbol{\varphi }\right) ]\,ds\, \\ 
\\ 
\qquad \qquad \qquad \quad \,+\xi ^{2}\left( \nabla _{\mathbf{y}}\mathbf{G}%
(t,\mathbf{y})\mathbf{z}_{n},\boldsymbol{\varphi }\right) \,d{\mathcal{W}}%
_{t},\text{\qquad }\forall \boldsymbol{\varphi }\in V_{n}, \\ 
\mathbf{z}_{n}(0)=0,%
\end{array}%
\right.
\end{equation*}%
that is 
\begin{align*}
\left( \xi ^{2}(t)\mathbf{z}_{n}(t),\boldsymbol{\varphi }\right) &
=\int_{0}^{t}\xi ^{2}(s)[-\nu \left( \mathbf{z}_{n}(s),\boldsymbol{\varphi }%
\right) _{V}\,+\nu \int_{\Gamma }b(s)(\boldsymbol{\varphi }\cdot {\bm{\tau }}%
)\,d\mathbf{\gamma } \\
& \quad -(\left( \mathbf{z}_{n}\cdot \nabla )\mathbf{y},\boldsymbol{\varphi }%
\right) -(\left( \mathbf{y}\cdot \nabla )\mathbf{z}_{n},\boldsymbol{\varphi }%
\right) -2\widetilde{h}(s)\left( \mathbf{z}_{n}(s),\boldsymbol{\varphi }%
\right) ]\,ds \\
& \quad +\int_{0}^{t}\xi ^{2}\left( \nabla _{\mathbf{y}}\mathbf{G}(s,\mathbf{%
y})\mathbf{z}_{n},\boldsymbol{\varphi }\right) \,d{\mathcal{W}}_{s}.
\end{align*}%
Multiplying the last obtained equality by arbitrary fixed $\eta \in
L^{2}(\Omega ),$ and taking the expectation, we derive%
\begin{eqnarray*}
\mathbb{E\ }\eta \left( \xi ^{2}(t)\mathbf{z}_{n}(t),\boldsymbol{\varphi }%
\right) &=&\mathbb{E\ }\eta \left\{ \int_{0}^{t}\xi ^{2}[-\nu \left( \mathbf{%
z}_{n},\boldsymbol{\varphi }\right) _{V}\,+\nu \int_{\Gamma }b(\boldsymbol{%
\varphi }\cdot {\bm{\tau }})\,d\mathbf{\gamma }\right. \\
&&-(\left( \mathbf{z}_{n}\cdot \nabla )\mathbf{y},\boldsymbol{\varphi }%
\right) -(\left( \mathbf{y}\cdot \nabla )\mathbf{z}_{n},\boldsymbol{\varphi }%
\right) -2\widetilde{h}\left( \mathbf{z}_{n},\boldsymbol{\varphi }\right)
]\,ds \\
&&\left. +\int_{0}^{t}\xi ^{2}\left( \nabla _{\mathbf{y}}\mathbf{G}(s,%
\mathbf{y})\mathbf{z}_{n},\boldsymbol{\varphi }\right) \,d{\mathcal{W}}%
_{s}\right\} .
\end{eqnarray*}%
Using that the right side of the last equation is continuous in the time
variable $t\in \lbrack 0,T]$ and applying (\ref{c1}), we pass to the limit $%
n\rightarrow \infty $ in this equality and easily deduce%
\begin{eqnarray*}
\mathbb{E\ }\eta \left( \xi ^{2}(t)\mathbf{z}(t),\boldsymbol{\varphi }%
\right) &=&\mathbb{E\ }\eta \left\{ \int_{0}^{t}\xi ^{2}[-\nu \left( \mathbf{%
z},\boldsymbol{\varphi }\right) _{V}\,+\nu \int_{\Gamma }b(\boldsymbol{%
\varphi }\cdot {\bm{\tau }})\,d\mathbf{\gamma }\right. \\
&&-(\left( \mathbf{z}\cdot \nabla )\mathbf{y},\boldsymbol{\varphi }\right)
-(\left( \mathbf{y}\cdot \nabla )\mathbf{z},\boldsymbol{\varphi }\right) -2%
\widetilde{h}\left( \mathbf{z},\boldsymbol{\varphi }\right) ]\,ds \\
&&\left. +\int_{0}^{t}\xi ^{2}\left( \nabla _{\mathbf{y}}\mathbf{G}(s,%
\mathbf{y})\mathbf{z},\boldsymbol{\varphi }\right) \,d{\mathcal{W}}%
_{s}\right\} ,\qquad \forall \boldsymbol{\varphi }\in V.
\end{eqnarray*}%
Since $\eta \in L^{2}(\Omega )$ is arbitrary, then we have\ the validity of
the equality%
\begin{align*}
\left( \xi ^{2}(t)\mathbf{z}(t),\boldsymbol{\varphi }\right) &
=\int_{0}^{t}\xi ^{2}[-\nu \left( \mathbf{z},\boldsymbol{\varphi }\right)
_{V}\,+\nu \int_{\Gamma }b(\boldsymbol{\varphi }\cdot {\bm{\tau }})\,d%
\mathbf{\gamma } \\
& \quad -(\left( \mathbf{z}\cdot \nabla )\mathbf{y},\boldsymbol{\varphi }%
\right) -(\left( \mathbf{y}\cdot \nabla )\mathbf{z},\boldsymbol{\varphi }%
\right) -2\widetilde{h}\left( \mathbf{z},\boldsymbol{\varphi }\right) ]\,ds
\\
& \quad +\int_{0}^{t}\xi ^{2}\left( \nabla _{\mathbf{y}}\mathbf{G}(s,\mathbf{%
y})\mathbf{z},\boldsymbol{\varphi }\right) \,d{\mathcal{W}}_{s},
\end{align*}%
that is%
\begin{align*}
d\left( \xi ^{2}(\mathbf{z},\boldsymbol{\varphi })\right) & =\xi ^{2}[-\nu
\left( \mathbf{z},\boldsymbol{\varphi }\right) _{V}\,+\nu \int_{\Gamma }g(%
\boldsymbol{\varphi }\cdot {\bm{\tau }})\,d\mathbf{\gamma } \\
& -(\left( \mathbf{z}\cdot \nabla )\mathbf{y},\boldsymbol{\varphi }\right)
-(\left( \mathbf{y}\cdot \nabla )\mathbf{z},\boldsymbol{\varphi }\right) -2%
\widetilde{h}\left( \mathbf{z},\boldsymbol{\varphi }\right) ]\,ds\, \\
& \quad \,+\xi ^{2}\left( \nabla _{\mathbf{y}}\mathbf{G}(t,\mathbf{y})%
\mathbf{z},\boldsymbol{\varphi }\right) \,d{\mathcal{W}}_{t}.
\end{align*}%
Moreover if we use It\^{o}'s formula $\ $%
\begin{equation*}
\ d\left( \mathbf{z},\boldsymbol{\varphi }\right) =d\left[ \xi ^{-2}\xi
^{2}\left( \mathbf{z},\boldsymbol{\varphi }\right) \right] =\xi ^{2}\left( 
\mathbf{z},\boldsymbol{\varphi }\right) d\left( \xi ^{-2}\right) +\xi ^{-2}d%
\left[ \xi ^{2}\left( \mathbf{z},\boldsymbol{\varphi }\right) \right] ,
\end{equation*}%
we obtain that the limit function $\mathbf{z}$ in the form $\mathbf{z}=%
\tilde{\mathbf{z}}+\mathbf{f}$ fulfills \ the stochastic differential
equation%
\begin{equation}
\left\{ 
\begin{array}{l}
d(\mathbf{z},\boldsymbol{\varphi })=\left[ -\nu \left( \mathbf{z},%
\boldsymbol{\varphi }\right) _{V}\,+\nu \int_{\Gamma }g(\boldsymbol{\varphi }%
\cdot {\bm{\tau }})\,d\mathbf{\gamma }-(\left( \mathbf{z}\cdot \nabla )%
\mathbf{y},\boldsymbol{\varphi }\right) -(\left( \mathbf{y}\cdot \nabla )%
\mathbf{z},\boldsymbol{\varphi }\right) \,\right] dt \\ 
\\ 
\qquad \quad \quad +\left( \nabla _{\mathbf{y}}\mathbf{G}(t,\mathbf{y})%
\mathbf{z},\boldsymbol{\varphi }\right) \,d{\mathcal{W}}_{t},\qquad \forall 
\boldsymbol{\varphi }\in V,\text{\quad a.e. in }\Omega \times (0,T), \\ 
\\ 
\mathbf{z}(0)=0.%
\end{array}%
\right.  \label{y12}
\end{equation}%
The uniqueness result follows from the linearity of this system by taking
into account the estimates \eqref{eqsa}-\eqref{eqsa4}.
\end{proof}

\section{G\^{a}teaux differentiability of the control-to-state mapping}

\label{sec7}\setcounter{equation}{0}

To derive the necessary conditions for first-order optimality, it is necessary to study
differentiability of the Gateau cost functional, which is based on
Gateau derivative mapping control to state

To deduce the necessary first-order optimality conditions, it is necessary to study 
the G\^{a}teaux differentiability of the cost functional $J$, { which is based } on the G\^{a}teaux derivative of the control-to-state mapping. Introducing
{ additional } assumptions, in this section we show that the G\^{a}teaux
derivative of the control-to-state mapping $(a,b)\rightarrow \mathbf{y}$, at
a point $(a,b)$, in any direction $(f,g)$, exists and is given by the
solution of the linearized system (\ref{linearized}).

\begin{proposition}
\label{Gat} Assume 
\begin{equation}
A_{\ast }\geqslant 32\max \left( \widehat{C}_{1},\widehat{C}_{2}\right)
\times \max \left( \nu ^{-1},1\right)  \label{C_A1}
\end{equation}%
with $A_{\ast }$ given by \eqref{aaa} and $\widehat{C}_{1}$, $\widehat{C}%
_{2} $ as in Theorem \ref{Lips} and Proposition \ref{ex_uniq_lin}. Let $%
(a,b),\;(f,g)\in \mathcal{A}$, $\mathbf{y}_{0}$ as in \eqref{bound0} and
consider the data 
\begin{equation}
a_{\varepsilon }=a+\varepsilon f,\quad b_{\varepsilon }=b+\varepsilon
g,\qquad \forall \varepsilon \in (0,1).  \label{abe}
\end{equation}%
Let $\left( \mathbf{y},q\right) $, $\left( \mathbf{y}_{\varepsilon
},q_{\varepsilon }\right) $ be the solutions of (\ref{NSy}) corresponding to 
$(a,b,\mathbf{y}_{0})$ and $(a_{\varepsilon },b_{\varepsilon },\mathbf{y}%
_{0}),$ respectively. Then the following representation holds 
\begin{equation}
\mathbf{y}_{\varepsilon }=\mathbf{y}+\varepsilon \mathbf{z}+\varepsilon \,%
\boldsymbol{\delta }_{\varepsilon }\quad \mbox{   with  }\quad
\lim_{\varepsilon \rightarrow 0}\mathbb{E}\int_{0}^{T}\left\Vert \boldsymbol{%
\delta }_{\varepsilon }\right\Vert _{V}^{2}\,ds=0,  \label{gateau_1}
\end{equation}%
where%
\begin{equation*}
\mathbf{z}\in C([0,T];H)\cap L_{2}(0,T;V),\quad P\text{-a.e. \ }\ \omega \in
\Omega ,
\end{equation*}%
is the solution of (\ref{linearized}) satisfying the estimate (\ref{eqsa}).
\end{proposition}

\begin{proof}
Let us define 
$\mathbf{z}_{\varepsilon }=\frac{\mathbf{y}_{\varepsilon }-\mathbf{y}}{\varepsilon }$ and $\pi _{\varepsilon }=\frac{q_{\varepsilon }-q%
}{\varepsilon }.$ By direct calculations we can check that the pair $(%
\mathbf{z}_{\varepsilon },\pi _{\varepsilon })$ fulfills the following
system in the distributional sense%
\begin{equation}
\left\{ 
\begin{array}{l}
d\mathbf{z}_{\varepsilon }=\left( \nu \Delta \mathbf{z}_{\varepsilon
}-\left( \mathbf{y}\cdot \nabla \right) \mathbf{z}_{\varepsilon }-\left( 
\mathbf{z}_{\varepsilon }\cdot \nabla \right) \mathbf{y}_{\varepsilon
}-\nabla \pi _{\varepsilon }\right) dt \\ 
\\ 
\quad \quad \quad +\frac{1}{\varepsilon }\left( \mathbf{G}(t,\mathbf{y}%
_{\varepsilon })-\mathbf{G}(t,\mathbf{y})\right) \,d{\mathcal{W}}_{t},\quad
\quad \mathrm{div}\,\mathbf{z}_{\varepsilon }=0\;\quad \text{ in}\quad {%
\mathcal{O}}_{T}, \\ 
\\ 
\mathbf{z}_{\varepsilon }\cdot \mathbf{n}=f,\;\quad \left[ 2D(\mathbf{z}%
_{\varepsilon })\,\mathbf{n}+\alpha \mathbf{z}_{\varepsilon }\right] \cdot {%
\ \mathbf{\bm{\tau }}}=g\;\;\quad \text{on}\quad \Gamma _{T}, \\ 
\\ 
\mathbf{z}_{\varepsilon }(0,\mathbf{x})=0\;\quad \text{in}\quad {\mathcal{O}}%
\end{array}%
\right.  \label{NSVG}
\end{equation}%
and $\boldsymbol{\delta }_{\varepsilon }=\mathbf{z}_{\varepsilon }-\mathbf{z}
$\ satisfies the system 
\begin{equation}
\left\{ 
\begin{array}{l}
d\boldsymbol{\delta }_{\varepsilon }=\left( \nu \Delta \boldsymbol{\delta }%
_{\varepsilon }-\left( \mathbf{y}\cdot \nabla \right) \boldsymbol{\delta }%
_{\varepsilon }-\left( \boldsymbol{\delta }_{\varepsilon }\cdot \nabla
\right) \mathbf{y}-\left( \mathbf{z}_{\varepsilon }\cdot \nabla \right)
\left( \mathbf{y}_{\varepsilon }-\mathbf{y}\right) \vspace{2mm}-\nabla
\left( \pi _{\varepsilon }-\pi \right) \right) dt \\ 
\\ 
\quad \quad \quad +\mathbf{R}\,d{\mathcal{W}}_{t},\quad \quad \mathrm{div}\,%
\boldsymbol{\delta }_{\varepsilon }=0\;\quad \text{ in}\quad {\mathcal{O}}%
_{T}, \\ 
\\ 
\boldsymbol{\delta }_{\varepsilon }\cdot \mathbf{n}=0,\;\quad \left[ 2D(%
\boldsymbol{\delta }_{\varepsilon })\,\mathbf{n}+\alpha \boldsymbol{\delta }%
_{\varepsilon }\right] \cdot {\mathbf{\bm{\tau }}}=0\;\;\quad \text{on}\quad
\Gamma _{T}, \\ 
\\ 
\boldsymbol{\delta }_{\varepsilon }(0,\mathbf{x})=0\;\quad \text{in}\quad {%
\mathcal{O}},%
\end{array}%
\right.  \label{NS00}
\end{equation}%
where 
\begin{align*}
\mathbf{R}& =\frac{1}{\varepsilon }\left[ \mathbf{G}(t,\mathbf{y}%
_{\varepsilon })-\mathbf{G}(t,\mathbf{y})\right] -\nabla _{\mathbf{y}}%
\mathbf{G}(t,\mathbf{y})\mathbf{z} \\
& =\frac{1}{\varepsilon }\left[ \mathbf{G}(t,\mathbf{y}+\varepsilon \mathbf{z%
})-\mathbf{G}(t,\mathbf{y})\right] -\nabla _{\mathbf{y}}\mathbf{G}(t,\mathbf{%
y})\mathbf{z}+\frac{1}{\varepsilon }\left[ \mathbf{G}(t,\mathbf{y}%
_{\varepsilon })-\mathbf{G}(t,\mathbf{y}+\varepsilon \mathbf{z})\right] .
\end{align*}%
The It\^{o} formula, applied to the stochastic differential equation of (%
\ref{NS00}), gives 
\begin{align}
d\left( ||\boldsymbol{\delta }_{\varepsilon }||_{2}^{2}\right) & +2\nu ||%
\boldsymbol{\delta }_{\varepsilon }||_{V}^{2}dt=\left( -\int_{\Gamma }a(%
\boldsymbol{\delta }_{\varepsilon }\cdot \bm{\tau })^{2}\,d\mathbf{\gamma -}%
\left( 2\left( \boldsymbol{\delta }_{\varepsilon }\cdot \nabla \right) 
\mathbf{y},\boldsymbol{\delta }_{\varepsilon }\right) \right) \,dt^{\prime }
\notag \\
& +\frac{1}{\varepsilon }\left( -2\left( \left( \mathbf{y}_{\varepsilon }-%
\mathbf{y}\right) \cdot \nabla \right) \left( \mathbf{y}_{\varepsilon }-%
\mathbf{y}\right) ,\boldsymbol{\delta }_{\varepsilon }\right) \,dt  \notag \\
&  \notag \\
& +\Vert \mathbf{R}\Vert _{2}^{2}\,dt+2\left( \mathbf{R},\boldsymbol{\delta }%
_{\varepsilon }\right) \,d{\mathcal{W}}_{t}  \notag \\
&  \notag \\
& =\left( I_{1}+I_{2}+I_{3}+I_{4}\right) dt+2\left( \mathbf{R},\boldsymbol{%
\delta }_{\varepsilon }\right) \,d{\mathcal{W}}_{t}.  \label{er}
\end{align}%
Applying the inequalities \eqref{LI}-\eqref{TT}, \eqref{yi} and \eqref{cal},
the following estimates hold%
\begin{eqnarray*}
I_{1} &\leqslant &C\Vert a\Vert _{L_{\infty }(\Gamma )}\Vert \boldsymbol{%
\delta }_{\varepsilon }\Vert _{L_{2}(\Gamma )}^{2}\leqslant C\nu
^{-1}||a||_{W_{p}^{1-\frac{1}{p}}(\Gamma )}^{2}||\boldsymbol{\delta }%
_{\varepsilon }||_{2}^{2}+\frac{\nu }{3}||\boldsymbol{\delta }_{\varepsilon
}||_{V}^{2}, \\
I_{2} &\leqslant &C\Vert \mathbf{y}\Vert _{H^{1}}\Vert \boldsymbol{\delta }%
_{\varepsilon }\Vert _{4}^{2}\leqslant C\nu ^{-1}\left( \Vert \mathbf{u}%
\Vert _{H^{1}}^{2}+||(a,b)||_{\mathcal{H}_{p}(\Gamma )}^{2}\right) \Vert 
\boldsymbol{\delta }_{\varepsilon }\Vert _{2}^{2}+\frac{\nu }{3}||%
\boldsymbol{\delta }_{\varepsilon }||_{V}^{2}
\end{eqnarray*}%
and 
\begin{eqnarray*}
I_{3} &=&\frac{2}{\varepsilon }\left( \left( \left( \mathbf{y}_{\varepsilon
}-\mathbf{y}\right) \cdot \nabla \right) \boldsymbol{\delta }_{\varepsilon
},\left( \mathbf{y}_{\varepsilon }-\mathbf{y}\right) \right) -2\int_{\Gamma
}f\left( \mathbf{y}_{\varepsilon }-\mathbf{y}\right) \cdot \boldsymbol{%
\delta }_{\varepsilon }d\gamma \\
\ &\leqslant &\frac{2}{\varepsilon }\Vert \mathbf{y}_{\varepsilon }-\mathbf{y%
}\Vert _{2}\Vert \mathbf{y}_{\varepsilon }-\mathbf{y}\Vert _{V}\Vert 
\boldsymbol{\delta }_{\varepsilon }\Vert _{V}+\Vert f\Vert _{L_{\infty
}(\Gamma )}\Vert \mathbf{y}_{\varepsilon }-\mathbf{y}\Vert _{2}^{1/2}\Vert 
\mathbf{y}_{\varepsilon }-\mathbf{y}\Vert _{V}^{1/2}\Vert \boldsymbol{\delta 
}_{\varepsilon }\Vert _{V} \\
\ &\leqslant &\frac{C\nu ^{-1}}{\varepsilon ^{2}}\Vert \mathbf{y}%
_{\varepsilon }-\mathbf{y}\Vert _{2}^{2}\Vert \mathbf{y}_{\varepsilon }-%
\mathbf{y}\Vert _{V}^{2}+C\nu ^{-1}||f||_{W_{p}^{1-\frac{1}{p}}(\Gamma
)}\Vert \mathbf{y}_{\varepsilon }-\mathbf{y}\Vert _{2}\Vert \mathbf{y}%
_{\varepsilon }-\mathbf{y}\Vert _{V}+\frac{\nu }{3}||\boldsymbol{\delta }%
_{\varepsilon }||_{V}^{2}.
\end{eqnarray*}%
Due to \eqref{G}, (\ref{cG}) and (\ref{LG}) we have {\ 
\begin{equation}
I_{4}=\Vert \mathbf{R}\Vert _{2}^{2}\leqslant \frac{2}{\varepsilon ^{2}}%
\Vert o\left( t,\left\Vert \varepsilon \mathbf{z}\right\Vert _{2}\right)
\Vert _{2}^{2}+2K\Vert \boldsymbol{\delta }_{\varepsilon }\Vert
_{2}^{2}\;\quad \text{a.e. in }[0,T]\times \Omega  \label{RR}
\end{equation}%
with 
\begin{equation*}
\frac{2}{\varepsilon ^{2}}\Vert o\left( t,\left\Vert \varepsilon \mathbf{z}%
\right\Vert _{2}\right) \Vert _{2}^{2}=\Vert \mathbf{z}\Vert
_{2}^{2}\left\Vert \frac{o\left( t,\Vert \varepsilon \mathbf{z}\Vert
_{2}\right) }{\Vert \varepsilon \mathbf{z}\Vert _{2}}\right\Vert _{2}^{2}%
\text{ }\longrightarrow 0\;\text{ as }\varepsilon \rightarrow 0,\quad \text{%
a.e. in }[0,T]\times \Omega ,
\end{equation*}%
and 
\begin{equation*}
\frac{2}{\varepsilon ^{2}}\Vert o\left( t,\left\Vert \varepsilon \mathbf{z}%
\right\Vert _{2}\right) \Vert _{2}^{2}\leqslant C\Vert \mathbf{z}\Vert
_{2}^{2},
\end{equation*}%
where $C$ is a constant independent of $\varepsilon .$ }

Considering the above deduced estimates for the terms $I_{i,}$ $i=1,...,4,$
we ensure that the following estimate holds%
\begin{eqnarray}
I_{1}+I_{2}+I_{3}+I_{4} &\leqslant &\bigl[2h(t)||\boldsymbol{\delta }%
_{\varepsilon }||_{2}^{2}+\frac{C\nu ^{-1}}{\varepsilon ^{2}}\Vert \mathbf{y}%
_{\varepsilon }-\mathbf{y}\Vert _{2}^{2}\Vert \mathbf{y}_{\varepsilon }-%
\mathbf{y}\Vert _{V}^{2}  \notag \\
&&  \notag \\
&&+C\nu ^{-1}||f||_{W_{p}^{1-\frac{1}{p}}(\Gamma )}\Vert \mathbf{y}%
_{\varepsilon }-\mathbf{y}\Vert _{2}\Vert \mathbf{y}_{\varepsilon }-\mathbf{y%
}\Vert _{V}+\nu ||\boldsymbol{\delta }_{\varepsilon }||_{V}^{2}  \notag \\
&&  \notag \\
&&+{\frac{2}{\varepsilon ^{2}}\Vert o\left( t,\left\Vert \varepsilon \mathbf{%
z}\right\Vert _{2}\right) \Vert _{2}^{2}}\bigr]dt,  \label{dt}
\end{eqnarray}%
where 
\begin{equation}
h(t)=C(\nu ^{-1}+1)\left( 1+\Vert \mathbf{u}\Vert _{V}^{2}+||(a,b)||_{%
\mathcal{H}_{p}(\Gamma )}^{2}\right)  \label{HH}
\end{equation}%
for some positive constant $C$. We can choose $C$ in \eqref{HH} and $%
\widehat{C}_{2}$ in \eqref{f_22} such a way that these constants are the
same. Therefore we can consider that the functions $h(t),$ $f_{2}(t)$ are
equal.

Now let us introduce the function 
\begin{equation*}
\beta (t)=e^{-\int_{0}^{t}2\text{max}\left( f_{1}(s),f_{2}(s)\right) \,ds}
\end{equation*}%
with $f_{1}$ and $f_{2}$ defined in \eqref{f__12} and \eqref{f_22} for the
data $(a_{\varepsilon },b_{\varepsilon }),$ $\mathbf{y}_{0}$ and $(a,b),$ $%
\mathbf{y}_{0}$, respectively. \ Considering the equality \eqref{er}, the 
It\^{o} formula yields 
\begin{eqnarray}
\beta ^{2}(t)\left\Vert \boldsymbol{\delta }_{\varepsilon }(t)\right\Vert
_{2}^{2} &+&\nu \int_{0}^{t}\beta ^{2}\left\Vert \boldsymbol{\delta }%
_{\varepsilon }\right\Vert _{V}^{2}\,ds\leqslant C\int_{0}^{t}g_{\varepsilon
}(s)\,ds+  \notag \\
&&  \notag \\
&+&{\int_{0}^{t}\beta ^{2}}\left( {\frac{2}{\varepsilon ^{2}}\Vert o\left(
s,\left\Vert \varepsilon \mathbf{z}\right\Vert _{2}\right) \Vert _{2}^{2}}%
\right) \,ds+2\int_{0}^{t}\beta ^{2}\left( \mathbf{R},\boldsymbol{\delta }%
_{\varepsilon }\right) \,d{\mathcal{W}}_{s},\qquad  \label{222}
\end{eqnarray}%
where%
\begin{equation*}
g_{\varepsilon }(t)=\beta ^{2}(t)\nu ^{-1}\left( \frac{1}{\varepsilon ^{2}}%
\Vert \mathbf{y}_{\varepsilon }-\mathbf{y}\Vert _{2}^{2}\Vert \mathbf{y}%
_{\varepsilon }-\mathbf{y}\Vert _{V}^{2}+||f||_{W_{p}^{1-\frac{1}{p}}(\Gamma
)}\Vert \mathbf{y}_{\varepsilon }-\mathbf{y}\Vert _{2}\Vert \mathbf{y}%
_{\varepsilon }-\mathbf{y}\Vert _{V}\right) .
\end{equation*}%
On the other hand, using the { H\"{o}lder } inequality and \eqref{2222}-%
\eqref{1111}, \eqref{regf}, \eqref{abe}, we derive 
\begin{equation*}
\mathbb{E}\int_{0}^{t}g_{\varepsilon }(s)\,ds\leqslant C\varepsilon
^{2},\qquad \forall t\in \lbrack 0,T],
\end{equation*}%
and {the Lebesgue dominated convergence theorem gives 
\begin{equation*}
\mathbb{E}\int_{0}^{t}\beta ^{2}\left( {\frac{2}{\varepsilon ^{2}}\Vert
o\left( s,\left\Vert \varepsilon \mathbf{z}\right\Vert _{2}\right) \Vert
_{2}^{2}}\right) ds\rightarrow 0,\text{ as }\varepsilon \rightarrow 0.
\end{equation*}%
Therefore, applying the mathematical expectation to \eqref{222} and taking
the limit, as $\varepsilon \rightarrow 0$, we deduce 
\begin{equation}
\lim_{\varepsilon \rightarrow 0}\mathbb{E}\beta ^{2}(t)\left\Vert 
\boldsymbol{\delta }_{\varepsilon }(t)\right\Vert _{2}^{2}=0,\qquad
\lim_{\varepsilon \rightarrow 0}\mathbb{E}\int_{0}^{t}\beta ^{2}\left\Vert 
\boldsymbol{\delta }_{\varepsilon }\right\Vert _{V}^{2}\,ds=0,\;\;\forall
t\in \lbrack 0,T].  \label{conv_delta}
\end{equation}%
Taking into account {(\ref{yy})}}$_{2}$, {\eqref{C_A1} and the definition of 
$\beta $, we have $\mathbb{E}\left( \beta ^{-8}(T)\right) <\infty .$ Hence
the H\"{o}lder inequality and the monotonicity of $\beta $ give 
\begin{eqnarray}
\mathbb{E}\int_{0}^{T}\left\Vert \boldsymbol{\delta }_{\varepsilon
}(s)\right\Vert _{V}^{2}\,ds &\leqslant &\mathbb{E}\left[ \beta
^{-2}(T)\int_{0}^{T}\beta ^{2}(s)\left\Vert \boldsymbol{\delta }%
_{\varepsilon }(s)\right\Vert _{V}^{2}\,ds\right]  \notag \\
&\leqslant &\left[ \mathbb{E}\left( \int_{0}^{T}\beta ^{2}(s)\left\Vert 
\boldsymbol{\delta }_{\varepsilon }(s)\right\Vert _{V}^{2}\,ds\right) \right]
^{\frac{1}{2}}\left[ \mathbb{E}\left( \beta ^{-8}(T)\right) \right] ^{\frac{1%
}{4}}  \notag \\
&&\left[ \mathbb{E}\left( \int_{0}^{T}\beta ^{2}(s)\left\Vert \boldsymbol{%
\delta }_{\varepsilon }(s)\right\Vert _{V}^{2}\,ds\right) ^{2}\right] ^{%
\frac{1}{4}},  \label{v_delta}
\end{eqnarray}%
then using sequentially the relation in \eqref{gateau_1}, \eqref{abe}, %
\eqref{1111} and \eqref{eqsa4}, we deduce the following uniform estimate
with respect to $\varepsilon $ 
\begin{eqnarray*}
\mathbb{E}\left( \int_{0}^{T}\beta ^{2}(s)\left\Vert \boldsymbol{\delta }%
_{\varepsilon }(s)\right\Vert _{V}^{2}\,ds\right) ^{2} &\leqslant &4\mathbb{E%
}\left( \int_{0}^{T}\xi _{1}^{2}(s)\left\Vert \frac{\mathbf{y}_{\varepsilon
}-\mathbf{y}}{\varepsilon }\right\Vert _{V}^{2}\,ds\right) ^{2} \\
&&+4\mathbb{E}\left( \int_{0}^{T}\xi _{2}^{2}(s)\Vert \mathbf{z}\Vert
_{V}^{2}\,ds\right) ^{2}\leqslant C,
\end{eqnarray*}%
since} $\beta $ \ $\leqslant \xi _{1},\,\xi _{2}$ for $\xi _{1}$ and $\xi
_{2}$ defined by \eqref{aux} and \eqref{xi2}.{{\ Using this estimate
together with \eqref{conv_delta}, we see that the inequality \eqref{v_delta}
yields {\eqref{gateau_1}.}}}$\hfill \hfill $
\end{proof}

\section{Adjoint equation}

\label{sec8}\setcounter{equation}{0}

This section is devoted to the study of the adjoint system.

Let $\mathbf{y}$ be the solution of the state equation \eqref{NSy}
corresponding to the given data $(a,b,\mathbf{y}_{0})$, then the adjoint
system is given by 
\begin{equation}
\left\{ 
\begin{array}{ll}
\begin{array}{l}
-d\mathbf{p}=\left[ \nu \Delta \mathbf{p}+2D(\mathbf{p})\mathbf{y}-\nabla
\pi +\mathbf{U}\right] \,dt \\ 
\\ 
\,\quad \quad \quad +\nabla _{\mathbf{y}}\mathbf{G}(t,\mathbf{y})^{T}\mathbf{%
q}\,dt-\mathbf{q}\,d{\mathcal{W}}_{t}\mathbf{,\qquad }\mathrm{div}\,\mathbf{p%
}=0%
\end{array}
& \mbox{in}\ {\mathcal{O}}_{T},\vspace{2mm} \\ 
\mathbf{p}\cdot \mathbf{n}=0,\qquad \left[ 2D(\mathbf{p})\mathbf{n}+(\alpha +%
{\frac{a}{\nu }})\mathbf{p}\right] \cdot \bm{\tau }=0, & \mbox{on}\ \Gamma
_{T},\vspace{2mm} \\ 
\mathbf{p}(T)=0, & \mbox{in}\ {\mathcal{O}}.%
\end{array}%
\right.  \label{adjoint}
\end{equation}

In the next { Section } \ref{sec9}, we will prove that the adjoint state $%
\mathbf{p}$ and the linearized state $\mathbf{z}$ are related through the
duality relation (\ref{duality}). In order to give a meaning to boundary
terms that will appear in that relation, it is necessary to have the $H^{2}$%
-regularity of the adjoint state $\mathbf{p}$, that we demonstrate in the
following proposition.

\begin{proposition}
\label{ex_uniq_adj} Let $(a,b)$, $\mathbf{u}_{0}$ satisfy \eqref{bound0} and 
$\mathbf{y}$ be the solution of the state system (\ref{NSy}), \ constructed
in Theorem \ref{the_1}. \ Assume that 
\begin{equation*}
\mathbf{U}\in L_{4}(\Omega \times (0,T);L_{2}(\mathcal{O}))
\end{equation*}
is a predictable stochastic process. In addition, we assume that 
\begin{equation}
\mathrm{min}(A_{\ast },B_{\ast },\frac{B_{\ast }}{\widehat{C}})\geqslant 24\,%
\mathrm{max}(\widetilde{C}_{1},\widetilde{C}_{2})\max (\nu ^{-2},1)
\label{cnd_1}
\end{equation}
with $A_{\ast },$ $\ B_{\ast },$ $\ \widehat{C}$ defined by (\ref{aaa}) and $%
\widetilde{C}_{1},\,\widetilde{C}_{2}$ defined by the expressions %
\eqref{J_12}, \eqref{con22} below.

Then, there exists a predictable stochastic process $(\mathbf{p,q,}\pi ),$
such that%
\begin{align}
& \mathbf{p}\in L_{\infty }(0,T;L_{2}(\Omega ;V))\cap L_{2}(\Omega \times
(0,T);H^{2}({\mathcal{O}})),\quad \mathbf{q}\in L_{2}(\Omega \times (0,T);V),
\notag \\
& \pi \in L_{2}(\Omega \times (0,T);H^{1}({\mathcal{O}})),  \label{up}
\end{align}%
which fulfills the system (\ref{adjoint}) for $P-$a.e. in $\Omega $.
Moreover, the following estimates hold 
\begin{eqnarray}  \label{qe1}
&&\sup_{t\in \lbrack 0,T]}\mathbb{E}\left( \left\Vert \mathbf{p}%
(t)\right\Vert _{2}^{2}\right) +\mathbb{E}\int_{0}^{T}(\nu \left\Vert 
\mathbf{p}\right\Vert _{V}^{2}+\left\Vert \mathbf{q}\right\Vert
_{2}^{2})\,dt\leqslant C\left( \mathbb{E}\int_{0}^{T}||\mathbf{U}%
||_{2}^{4}\,dt\right) ^{\frac{1}{2}},  \notag \\
&&\sup_{t\in \lbrack 0,T]}\left( \mathbb{E}\,\left\Vert \mathbf{p}%
(t)\right\Vert _{V}^{2}\right) +\mathbb{E}\int_{0}^{T}(\nu \left\Vert 
\mathbf{p}\right\Vert _{H^{2}}^{2}+\Vert \nabla \pi \Vert _{2}+\left\Vert 
\mathbf{q}\right\Vert _{V}^{2})\,dt\leqslant C\left( \mathbb{E}\int_{0}^{T}||%
\mathbf{U}||_{2}^{4}\,dt\right) ^{\frac{1}{2}}.  \notag \\
\end{eqnarray}%

\end{proposition}

\begin{proof}
The proof is divided into three steps.  \vspace{1mm}\newline
\textit{Step 1. Finite dimensional approximations. }The existence of a
solution for the system (\ref{adjoint})\textbf{\ }will be shown by
Galerkin's method. Here, we introduce a basis $\{\mathbf{h}_{k}\}\subset
H^{2}({\mathcal{O}})\cap V$ of eigenfunctions of the Stokes problem with the
homogeneous Navier-slip boundary conditions 
\begin{equation}
\left\{ 
\begin{array}{ll}
-\Delta \mathbf{h}_{k}+\nabla \pi _{k}=\lambda _{k}\mathbf{h}_{k},\mathbf{%
\qquad }\mbox{div}\,\mathbf{h}_{k}=0, & \mbox{in}{\ \mathcal{O}},\vspace{2mm}
\\ 
&  \\ 
\mathbf{h}_{k}\cdot \mathbf{n}=0,\;\quad \left[ 2D(\mathbf{h}_{k})\,\mathbf{n%
}+\left( \alpha +{\frac{a}{\nu }}\right) \mathbf{h}_{k}\right] \cdot {%
\bm{\tau }}=0\;\quad & \text{on}\ \Gamma .%
\end{array}%
\right.  \label{y4}
\end{equation}%
Let 
\begin{equation*}
\mathbf{p}_{n}(t)=\sum_{j=1}^{n}s_{j}^{(n)}(t)\mathbf{h}_{j}\in
C([0,T];V_{n})\quad \text{and}\quad \mathbf{q}_{n}=\sum_{j=1}^{n}\mathfrak{q}%
_{j}^{(n)}(t)\mathbf{h}_{j}
\end{equation*}%
be the solution of the following backward stochastic system 
\begin{equation}
\left\{ 
\begin{array}{l}
-d\left( \mathbf{p}_{n}(t),\mathbf{h}_{k}\right) +\{\nu \left( \mathbf{p}%
_{n},\,\mathbf{h}_{k}\right) _{V}-\left( 2D(\mathbf{p}_{n})\mathbf{y},%
\mathbf{h}_{k}\right) \\ 
\\ 
\qquad \qquad \qquad \,\,-\left( \mathbf{U},\mathbf{h}_{k}\right)
+\int_{\Gamma }a(\mathbf{p}_{n}\cdot \bm{\tau })(\mathbf{h}_{k}\cdot %
\bm{\tau })\,d\mathbf{\gamma }\}\,dt \\ 
\\ 
\qquad \qquad \qquad \,\,=(\nabla _{\mathbf{y}}\mathbf{G}(t,\mathbf{y})^{T}%
\mathbf{q}_{n},\mathbf{h}_{k})\,dt-\left( \mathbf{q}_{n},\mathbf{h}%
_{k}\right) \,d{\mathcal{W}}_{t}, \\ 
\\ 
\mathbf{p}_{n}(T)=\mathbf{0},\text{\qquad }k=1,2,\dots ,n,\qquad P\text{%
-a.e. in }\Omega \text{,\quad a.e. on }(0,T).%
\end{array}%
\right.  \label{pn}
\end{equation}%
Following \cite{backward-SPDE}, Prop. 6.20, the system \ (\ref{pn}) of $n$
linear stochastic differential equations has a unique global-in-time
solution $\left( \mathbf{p}_{n},\mathbf{q}_{n}\right) $, which is an adapted
process verifying 
\begin{equation}
\left( \mathbf{p}_{n},\mathbf{q}_{n}\right) \in L_{r}(\Omega ;C([0,{T}%
];H_{n}))\times L_{r}(\Omega _{T};L_{r}({\mathcal{O}})),\qquad 1\leqslant
r<\infty .  \label{pq_int}
\end{equation}

\textit{Step 2. Uniform estimates.} Now, we show the first uniform estimates
for $\left( \mathbf{p}_{n},\mathbf{q}_{n}\right) $ with respect to $n$. The
It\^{o} formula gives 
\begin{align*}
& -d\left( \left( \mathbf{p}_{n},\mathbf{h}_{k}\right) ^{2}\right) +2\left( 
\mathbf{p}_{n},\mathbf{h}_{k}\right) \{\nu \left( \mathbf{p}_{n},\,\mathbf{h}%
_{k}\right) _{V}-\left( 2D(\mathbf{p}_{n})\mathbf{y},\mathbf{h}_{k}\right) \,%
\vspace{2mm} \\
& -\left( \mathbf{U},\mathbf{h}_{k}\right) +\int_{\Gamma }a(\mathbf{p}%
_{n}\cdot \bm{\tau })(\mathbf{h}_{k}\cdot \bm{\tau })\,d\mathbf{\gamma }%
\}\,dt\, \\
& =2\left( \mathbf{p}_{n},\mathbf{h}_{k}\right) \left( \nabla _{\mathbf{y}}%
\mathbf{G}(t,\mathbf{y})^{T}\mathbf{q}_{n},\mathbf{h}_{k}\right)
\,dt-2\left( \mathbf{p}_{n},\mathbf{h}_{k}\right) \left( \mathbf{q}_{n},%
\mathbf{h}_{k}\right) \,d{\mathcal{W}}_{t}-|\left( \mathbf{q}_{n},\mathbf{h}%
_{k}\right) |^{2}\,dt.
\end{align*}%
Summing these equalities over $k=1,\dots ,n,$ we obtain 
\begin{align}
& -d\left( \Vert \mathbf{p}_{n}\Vert _{2}^{2}\right) +\left( 2\nu \Vert 
\mathbf{p}_{n}\Vert _{V}^{2}+\Vert \mathbf{q}_{n}\Vert _{2}^{2}\right)
\,dt=\left\{ 2\left( \left( (\nabla \mathbf{p}_{n})^{T}\right) \mathbf{y}+%
\mathbf{U},\mathbf{p}_{n}\right) \right.  \notag \\
&  \notag \\
& \left. -\int_{\Gamma }a(\mathbf{p}_{n}\cdot \bm{\tau })^{2}\,d\mathbf{%
\gamma }+2\left( \nabla _{\mathbf{y}}\mathbf{G}(t,\mathbf{y})^{T}\mathbf{q}%
_{n},\mathbf{p}_{n}\right) \right\} \,dt-2\left( \mathbf{q}_{n},\mathbf{p}%
_{n}\right) \,d{\mathcal{W}}_{t}  \notag \\
&  \notag \\
& =\left[ J_{1}+J_{2}+J_{3}\right] dt-2\left( \mathbf{q}_{n},\mathbf{p}%
_{n}\right) \,d{\mathcal{W}}_{t}.  \label{eqeq22}
\end{align}%
Let us estimate the terms $J_{1}$, $J_{2}$ and $J_{3}.$ Knowing that $div \, 
\mathbf{p}_{n}=0$, applying the Gagliardo-Nirenberg-Sobolev inequality (\ref%
{LI}) with $q=4$ and Young's inequality (\ref{yi}), we obtain%
\begin{eqnarray*}
J_{1} &=&-2\left( \left( \nabla \mathbf{y}\right) \mathbf{p}_{n},\mathbf{p}%
_{n}\right) +2\left( \mathbf{U},\mathbf{p}_{n}\right) \leqslant \,C\Vert 
\mathbf{y}\Vert _{H^{1}}\,\Vert \mathbf{p}_{n}\Vert _{4}^{2}+C\Vert \mathbf{U%
}\Vert _{2}\Vert \mathbf{p}_{n}\Vert _{2} \\
&& \\
&\leqslant &C\Vert \mathbf{y}\Vert _{H^{1}}\Vert \mathbf{p}_{n}\Vert
_{2}\Vert \nabla \mathbf{p}_{n}\Vert _{2}+C\Vert \mathbf{U}\Vert _{2}\Vert 
\mathbf{p}_{n}\Vert _{2} \\
&& \\
&\leqslant &C(\nu ^{-1}\Vert \mathbf{y}\Vert _{H^{1}}^{2}+1)\Vert \mathbf{p}%
_{n}\Vert _{2}^{2}+\frac{\nu }{2}||\mathbf{p}_{n}||_{V}^{2}+\Vert \mathbf{U}%
\Vert _{2}^{2},
\end{eqnarray*}%
and%
\begin{align*}
J_{2}& \leqslant C\Vert {a}\Vert _{L_{\infty }(\Gamma )}\Vert \mathbf{p}%
_{n}\Vert _{L_{2}(\Gamma )}^{2}\leqslant C\Vert {a}\Vert _{W_{p}^{1-\frac{1}{%
p}}(\Gamma )}||\mathbf{p}_{n}||_{L_{2}(\Omega )}||\nabla \mathbf{p}%
_{n}||_{L_{2}(\Omega )} \\
& \leqslant C\nu ^{-1}\Vert a\Vert {}_{W_{p}^{1-\frac{1}{p}}(\Gamma )}^{2}||%
\mathbf{p}_{n}||_{L_{2}(\Omega )}^{2}+\frac{\nu }{2}||\mathbf{p}%
_{n}||_{V}^{2}.
\end{align*}%
Also we have%
\begin{eqnarray*}
J_{3} &=&2|\left( \nabla _{\mathbf{y}}\mathbf{G}(t,\mathbf{y})^{T}\mathbf{q}%
_{n},\mathbf{p}_{n}\right) |\leqslant C||\nabla _{\mathbf{y}}\mathbf{G}(t,%
\mathbf{y})^{T}\mathbf{q}_{n}||_{2}||\mathbf{p}_{n}||_{2} \\
&\leqslant &\frac{1}{2}||\mathbf{q}_{n}||_{2}^{2}+C||\mathbf{p}_{n}||_{2}^{2}
\end{eqnarray*}%
by the assumption \eqref{cG}. Therefore, gathering the deduced estimates for
the terms $J_{1},$ $J_{2},$ $J_{2},$\ we obtain the existence of a positive
constant $\widetilde{C}_{1},$ such that 
\begin{equation}
J_{1}+J_{2}\ +J_{3}\leqslant 2r_{1}(t)||\mathbf{p}_{n}\Vert _{2}^{2}+\Vert 
\mathbf{U}\Vert _{2}^{2}+\frac{1}{2}||\mathbf{q}_{n}||_{2}^{2}+\nu ||\mathbf{%
p}_{n}||_{V}^{2}  \label{J_12}
\end{equation}%
with 
\begin{equation*}
r_{1}(t)=\widetilde{C}_{1}\max \left( \nu ^{-1},1\right) \left( 1+||(a,b)||_{%
\mathcal{H}_{p}(\Gamma )}^{2}+\Vert \mathbf{u}\Vert _{V}^{2}\right) \in
L_{1}(0,T),\quad P\text{-a.e. in }\Omega ,
\end{equation*}%
depending only on the data \eqref{eq00sec12} of our problem by \eqref{uny}.

Let us consider the function%
\begin{equation*}
\beta _{1}(t)=e^{\int_{0}^{t}r_{1}(s)\,ds}
\end{equation*}%
and apply the It\^{o} formula $\ $%
\begin{equation*}
\ d\left( \beta _{1}^{2}||\mathbf{p}_{n}\Vert _{2}^{2}\right) =\beta
_{1}^{2}d\left( ||\mathbf{p}_{n}\Vert _{2}^{2}\right) +2r_{1}\beta _{1}^{2}||%
\mathbf{p}_{n}\Vert _{2}^{2}\,dt.
\end{equation*}%
Using the equality \eqref{eqeq22}, integrating over the time interval $(t,T)$
and using the estimate \eqref{J_12}, we deduce the estimate 
\begin{eqnarray}
\beta _{1}^{2}(t)\left\Vert \mathbf{p}_{n}(t)\right\Vert _{2}^{2}
&+&\int_{t}^{T}\beta _{1}^{2}(\nu \left\Vert \mathbf{p}_{n}\right\Vert
_{V}^{2}+\frac{1}{2}\Vert \mathbf{q}_{n}\Vert _{2}^{2})\,ds  \notag
\label{eq_imp} \\
&\leqslant &\int_{t}^{T}\beta _{1}^{2}\Vert \mathbf{U}\Vert
_{2}^{2}\,ds-2\int_{t}^{T}\beta _{1}^{2}\left( \mathbf{q}_{n},\mathbf{p}%
_{n}\right) \,d{\mathcal{W}}_{s}.
\end{eqnarray}
Let us notice that the assumptions \eqref{cnd_1} imply 
\begin{equation*}
A_{\ast }\geqslant 16\widetilde{C}_{1}\max \left( \nu ^{-1},1\right) ,
\end{equation*}%
which gives $\mathbb{E}\left[ \beta _{1}^{8}(T)\right] <\infty .$ Therefore,
taking into accounting \eqref{pq_int}, the H\"{o}lder inequality yields 
\begin{eqnarray*}
\mathbb{E}\int_{0}^{T}\left( \beta _{1}^{2}\left( \mathbf{q}_{n},\mathbf{p}%
_{n}\right) \right) ^{2}\,ds &\leqslant &C\left( \mathbb{E}\int_{0}^{T}\beta
_{1}^{8}\,ds\right) ^{\frac{1}{2}}\left( \mathbb{E}\int_{0}^{T}\Vert \mathbf{%
q}_{n}\Vert _{2}^{8}\,ds\right) ^{\frac{1}{4}} \\
&\times &\left( \mathbb{E}\int_{0}^{T}\Vert \mathbf{p}_{n}\Vert
_{2}^{8}\,ds\right) ^{\frac{1}{4}}<C_{n}<\infty ,
\end{eqnarray*}%
which ensures that the stochastic process $M_{t}=\int_{t}^{T}\beta
_{1}^{2}\left( \mathbf{q}_{n},\mathbf{p}_{n}\right) \,d{\mathcal{W}}_{s}$ is
a square integrable martingale. Hence, taking the expectation in %
\eqref{eq_imp}, we infer that%
\begin{eqnarray*}
\mathbb{E}\left( \beta _{1}^{2}(t)\left\Vert \mathbf{p}_{n}(t)\right\Vert
_{2}^{2}\right) &+&\mathbb{E}\int_{t}^{T}\beta _{1}^{2}(\nu \left\Vert 
\mathbf{p}_{n}\right\Vert _{V}^{2}+\frac{1}{2}\Vert \mathbf{q}_{n}\Vert
_{2}^{2})\,ds \\
&\leqslant &\mathbb{E}\int_{t}^{T}\beta _{1}^{2}\Vert \mathbf{U}\Vert
_{2}^{2}\,ds\qquad \text{for }t\in (0,T).
\end{eqnarray*}%
Since $\forall t\in \lbrack 0,T]$, $\beta _{1}(t)$ $\geqslant 1,$ this
deduced estimate implies 
\begin{eqnarray*}
\sup_{t\in \lbrack 0,T]}\mathbb{E}\left( \left\Vert \mathbf{p}%
_{n}(t)\right\Vert _{2}^{2}\right) &+&\mathbb{E}\int_{0}^{T}(\nu \left\Vert 
\mathbf{p}_{n}\right\Vert _{V}^{2}+\frac{1}{2}\Vert \mathbf{q}_{n}\Vert
_{2}^{2})\,ds \\
&\leqslant &\mathbb{E}\int_{0}^{T}\beta _{1}^{2}\Vert \mathbf{U}\Vert
_{2}^{2}\,ds\leqslant C\left( \mathbb{E}\left[ \beta _{1}^{4}(T)\right]
\right) ^{\frac{1}{2}}\left( \mathbb{E}\int_{0}^{T}\Vert U\Vert
_{2}^{4}\right) ^{\frac{1}{2}},
\end{eqnarray*}%
then the first estimate of \eqref{qe1} holds.

\hfill $\hfill \hfill $

\textit{Step 3. Improvement of the uniform estimates. }Let us show the
second a priori estimate of \eqref{qe1}\ for $\left( \mathbf{p}_{n},\mathbf{q%
}_{n}\right) $ with a better regularity.\ Let us consider Galerkin's
approximation $\mathbf{p}_{n}$ that verifies (\ref{pn}). For each $n\in 
\mathbb{N},$ the function $\mathbf{p}_{n}(t,\cdot )\in H^{2}({\mathcal{O}}%
)\cap V_{n}$ a.e. in $\Omega ,$ being a linear combination of $\{\mathbf{h}%
_{k}\}\subset H^{2}({\mathcal{O}})\cap V$. The eigenfunctions $\left\{ 
\mathbf{h}_{k}\right\} _{k=1}^{\infty }$ of problem (\ref{y4}) fulfil the
homogeneous Navier-slip boundary condition in the system (\ref{y4}),
therefore the integration by parts of (\ref{pn}) yields the equality 
\begin{equation}
\begin{array}{l}
-d\left( \mathbf{p}_{n}(t),\mathbf{h}_{k}\right) =\left\{ \left( \left[ \nu
\Delta \mathbf{p}_{n}+2D(\mathbf{p}_{n})\mathbf{y}+\mathbf{U}\right] ,%
\mathbf{h}_{k}\right) -\int_{\Gamma }a(\mathbf{p}_{n}\cdot \bm{\tau })(%
\mathbf{h}_{k}\cdot \bm{\tau })\,d\mathbf{\gamma }\right. \\ 
\\ 
\qquad \mathbf{\qquad \qquad \quad ~}\left. +\left( \nabla _{\mathbf{y}}%
\mathbf{G}(t,\mathbf{y})^{T}\mathbf{q}_{n},\mathbf{h}_{k}\right) \right\}
\,dt-\left( \mathbf{q}_{n},\mathbf{h}_{k}\right) \,d{\mathcal{W}}_{t}%
\end{array}
\label{help}
\end{equation}%
for $k=1,2,\dots ,n.$ \ 

Let us introduce the Helmholtz projector $\mathbb{P}_{n}:L_{2}({\mathcal{O}}%
)\longrightarrow V_{n}$ of $V$ and define the function 
\begin{equation*}
A\mathbf{h}_{k}=\mathbb{P}_{n}\left( -\triangle \mathbf{h}_{k}\right)
=-\triangle \mathbf{h}_{k}+\nabla \pi _{k}\in V_{n}\text{\qquad with\qquad }%
\pi _{k}\in H^{1}({\mathcal{O}}),
\end{equation*}%
defined in (\ref{y4}). Multiplying (\ref{help}) by the eigenvalue $\lambda
_{k}$ of the problem (\ref{y4}), we have%
\begin{equation*}
\begin{array}{l}
-d\left( \mathbf{p}_{n}(t),\mathbf{h}_{k}\right) _{V}=\lambda _{k}\left\{
\left( \left[ \nu \Delta \mathbf{p}_{n}+2D(\mathbf{p}_{n})\mathbf{y}+\mathbf{%
U}\right] ,\mathbf{h}_{k}\right) -\int_{\Gamma }\left( a\mathbf{p}%
_{n}\right) \cdot \mathbf{h}_{k}\,d\mathbf{\gamma }\right. \\ 
\\ 
\qquad \quad \mathbf{\qquad \qquad \quad ~}\left. +\left( \nabla _{\mathbf{y}%
}\mathbf{G}(t,\mathbf{y})^{T}\mathbf{q}_{n},\mathbf{h}_{k}\right) \right\}
\,dt-\left( \mathbf{q}_{n},\mathbf{h}_{k}\right) _{V}\,d{\mathcal{W}}_{t},%
\end{array}%
\end{equation*}%
accounting that $\mathbf{p}_{n},$ $\mathbf{h}_{k}$\ are tangent to the
boundary $\Gamma $ and 
\begin{equation*}
\left( \mathbf{p}_{n},\lambda _{k}\mathbf{h}_{k}\right) =\left( \mathbf{p}%
_{n},\mathbf{h}_{k}\right) _{V}=\left( A\mathbf{p}_{n},\mathbf{h}_{k}\right)
,\text{\qquad }\left( \mathbf{q}_{n},\lambda _{k}\mathbf{h}_{k}\right)
=\left( \mathbf{q}_{n},\mathbf{h}_{k}\right) _{V}.
\end{equation*}%
The It\^{o} formula gives 
\begin{align*}
-d\left( \left( \mathbf{p}_{n},\mathbf{h}_{k}\right) _{V}^{2}\right) &
=2\lambda _{k}\left( A\mathbf{p}_{n}(t),\mathbf{h}_{k}\right) \left\{ \left( %
\left[ -\nu A\mathbf{p}_{n}+2D(\mathbf{p}_{n})\mathbf{y}+\mathbf{U}\right] ,%
\mathbf{h}_{k}\right) \right. \\
& \\
& \quad \left. -\int_{\Gamma }\left( a\mathbf{p}_{n}\right) \cdot \mathbf{h}%
_{k}\,d\mathbf{\gamma }+\left( \nabla _{\mathbf{y}}\mathbf{G}(t,\mathbf{y}%
)^{T}\mathbf{q}_{n},\mathbf{h}_{k}\right) \right\} \,dt \\
& \\
& \quad -2\lambda _{k}\left( A\mathbf{p}_{n}(t),\mathbf{h}_{k}\right) \left( 
\mathbf{q}_{n},\mathbf{h}_{k}\right) \,d{\mathcal{W}}_{t}-\left( \mathbf{q}%
_{n},\mathbf{h}_{k}\right) _{V}^{2}\,dt.
\end{align*}%
Dividing this equality by $\lambda _{k}$ and summing over $k=1,\dots ,n,$ we
obtain%
\begin{align}
-d\left[ \left\Vert \mathbf{p}_{n}\right\Vert _{V}^{2}\right] & +\left( 2\nu
||A\mathbf{p}_{n}||_{2}^{2}+||\mathbf{q}_{n}\Vert _{V}^{2}\right)
\,dt=2\left\{ \left( 2D(\mathbf{p}_{n})\mathbf{y},A\mathbf{p}_{n}\right)
+\left( \mathbf{U,}A\mathbf{p}_{n}\right) \right.  \notag \\
&  \notag \\
& \quad -\left. \int_{\Gamma }(a\mathbf{p}_{n})\cdot A\mathbf{p}_{n}\,d%
\mathbf{\gamma }+\left( \nabla _{\mathbf{y}}\mathbf{G}(t,\mathbf{y})^{T}%
\mathbf{q}_{n},A\mathbf{p}_{n}\right) \right\} dt  \notag \\
&  \notag \\
& =\left[ I_{1}+I_{2}+I_{3}+I_{4}\right] dt-2\left( \mathbf{q}_{n},\mathbf{p}%
_{n}\right) _{V}\,d{\mathcal{W}}_{t}.  \label{a2}
\end{align}%
Let us estimate the terms $I_{1},$ $I_{2},$ $I_{3}$\ and $I_{4}$. \ We have%
\begin{equation*}
I_{1}\leqslant C\Vert \mathbf{y}\Vert _{6}\Vert D\mathbf{p}_{n}\Vert
_{3}\,\Vert A\mathbf{p}_{n}\Vert _{2}.
\end{equation*}%
Using the Gagliardo-Nirenberg-Sobolev inequality (\ref{LI}) with $q=6$ and
with $q=3,$ respectively, 
\begin{equation*}
\Vert \mathbf{y}\Vert _{6}\leqslant Cf(t)\equiv C(\Vert \mathbf{y}\Vert
_{2}^{1/3}\Vert \nabla \mathbf{y}\Vert _{2}^{2/3}+\Vert \mathbf{y}\Vert
_{2}),
\end{equation*}%
and 
\begin{equation*}
\Vert D\mathbf{p}_{n}\Vert _{3}\leqslant C\left( \Vert D\mathbf{p}_{n}\Vert
_{2}^{2/3}\Vert \nabla \left( D\mathbf{p}_{n}\right) \Vert _{2}^{1/3}+\Vert D%
\mathbf{p}_{n}\Vert _{2}\right) ,
\end{equation*}%
we get%
\begin{eqnarray*}
I_{1} &\leqslant &Cf(t)\left( \Vert \mathbf{p}_{n}\Vert _{V}^{2/3}\Vert A%
\mathbf{p}_{n}\Vert _{2}^{1/3}+\Vert \mathbf{p}_{n}\Vert _{V}\right) \,\Vert
A\mathbf{p}_{n}\Vert _{2} \\
&\leqslant &2C\max (\nu ^{-2},1)\left( f^{3}(t)+f^{2}(t)\right) \Vert 
\mathbf{p}_{n}\Vert _{V}^{2}+\frac{\nu }{3}\Vert A\mathbf{p}_{n}\Vert
_{2}^{2}.
\end{eqnarray*}%
Applying Young's inequality (\ref{yi}), we can show%
\begin{eqnarray*}
f^{3}(t)+f^{2}(t) &\leqslant &C(\Vert \mathbf{y}\Vert _{2}\Vert \nabla 
\mathbf{y}\Vert _{2}^{2}+\Vert \mathbf{y}\Vert _{2}^{3}+\Vert \nabla \mathbf{%
y}\Vert _{2}^{2}+1) \\
&\leqslant &C(\Vert \mathbf{u}\Vert _{2}^{2}\Vert \mathbf{u}\Vert
_{V}^{2}+\Vert \mathbf{u}\Vert _{2}^{4}+||(a,b)||_{\mathcal{H}_{p}(\Gamma
)}^{4}+1),
\end{eqnarray*}%
therefore there exists a positive constant $C,$ such that 
\begin{equation*}
I_{1}\leqslant 2h(t)\Vert \mathbf{p}_{n}\Vert _{V}^{2}+\frac{\nu }{3}\Vert A%
\mathbf{p}_{n}\Vert _{2}^{2}
\end{equation*}%
with $\ $%
\begin{equation*}
h(t)=C\max (\nu ^{-2},1)(1+||(a,b)||_{\mathcal{H}_{p}(\Gamma )}^{4}+\Vert 
\mathbf{u}\Vert _{2}^{2}\Vert \mathbf{u}\Vert _{V}^{2}+\Vert \mathbf{u}\Vert
_{2}^{4})\in L_{1}(0,T),\ \ P-\text{a.e. in }\Omega ,
\end{equation*}%
by the estimates (\ref{uny}).

Also we have
\begin{equation*}
I_{2}=\int_{{\mathcal{O}}}\mathbf{U\cdot }A\mathbf{p}_{n}d\mathbf{x}%
\leqslant \Vert \mathbf{U}\Vert _{2}\Vert A\mathbf{p}_{n}\Vert _{2}\leqslant
C\nu ^{-1}\Vert \mathbf{U}\Vert _{2}^{2}+\frac{\nu }{3}\Vert A\mathbf{p}%
_{n}\Vert _{2}^{2}
\end{equation*}%
and 
\begin{eqnarray*}
I_{3} &=&-\int_{\Gamma }(a\mathbf{p}_{n})\cdot A\mathbf{p}_{n}\,d\mathbf{%
\gamma }\leqslant \Vert a\Vert _{L_{\infty }(\Gamma )}\Vert \mathbf{p}%
_{n}\Vert _{H^{1/2}({\mathcal{O}})}\Vert A\mathbf{p}_{n}\Vert _{H^{-1/2}({%
\mathcal{O}})} \\
&\leqslant &C\Vert a\Vert _{W_{p}^{1-\frac{1}{p}}(\Gamma )}\Vert \mathbf{p}%
_{n}\Vert _{H^{1}}\Vert A\mathbf{p}_{n}\Vert _{2}+\frac{\nu }{4}\Vert A%
\mathbf{p}_{n}\Vert _{2}^{2} \\
&\leqslant &2C\nu ^{-1}\Vert a\Vert _{W_{p}^{1-\frac{1}{p}}(\Gamma
)}^{2}\Vert \mathbf{p}_{n}\Vert _{V}^{2}+\frac{\nu }{3}\Vert A\mathbf{p}%
_{n}\Vert _{2}^{2}
\end{eqnarray*}%
by Theorem 1.5, p. 8 of \cite{gir}.

The term $I_{4}$ is estimated as%
\begin{eqnarray*}
I_{4} &=&2|\left( \nabla _{\mathbf{y}}\mathbf{G}(t,\mathbf{y})^{T}\mathbf{q}%
_{n},\mathbf{p}_{n}\right) _{V}|\leqslant C||\nabla _{\mathbf{y}}\mathbf{G}%
(t,\mathbf{y})^{T}\mathbf{q}_{n}||_{V}||\mathbf{p}_{n}||_{V} \\
&\leqslant &\frac{1}{2}||\mathbf{q}_{n}||_{V}^{2}+2C||\mathbf{p}%
_{n}||_{V}^{2}
\end{eqnarray*}%
by \eqref{cG}.

Therefore, substituting the above deduced estimates for the terms $I_{1},$ $%
I_{2},$ $I_{3},$\ $I_{4}$\ in \eqref{a2} we obtain that there exists a
positive constant $\widetilde{C}_{2},$ such that%
\begin{align}
-d\left( \left\Vert \mathbf{p}_{n}\right\Vert _{V}^{2}\right) & +\left( \nu
||A\mathbf{p}_{n}||_{2}^{2}+\frac{1}{2}||\mathbf{q}_{n}\Vert _{V}^{2}\right)
dt  \notag \\
& \leqslant \left( 2r_{2}(t)\Vert \mathbf{p}_{n}\Vert _{V}^{2}+\Vert \mathbf{%
U}\Vert _{2}^{2}\right) dt-2\left( \mathbf{q}_{n},\mathbf{p}_{n}\right)
_{V}\,d{\mathcal{W}}_{t}  \label{con22}
\end{align}%
with 
\begin{equation*}
r_{2}(t)=\widetilde{C}_{2}\max (\nu ^{-2},1)(1+||(a,b)||_{\mathcal{H}%
_{p}(\Gamma )}^{4}+\Vert \mathbf{u}\Vert _{2}^{2}\Vert \mathbf{u}\Vert
_{V}^{2}+\Vert \mathbf{u}\Vert _{2}^{4})\in L_{1}(0,T),
\end{equation*}%
$P-$a.e. in $\Omega ,$ by (\ref{uny}).

Introducing the function 
\begin{equation*}
\beta _{2}(t)=e^{\int_{0}^{t}r_{2}(s)ds}
\end{equation*}%
and applying the It\^{o} formula $\ $%
\begin{equation*}
\ d\left( \beta _{2}^{2}||\mathbf{p}_{n}\Vert _{V}^{2}\right) =\beta
_{2}^{2}d\left( ||\mathbf{p}_{n}\Vert _{V}^{2}\right) +2r_{2}\beta _{2}^{2}||%
\mathbf{p}_{n}\Vert _{V}^{2}\,dt
\end{equation*}%
to (\ref{con22}), we obtain that 
\begin{eqnarray}
\beta _{2}^{2}(t)\left\Vert \mathbf{p}_{n}(t)\right\Vert _{V}^{2}
&+&\int_{t}^{T}\beta _{2}^{2}\left( \nu ||A\mathbf{p}_{n}||_{2}^{2}+\frac{1}{%
2}||\mathbf{q}_{n}\Vert _{V}^{2}\right) \,ds  \notag \\
&\leqslant &\int_{t}^{T}\beta _{2}^{2}\Vert \mathbf{U}\Vert
_{2}^{2}\,ds-2\int_{t}^{T}\beta _{2}^{2}\left( \mathbf{q}_{n},\mathbf{p}%
_{n}\right) _{V}\,d{\mathcal{W}}_{s}  \label{0qe}
\end{eqnarray}%
by \eqref{con22}.

Let us notice that the assumptions \eqref{cnd_1} ensure that 
\begin{equation*}
\mathrm{min}(B_{\ast },\frac{B_{\ast }}{\widehat{C}})\geqslant 24\widetilde{C%
}_{2}\max (\nu ^{-2},1) ,
\end{equation*}%
which imply $\mathbb{E}\left[ \beta _{2}^{8}(T)\right] <\infty .$ Then the
regularity property \eqref{pq_int} and the H\"{o}lder inequality give 
\begin{eqnarray*}
\mathbb{E}\int_{0}^{T}\left( \beta _{2}^{2}\left( \mathbf{p}_{n},\mathbf{q}%
_{n}\right) \right) ^{2}\,ds &\leqslant &C\left( \mathbb{E}\int_{0}^{T}\beta
_{2}^{8}\,ds\right) ^{\frac{1}{2}}\left( \mathbb{E}\int_{0}^{T}\Vert \mathbf{%
q}_{n}\Vert _{2}^{8}\,ds\right) ^{\frac{1}{4}} \\
&\times &\left( \mathbb{E}\int_{0}^{T}\Vert \mathbf{p}_{n}\Vert
_{2}^{8}\,ds\right) ^{\frac{1}{4}}<C_{n}<\infty ,
\end{eqnarray*}%
then the stochastic process $M_{t}=\int_{t}^{T}\beta _{2}^{2}\left( \mathbf{q%
}_{n},\mathbf{p}_{n}\right) \,d{\mathcal{W}}_{s}$ is a square integrable
martingale.

Thus, taking the expectation of \eqref{0qe}, we derive 
\begin{eqnarray*}
\mathbb{E}\left( \beta _{2}^{2}(t)\left\Vert \mathbf{p}_{n}(t)\right\Vert
_{V}^{2}\right) &+&\mathbb{E}\int_{t}^{T}\beta _{2}^{2}(\nu \left\Vert A%
\mathbf{p}_{n}\right\Vert _{2}^{2}+\frac{1}{2}\Vert \mathbf{q}_{n}\Vert
_{V}^{2})\,ds \\
&\leqslant &\mathbb{E}\int_{t}^{T}\beta _{2}^{2}\Vert \mathbf{U}\Vert
_{2}^{2}\,ds \\
&\leqslant &C\left( \mathbb{E}\left[ \beta _{1}^{4}(T)\right] \right) ^{%
\frac{1}{2}}\left( \mathbb{E}\int_{0}^{T}\Vert \mathbf{U}\Vert
_{2}^{4}\right) ^{\frac{1}{2}},\qquad \text{for }t\in (0,T),
\end{eqnarray*}%
which implies 
\begin{eqnarray*}
\sup_{t\in \lbrack 0,T]}\mathbb{E}\left( \left\Vert \mathbf{p}%
_{n}(t)\right\Vert _{V}^{2}\right) &+&\mathbb{E}\int_{0}^{T}(\nu \left\Vert A%
\mathbf{p}_{n}\right\Vert _{2}^{2}+\frac{1}{2}\Vert \mathbf{q}_{n}\Vert
_{V}^{2})\,ds \\
&\leqslant &C\left( \mathbb{E}\int_{0}^{T}\Vert \mathbf{U}\Vert
_{2}^{4}\right) ^{\frac{1}{2}},
\end{eqnarray*}%
since $\beta _{2}(t)$ $\geqslant 1,$ $\forall t\in \lbrack 0,T]$.


\textit{Step 4. Passage to the limit. } The following inequality 
\begin{equation}
\Vert \mathbf{p}_{n}\Vert _{H^{2}}\leqslant C(\Vert A\mathbf{p}_{n}\Vert
_{2}+\Vert a\mathbf{p}_{n}\Vert _{H^{1/2}(\Gamma )})  \label{ap}
\end{equation}%
holds due to the regular properties of the Stokes operator $A$ (see Theorem
9 of \cite{amr2} and Theorem 2 of \cite{S73}).

Therefore the estimates \eqref{qe1} imply that there exists a suitable
subsequence of $\left\{ (\mathbf{p}_{n},\mathbf{q}_{n})\right\} ,$ such that 
\begin{eqnarray*}
\mathbf{p}_{n} &\rightharpoonup &\mathbf{p}\qquad \mbox{*- weakly in
}\ \ L_{\infty }(0,T;L_{2}(\Omega ;H)), \\
&&\quad \quad \qquad \mbox{weakly in
}\ L_{2}(\Omega \times (0,T);V),\qquad \\
\mathbf{q}_{n} &\rightharpoonup &\mathbf{q}\qquad \mbox{weakly in
}\ L_{2}(\Omega \times (0,T);H)
\end{eqnarray*}%
and 
\begin{eqnarray*}
\mathbf{p}_{n} &\rightharpoonup &\mathbf{p}\qquad \mbox{*- weakly in
}\ L_{\infty }(0,T;L_{2}(\Omega ;V)), \\
&&\quad \quad \quad \quad \mbox{ weakly in
}\ L_{2}(\Omega \times (0,T);H^{2}({\mathcal{O}})), \\
A\mathbf{p}_{n} &\rightharpoonup &A\mathbf{p}\qquad \mbox{weakly in
}\ L_{2}(\Omega \times {\mathcal{O}}_{T}),\qquad \\
\mathbf{q}_{n} &\rightharpoonup &\mathbf{q}\qquad \mbox{weakly in
}\ L_{2}(\Omega \times (0,T);V).
\end{eqnarray*}

Taking the limit transition in the system (\ref{pn}), written in the
distributional sense, we derive that the limit pair $(\mathbf{p},\mathbf{q})$
is the solution of\ the following integral system%
\begin{align*}
\left( \mathbf{p}(t),\boldsymbol{\phi }\right) & +\int_{t}^{T}\left\{
-\left( 2D(\mathbf{p})\mathbf{y},\boldsymbol{\phi }\right) +\nu \left( 
\mathbf{p},\,\boldsymbol{\phi }\right) _{V}-\left( \mathbf{U},\boldsymbol{%
\phi }\right) +\int_{\Gamma }{a}(\mathbf{p}\cdot \bm{\tau })(\boldsymbol{%
\phi }\cdot \bm{\tau })\,d\mathbf{\gamma }\right\} \,ds \\
& =\int_{t}^{T}\left( \nabla _{\mathbf{y}}\mathbf{G}(t,\mathbf{y})^{T}%
\mathbf{q},\boldsymbol{\phi }\right) \,ds-\int_{t}^{T}\left( \boldsymbol{q},%
\boldsymbol{\phi }\right) \,d{\mathcal{W}}_{s},\mathbf{\qquad }\text{a.e. in 
}\Omega ,
\end{align*}%
which is valid for all $t\in \lbrack 0,T]$, $\boldsymbol{\phi }\in V.$

By the result given on the page 208 of \cite{tem}, we deduce the existence
of the pressure $\pi \in H^{-1}(0,T;L_{2}({\mathcal{O}})),$ a.e. in $\Omega
. $ Moreover, reasoning as in Proposition 1.2, p. 182 of \cite{tem} and
using the inequality%
\begin{equation}
\Vert \mathbf{p}\Vert _{H^{2}}+\Vert \nabla \pi \Vert _{2}\leqslant C\left(
\Vert A\mathbf{p}\Vert _{2}+\Vert a\mathbf{p}\Vert _{H^{1/2}(\Gamma
)}\right) ,  \label{n}
\end{equation}%
by Theorem 9 of \cite{amr2}, we derive that $\pi \in L_{2}(0,T;H^{1}({%
\mathcal{O}})).$

In addition, it follows that the triple $(\mathbf{p},\mathbf{q},\pi )$
satisfies the estimates (\ref{qe1}) by the lower semi-continuity of integral
in the $L_{2}$-space.
\end{proof}

\section{Duality property}

\label{sec9}\setcounter{equation}{0}

In the next proposition we present the \textit{duality} property for the
solution $\mathbf{z}$ of the linearized equation (\ref{linearized}) and the
adjoint triplet $(\mathbf{p},\mathbf{q},\pi ),$ being the solution of (\ref%
{adjoint}).

\begin{proposition}
For $P$-a.e. in $\Omega $ the solution $\mathbf{z}$ of \ the system (\ref%
{linearized}) and the solution $(\mathbf{p},\pi )$ of the adjoint system (%
\ref{adjoint}) verify the following \textit{\ duality relation}%
\begin{eqnarray}
\int_{0}^{T}\left( \mathbf{z},\mathbf{U}\right) \,dt
&=&\int_{0}^{T}\int_{\Gamma }\left\{ \nu g(\mathbf{p}\cdot \bm{\tau })+f%
\left[ \pi -\left( \mathbf{p}\cdot \mathbf{y}\right) -2\nu \left( D(\mathbf{p%
})\mathbf{n}\right) \cdot \mathbf{n}\,\right] \right\} \ d\mathbf{\gamma }dt
\notag \\
&&+\int_{0}^{T}\left[ \left( \mathbf{q,z}\right) \,+\left( \nabla _{\mathbf{y%
}}\mathbf{G}(t,\mathbf{y})\mathbf{z},\mathbf{p}\right) \right] \,d{\mathcal{W%
}}_{t}.  \label{duality}
\end{eqnarray}
\end{proposition}

\begin{proof}
Taking $\boldsymbol{\varphi }=\mathbf{p}\in L_{2}(0,T;V)$ \ in Definition %
\ref{def6.1}, we obtain the equality%
\begin{eqnarray}
\left( d\mathbf{z},\mathbf{p}\right) &=&\left[ -\nu \left( \mathbf{z},%
\mathbf{p}\right) _{V}\,+\nu \int_{\Gamma }g(\mathbf{p}\cdot {\bm{\tau }})\,d%
\mathbf{\gamma }-(\left( \mathbf{z}\cdot \nabla )\mathbf{y},\mathbf{p}%
\right) -(\left( \mathbf{y}\cdot \nabla )\mathbf{z},\mathbf{p}\right) \,%
\right] \,dt  \notag \\
&&+\left( \nabla _{\mathbf{y}}\mathbf{G}(t,\mathbf{y})\mathbf{z},\mathbf{p}%
\right) \,d{\mathcal{W}}_{t}.  \label{dzp}
\end{eqnarray}%
Multiplying (\ref{adjoint})$\ $by $-\mathbf{z}$, we have%
\begin{align}
\left( d\mathbf{p},\mathbf{z}\right) & =\left[ -\left( \nu \Delta \mathbf{p,z%
}\right) -\left( 2D(\mathbf{p})\mathbf{y,z}\right) +\left( \nabla \pi 
\mathbf{,z}\right) -\left( \mathbf{U,z}\right) \right] \,dt  \notag \\
& \quad -\left( \nabla _{\mathbf{y}}\mathbf{G}(t,\mathbf{y})^{T}\mathbf{q,z}%
\right) \,dt+\left( \mathbf{q,z}\right) \,d{\mathcal{W}}_{t}\mathbf{\qquad }%
\text{a.e. in }\Omega \times (0,T).  \label{duality10}
\end{align}%
Using that the functions $\mathbf{y}$, $\mathbf{z}$ and $\mathbf{p}$ are
divergence free, the equality (\ref{integrate}) and the integration by
parts\ give the following three relations%
\begin{equation*}
-\nu \left( \Delta \mathbf{p},\mathbf{z}\right) =-\nu \int_{\Gamma }2\left(
D(\mathbf{p})\mathbf{n}\right) \cdot \mathbf{z}\,d\mathbf{\gamma }+2\nu
\left( D\mathbf{p},D\mathbf{z}\right) ,
\end{equation*}%
\begin{align*}
-\left( \mathbf{z},\left( 2D(\mathbf{p})\mathbf{y}\right) \right) & =\left( %
\left[ \left( \mathbf{y}\cdot \nabla \right) \mathbf{z}+\left( \mathbf{z}%
\cdot \nabla \right) \mathbf{y}\right] ,\mathbf{p}\right) \\
& \quad -\int_{\Gamma }\left( \left( \mathbf{y}\cdot \mathbf{n}\right)
\left( \mathbf{p}\cdot \mathbf{z}\right) +\left( \mathbf{z}\cdot \mathbf{n}%
\right) \left( \mathbf{p}\cdot \mathbf{y}\right) \right) \,d\mathbf{\gamma }
\end{align*}%
and 
\begin{equation*}
\left( \mathbf{z},\nabla \pi \right) =\int_{\Gamma }\left( \mathbf{z}\cdot 
\mathbf{n}\right) \pi \,d\mathbf{\gamma }.
\end{equation*}%
Using these equalities in the right hand side of (\ref{duality10}), we derive%
\begin{align}
\left( d\mathbf{p},\mathbf{z}\right) & =\left\{ \left( \left[ \left( \mathbf{%
y}\cdot \nabla \right) \mathbf{z}+\left( \mathbf{z}\cdot \nabla \right) 
\mathbf{y}\right] ,\mathbf{p}\right) +2\nu \,\left( D\mathbf{p},D\mathbf{z}%
\right) -\left( \mathbf{U,z}\right) \right\} \,dt  \notag \\
& \quad +\left( \int_{\Gamma }\left[ \left( \mathbf{z}\cdot \mathbf{n}%
\right) \pi -\left\{ \left( \mathbf{y}\cdot \mathbf{n}\right) \left( \mathbf{%
p}\cdot \mathbf{z}\right) +\left( \mathbf{z}\cdot \mathbf{n}\right) \left( 
\mathbf{p}\cdot \mathbf{y}\right) +\left( 2\nu D(\mathbf{p})\mathbf{n}%
\right) \cdot \mathbf{z}\right\} \right] \,d\mathbf{\gamma }\right) \,dt 
\notag \\
& \quad -\left( \nabla _{\mathbf{y}}\mathbf{G}(t,\mathbf{y})^{T}\mathbf{q,z}%
\right) \,dt+\left( \mathbf{q,z}\right) \,d{\mathcal{W}}_{t}.  \label{dpz}
\end{align}%
Hence, applying the It\^{o} formula 
\begin{equation*}
d\left( \mathbf{p},\mathbf{z}\right) =\left( d\mathbf{p},\mathbf{z}\right)
+\left( \mathbf{p},d\mathbf{z}\right) +\left( d\mathbf{p},d\mathbf{z}\right)
,
\end{equation*}%
we deduce%
\begin{align*}
d\left( \mathbf{z},\mathbf{p}\right) & =\left\{ -\left( \mathbf{U,z}\right)
+\nu \int_{\Gamma }[g(\mathbf{p}\cdot {\bm{\tau })}-\alpha (\mathbf{z}\cdot {%
\bm{\tau })}(\mathbf{p}\cdot {\bm{\tau })}+\left( \mathbf{z}\cdot \mathbf{n}%
\right) \pi \right. \, \\
& \quad \left. -\left\{ \left( \mathbf{y}\cdot \mathbf{n}\right) \left( 
\mathbf{p}\cdot \mathbf{z}\right) +\left( \mathbf{z}\cdot \mathbf{n}\right)
\left( \mathbf{p}\cdot \mathbf{y}\right) +\left( 2\nu D(\mathbf{p})\mathbf{n}%
\right) \cdot \mathbf{z}\right\} \,]\ d\mathbf{\gamma }\right\} \,dt \\
& \quad +\left[ \left( \mathbf{q,z}\right) \,+\left( \nabla _{\mathbf{y}}%
\mathbf{G}(t,\mathbf{y})\mathbf{z},\mathbf{p}\right) \right] \,d{\mathcal{W}}%
_{t}.
\end{align*}%
Due to the boundary conditions for $\mathbf{y}$, $\ \mathbf{z}$ and $\mathbf{%
p}$%
\begin{eqnarray*}
\mathbf{y}\cdot \mathbf{n} &=&a,\qquad \mathbf{z}\cdot \mathbf{n}=f,\qquad 
\mathbf{p}\cdot \mathbf{n}=0, \\
\left( 2\nu D(\mathbf{p})\mathbf{n}\right) \cdot \bm{\tau } &=&-\nu ({\frac{a%
}{\nu }}+\alpha )(\mathbf{p}\cdot \bm{\tau }),
\end{eqnarray*}%
we obtain%
\begin{align*}
d\left( \mathbf{z},\mathbf{p}\right) & =\left\{ -\left( \mathbf{U,z}\right)
+\int_{\Gamma }\nu g(\mathbf{p}\cdot {\bm{\tau })}-\nu \alpha (\mathbf{p}%
\cdot {\bm{\tau })}(\mathbf{z}\cdot {\bm{\tau })}+f\pi \right. \, \\
& \quad -\{a(\mathbf{p}\cdot {\bm{\tau })}(\mathbf{z}\cdot {\bm{\tau })}+f(%
\mathbf{p}\cdot \mathbf{y}{)}+\left( \left( 2\nu D(\mathbf{p})\mathbf{n}%
\right) \cdot \mathbf{n}\right) f \\
& \quad \left. -\nu ({\frac{a}{\nu }}+\alpha )(\mathbf{p}\cdot \bm{\tau })(%
\mathbf{z}\cdot {\bm{\tau })}\}\,d\mathbf{\gamma }\right\} \,dt \\
& \,\quad +\left[ \left( \mathbf{q,z}\right) \,+\left( \nabla _{\mathbf{y}}%
\mathbf{G}(t,\mathbf{y})\mathbf{z},\mathbf{p}\right) \right] \,d{\mathcal{W}}%
_{t} \\
& =\left\{ -\left( \mathbf{U,z}\right) +\int_{\Gamma }\left[ \nu g(\mathbf{p}%
\cdot {\bm{\tau })}+f\pi -\{f(\mathbf{p}\cdot \mathbf{y}{)}+\left( \left(
2\nu D(\mathbf{p})\mathbf{n}\right) \cdot \mathbf{n}\right) f\}\right] \,d%
\mathbf{\gamma }\right\} \\
& +\left[ \left( \mathbf{q,z}\right) \,+\left( \nabla _{\mathbf{y}}\mathbf{G}%
(t,\mathbf{y})\mathbf{z},\mathbf{p}\right) \right] \,d{\mathcal{W}}_{t}
\end{align*}%
by \eqref{dzp}, \eqref{dpz}. Integrating this equality over the time
interval $(0,T)$, we obtain the duality property (\ref{duality}).
\end{proof}

\bigskip $\hfill $

\section{Main result. First-order optimality { condition }}

\label{sec10}\setcounter{equation}{0}

Let us consider 
\begin{equation*}
\widehat{C}_{\max }=24\,\max \left( \widehat{C}_{1},\widehat{C}_{2},%
\widetilde{C}_{1},\widetilde{C}_{2}\right) \times \max \left( \nu
^{-2},1\right)
\end{equation*}%
with the constants $\widehat{C}_{1},\widehat{C}_{2},\widetilde{C}_{1},%
\widetilde{C}_{2}$ defined in Theorem \ref{Lips}, Proposition \ref%
{ex_uniq_lin} and Proposition \ref{ex_uniq_adj}. In what follows we will
assume that 
\begin{equation}  \label{CF}
\min \left( A_{\ast },B_{\ast },\frac{B_{\ast }}{\widehat{C}}\right)
\geqslant \widehat{C}_{\max },
\end{equation}
where $A_{\ast },$ $B_{\ast }$ and \ $\widehat{C}$ defined by (\ref{aaa})
and (\ref{CCC}).

Let us point that the assumptions (\ref{bound0}) and Theorem \ref{the_1},
Propositions \ref{pop}, \ref{ex_uniq_lin}, \ref{ex_uniq_adj} with the help
of { H\"{o}lder's } inequality imply that the stochastic processes $\mathbf{y,}$ $\mathbf{z}$
and $\mathbf{p,}$ $\mathbf{q,}$ $\pi ,$\ being the solutions for the systems
(\ref{NSy}), (\ref{linearized}) and (\ref{adjoint}),\ have the following
regularity 
\begin{eqnarray*}
\mathbf{y} &\in &L_{4}(\Omega ;C([0,T];L_{2}({\mathcal{O}})))\cap
L_{4}(\Omega \times (0,T);H^{1}({\mathcal{O}})), \\
\mathbf{z} &\in &L_{2}(\Omega ;C([0,T];L_{2}({\mathcal{O}})))\cap
L_{2}(\Omega \times (0,T);H^{1}({\mathcal{O}}))
\end{eqnarray*}%
and%
\begin{eqnarray*}
\mathbf{p} &\in &L_{2}(\Omega ;C([0,T];V)\cap L_{2}(0,T;H^{2}({\mathcal{O}}%
))),\qquad \mathbf{q}\in L_{2}(\Omega \times (0,T);V), \\
\pi &\in &L_{2}(0,T;H^{1}({\mathcal{O}})).
\end{eqnarray*}

Now, we demonstrate the main result of the article, establishing the
first-order necessary optimality condition for the problem $(\mathcal{P})$.

\begin{teo}
\label{main_1} Under the assumptions of Theorem \ref{main_existence}, let us
assume that $(\widehat{a},\widehat{b},\widehat{\mathbf{y}})$ is
 { an} optimal solution of problem $(\mathcal{P})$. Then $P$-a.e. in $\Omega $ there exists
a unique solution \ $(\widehat{\mathbf{p}},\widehat{\mathbf{q}},\widehat{\pi 
})$ of the adjoint system 
\begin{equation}
\left\{ 
\begin{array}{ll}
-d\widehat{\mathbf{p}}=\left[ \nu \Delta \widehat{\mathbf{p}}+2D(\widehat{%
\mathbf{p}})\widehat{\mathbf{y}}-\nabla \widehat{\pi }+\left( \widehat{%
\mathbf{y}}-\mathbf{y}_{d}\right) \right] \,dt & \, \\ 
&  \\ 
\,\quad \quad \quad +\nabla _{\mathbf{y}}\mathbf{G}(t,\widehat{\mathbf{y}}%
)^{T}\widehat{\mathbf{q}}\,dt-\widehat{\mathbf{q}}\,d{\mathcal{W}}_{t}%
\mathbf{,\qquad }\mathrm{div}\,\widehat{\mathbf{p}}=0 & \mbox{in}\ {\mathcal{%
O}}_{T},\vspace{2mm} \\ 
&  \\ 
\widehat{\mathbf{p}}\cdot \mathbf{n}=0,\qquad \left[ 2D(\widehat{\mathbf{p}})%
\mathbf{n}+(\alpha +{\frac{\widehat{a}}{\nu }})\widehat{\mathbf{p}}\right]
\cdot {\mathbf{\tau }}=0, & \mbox{on}\ \Gamma _{T},\vspace{2mm} \\ 
\widehat{\mathbf{p}}(T)=0, & \mbox{in}\ {\mathcal{O}},%
\end{array}%
\right.  \label{MR2}
\end{equation}%
verifying the optimality condition 
\begin{eqnarray}
\mathbb{E}\int_{\Gamma _{T}}\{(f &-&\widehat{a})[\widehat{\pi }-(\widehat{%
\mathbf{p}}\cdot \widehat{\mathbf{y}})-\left( 2\nu D(\widehat{\mathbf{p}})%
\mathbf{n}\right) \cdot \mathbf{n}]+\nu (g-\widehat{b})(\widehat{\mathbf{p}}%
\cdot \mathbf{\bm{\tau })}  \notag \\
&+&\lambda _{1}\widehat{a}\left( f-\widehat{a}\right) \,+\lambda _{2}%
\widehat{b}(g-\widehat{b})\,\}~d\mathbf{\gamma }dt\geqslant 0  \label{zvMR2}
\end{eqnarray}%
for all $(f,g)\in L_{4}(\Omega \times (0,T);\mathcal{H}_{p}(\Gamma ))$.
\end{teo}

\begin{proof}
Let $(\widehat{a},\widehat{b},\widehat{\mathbf{y}})$ be a solution of the
problem $(\mathcal{P}).$

According to Theorem \ref{the_1} and Proposition \ref{Gat}, for any $%
(a,b)\in L_{4}(\Omega \times (0,T);\mathcal{H}_{p}(\Gamma ))$ the
corresponding state system (\ref{NSy}) has a unique solution $\mathbf{y}$
and the control-to-state mapping 
\begin{equation*}
(a,b)\rightarrow \mathbf{y}
\end{equation*}%
is the G\^{a}teaux differentiable at $(\widehat{a},\widehat{b})$.

For $\varepsilon \in (0,1)$ and $(f,g)\in L_{4}(\Omega \times (0,T);\mathcal{%
H}_{p}(\Gamma ))$, let us set 
\begin{equation*}
a_{\varepsilon }=\widehat{a}+\varepsilon (f-\widehat{a}),\qquad
b_{\varepsilon }=\widehat{b}+\varepsilon (g-\widehat{b}).
\end{equation*}%
We consider the strong solution $\mathbf{y}_{\varepsilon }$ of (\ref{NSy})
with the boundary conditions $(a_{\varepsilon },b_{\varepsilon })$ (see
definition \ref{1def}). Since $(\widehat{a},\widehat{b},\widehat{\mathbf{y}}%
) $ is the optimal solution and $(a_{\varepsilon },b_{\varepsilon },\mathbf{y%
}_{\varepsilon })$ is admissible triple, using the convexity of the
functional $J$, we have 
\begin{equation*}
\lim_{\varepsilon \rightarrow 0^{+}}\frac{J(a_{\varepsilon },b_{\varepsilon
},\mathbf{y}_{\varepsilon })-J(\widehat{a},\widehat{b},\widehat{\mathbf{y}})%
}{\varepsilon }\geqslant 0,
\end{equation*}%
that, taking into account Proposition \ref{Gat}, gives the inequality 
\begin{equation}
\mathbb{E}\left[ \int_{{\mathcal{O}}_{T}}\widehat{\mathbf{z}}\cdot \left( 
\widehat{\mathbf{y}}-\mathbf{y}_{d}\right) \,d\mathbf{x}dt+\int_{\Gamma
_{T}}\left\{ \lambda _{1}\,\widehat{a}(f-\widehat{a})\,+\lambda _{2}\widehat{%
b}(g-\widehat{b})\right\} \,d\mathbf{\gamma }dt\right] \geqslant 0,
\label{zv}
\end{equation}%
where 
\begin{equation*}
\widehat{\mathbf{z}}=\lim_{\varepsilon \rightarrow 0^{+}}\frac{\widehat{%
\mathbf{y}}_{\varepsilon }-\widehat{\mathbf{y}}}{\varepsilon }
\end{equation*}%
is the unique strong solution of the system 
\begin{equation*}
\left\{ 
\begin{array}{l}
\begin{array}{l}
d\widehat{\mathbf{z}}=(\nu \Delta \widehat{\mathbf{z}}-\left( \widehat{%
\mathbf{z}}\mathbf{\cdot }\nabla \right) \widehat{\mathbf{y}}-\left( 
\widehat{\mathbf{y}}\mathbf{\cdot }\nabla \right) \widehat{\mathbf{z}}%
-\nabla \pi )\,dt+\nabla _{\mathbf{y}}\mathbf{G}(t,\widehat{\mathbf{y}})%
\widehat{\mathbf{z}}\,d{\mathcal{W}}_{t}\quad \text{in}\ {\mathcal{O}}_{T},%
\vspace{2mm} \\ 
\mathrm{div}\ \widehat{\mathbf{z}}=0,%
\end{array}
\\ 
\widehat{\mathbf{z}}\cdot \mathbf{n}=f-\widehat{a},\mathbf{\qquad }\left[ 2D(%
\widehat{\mathbf{z}})\,\mathbf{n}+\alpha \widehat{\mathbf{z}}\right] \cdot %
\bm{\tau }=g-\widehat{b}\mathbf{\qquad \qquad \qquad \qquad }\text{on}\
\Gamma _{T},\vspace{2mm} \\ 
\widehat{\mathbf{z}}(0)=0\mathbf{\qquad \qquad \qquad \qquad \qquad \qquad
\qquad \qquad \qquad \qquad }\text{in}\ {\mathcal{O}}%
\end{array}%
\right.
\end{equation*}%
in the sense of Definition \ref{def6.1}.

On the other hand, if we consider the system (\ref{adjoint}) with 
\begin{equation*}
\mathbf{U}=\widehat{\mathbf{y}}-\mathbf{y}_{d},\quad a=\widehat{a}\quad 
\text{and }\quad \mathbf{y}=\widehat{\mathbf{y}},
\end{equation*}
then Proposition \ref{ex_uniq_adj} implies that there exists the adjoint
state pair $(\widehat{\mathbf{p}},\widehat{\pi }),$ which verifies the
equation (\ref{MR2}), $P$-a.e. in $\Omega $. Since the duality property (\ref%
{duality}) is valid for 
\begin{equation*}
\mathbf{z}=\widehat{\mathbf{z}},\quad \mathbf{U}=\widehat{\mathbf{y}}-%
\mathbf{y}_{d}\quad \text{and }\quad (\mathbf{p},\pi )=(\widehat{\mathbf{p}},%
\widehat{\pi }),
\end{equation*}%
\ we deduce the equality 
\begin{align*}
\int_{{\mathcal{O}}_{T}}& \widehat{\mathbf{z}}\cdot \left( \widehat{\mathbf{y%
}}-\mathbf{y}_{d}\right) \,d\mathbf{x}dt \\
& =\int_{\Gamma _{T}}\left\{ \nu (g-\widehat{b})\left( \widehat{\mathbf{p}}%
\cdot \mathbf{\bm{\tau }}\right) +(f-\widehat{a})\left[ \pi -(\widehat{%
\mathbf{p}}\cdot \widehat{\mathbf{y}})-2\nu \left( \mathbf{n}\cdot D(%
\widehat{\mathbf{p}})\right) \cdot \mathbf{n}\right] \right\} \,d\mathbf{%
\gamma }dt \\
& \quad +\int_{0}^{T}\left[ \left( \widehat{\mathbf{q}}\mathbf{,\widehat{%
\mathbf{z}}}\right) \,+\left( \nabla _{\mathbf{y}}\mathbf{G}(t,\widehat{%
\mathbf{y}})\widehat{\mathbf{z}},\widehat{\mathbf{p}}\right) \right] \,d{%
\mathcal{W}}_{t}.
\end{align*}%
Substituting this equality into (\ref{zv}) and taking the expectation, we
obtain the optimality condition (\ref{zvMR2}).
\end{proof}

\begin{remark}
The condition (\ref{CF}) in Theorem \ref{main_1} can be interpreted as
follows:\ 

a) Under the presence of strong noise (corresponding to high values of $L$),
the fluid behaviour can be optimally controlled through the injection/suction
device modelled by the Navier-slip boundary conditions if the fluid is very
viscous.

b) Under very small random disturbances (corresponding to small enough
values of $L$), the condition (\ref{CF}) holds for high viscosity values $%
\nu $, then the fluid behavior can be optimally controlled.
\end{remark}

\textbf{Acknowledgement} {\footnotesize {\ The work of N.V. Chemetov was
supported by FAPESP (Funda\c{c}\~{a}o de Amparo \`{a} Pesquisa do Estado de S%
\~{a}o Paulo), project 2021/03758-8, "Mathematical problems in fluid
dynamics". }}

{\footnotesize Also the collaboration scientific work of F. Cipriano and
N.V. Chemetov was supported by FAPESP (Funda\c{c}\~{a}o de Amparo \`{a}
Pesquisa do Estado de S\~{a}o Paulo), through the Visiting Professor Project
2023/05271-4, "Problema de controlo otimo atrav\'{e}s da fronteira para equa%
\c{c}\~{o}es de Navier-Stokes estoc\'{a}sticas". }

{\footnotesize The work of F. Cipriano is funded by national funds through
the FCT - Funda\c{c}\~{a}o para a Ci\^{e}ncia e a Tecnologia, I.P., under
the scope of the projects UIDB/00297/2020 and UIDP/00297/2020 (Center for
Mathematics and Applications).}

\end{document}